\documentclass[a4paper,11pt,twoside]{amsart}
\pdfoutput=1 

 \usepackage[monochrome]{xcolor}

\input{degFG_01_mise_en_forme.def}
\input{degFG_02_notations.def}

\begin{document}
\title[Invariant subspaces for surface groups 
acting on $\typeA_2$-buildings]%
{Invariant subspaces for some surface groups 
acting on $\typeA_2$-Euclidean buildings.}

\author{Anne Parreau}

\address{
Université Grenoble I et CNRS, Institut Fourier\\ 
BP 74, 38402 St-Martin-d'Hères cedex, France.}

\email{Anne.Parreau@ujf-grenoble.fr}

\begin{abstract}

  This paper deals with
  %
  non-Archimedean representations of punctured surface groups in  $\PGL_3$, 
  associated actions on Euclidean buildings (of type $\typeA_2$), 
  and  degenerations of real convex projective structures on surfaces.
  %
  %
  The main result is that,
  under good conditions on Fock-Goncharov generalized shear parameters,
  non-Archimedean representations acting on the Euclidean building
  preserve a cocompact weakly convex subspace, 
  which is part flat surface and part tree.
  %
%
  %
%
  %
  %
  In particular the eigenvalue and length(s) spectra
  are given by an explicit finite {\em $\typeA_2$-complex}. 
  %
  %
  %
  %
  We use this result to describe 
  degenerations 
  of real convex projective structures on surfaces
  for an open cone of parameters.
  %
  The main tool is a geometric interpretation of Fock-Goncharov 
  parametrization in  $\typeA_2$-buildings.
\end{abstract}

\maketitle

\section*{Introduction}

One motivation for 
the study of actions of surface groups on
nondiscrete Euclidean $\typeA_2$-buildings 
is that,
in the same way that
degenerations of hyperbolic structures on surfaces
give rise to actions of the surface group on real trees 
(see Bestvina \cite{Bestvina88}, Paulin \cite{Paulin88}),
degenerations  of representations in $\SL_3(\RR)$
give rise to actions on
nondiscrete Euclidean $\typeA_2$-buildings
(see for example \cite{PauDg}, 
\cite{KlLe97}, 
\cite{ParComp}).
More specifically, 
in \cite{ParComp} we constructed 
a compactification of the space $\SPC(\Sf)$ 
of convex real projective structures 
on a closed surface $\Sf$, 
%
whose boundary points are marked length spectra 
of actions of $\Ga=\pi_1(\Sf)$ 
on nondiscrete Euclidean $\typeA_2$-buildings.
These actions come from representations of $\Ga$ in $\SL_3(\KK)$ 
for some ultrametric valued fields $\KK$.
 Degenerations of convex projective structures,
 or more generally of Hitchin representations,
 have been studied by numerous 
people, including
J.~Loftin \cite{Loftin07}, 
D.~Cooper, K.~Delp, D.~Long and M.~Thistle\-thwaite (forthcoming work),
D.~Alessandrini \cite{Alessandrini08}, 
I.~Le \cite{Le12},
T.~Zhang \cite{Zhang13},
B.~Collier, Q.~Li \cite{CoLi14}.
Given an action of a group $\Ga$ on a Euclidean building $\Es$, 
a natural question, in the spirit of minimal invariant subtrees for
actions on trees and convex cores for actions on hyperbolic space,
is whether it is possible to find a nice invariant convex subset 
$\Y \subset \Es$, for example cocompact or a minimal subbuilding... 
One of the motivations is that the length spectrum 
for instance would then be  recoverable from $\Y$ alone.
But convexity is a quite rigid
property in higher rank
(see for instance \cite{Quint05}, \cite{KlLe06}).
 We introduce here a more flexible 
notion of {\em weak convexity} for subsets $\Y$ of Euclidean buildings $\Es$,
that we call {\em $\Cc$-convexity},
 such that the length spectrum is still recoverable 
- in a more indirect way - from $\Y$.

In the case where 
$\Sf$ is a compact oriented surface with  nonempty boundary,
and $\KK$ any ultrametric valuated field,
for a large family of representations $\rho:\Ga\to \PGL_3(\KK)$,
we construct explicitly 
a simple, weakly convex, invariant subcomplex $\Y$ 
in the associated Euclidean building $\Es$,
on which $\Ga$ acts freely properly cocompactly. 
The subcomplex $\Y$ which is piecewise a flat surface or a tree.
We introduce also the notion of {\em $\typeA_2$-surface},  
and more generally of {\em $(\Aa,\W)$-complexes}, 
that is simplicial complexes
modelled on a finite reflection group $(\Aa,\W)$.
Natural examples are subcomplexes 
of Euclidean buildings with model flat $(\Aa,\W)$.
A {\em $\typeA_2$-structure} 
on a surface $\Sf$
    is a $(\Aa,\W)$-structure with singularities, 
for the finite reflection group $(\AA,\W)$ of type $\typeA_2$,
(corresponding to $\PGL_3$) (see section \ref{sss- tTsurf}).
Such structures are analogous to translation and
half-translation surfaces (and will be called {\em \tTsurfaces{}}), 
and are closely related to cubic holomorphic differentials on the
surface 
(for which we refer to
Labourie \cite{Lab07},
Loftin \cite{Loftin01},
Benoist-Hulin \cite{BeHu14},
Dumas-Wolf \cite{DuWo14}%
).
As a consequence of the previous result, 
we construct a family of explicit finite $\typeA_2$-complexes $\CEsgZ$,
homotopy equivalent to 
$\Sf$, 
parametrized by a $8\abs{\chi(\Sf)}$-dimensional real parameter $\gFGp$,
which encodes the absolute values of eigenvalues 
of the representations $\rho$ above ($\CEsgZ \simeq \Y/\rho(\Ga)$).
The main tool is the Fock-Goncharov parametrization of representations
$\rho:\Ga\to \PGL_3(\KK)$ (generalized shear coordinates).

We now state the definitions and results in more details.

 The model flat (of type $\typeA_2$)
is the $2$-dimensional real vector space 
$$\Aa=\{\vA=(\vA_1,\vA_2,\vA_3)\in\RR^3/\ \sum_\indN \vA_\indN=0\}$$
endowed
with the action of the Weyl group $\W=\Sym_3$ acting on $\Aa$ by
permutation of coordinates (finite reflection group).
The model Weyl chamber 
is the cone 
$$\Cc=\{\vA\in\Aa/\ \vA_1 > \vA_2 >\vA_3\}$$
 in $\Aa$.
Its closure $\Cb$ is a strict fundamental domain for the action of
 $\W$ on $\Aa$.
A vector $\vA\in\Aa$ is {\em singular} if it belongs to one of the
three {\em singular lines} $\vA_\indN=\vA_\indNp$. 
The two distinct types of singular directions (rays) in
$\Aa$, corresponding to the orbits under $\W$ of two rays 
$\vA_1 > \vA_2=\vA_3$  and $\vA_1=\vA_2> \vA_3$ 
bounding $\Cc$, which will respectively be called 
{\em  type  $\stypep$} and {\em  type $\stypeD$}.
In the figures (Figure \ref{fig-intro Aa} and the sequel), 
the type of singular directions will be represented by an arrow
$\triangleright$ indicating the induced orientation on singular lines
(towards the type  $\stypep$ extremity).
We will 
use as canonical coordinates on $\Aa$ the simple roots,  
i.e. the linear forms 
$\rac1(\vA) = \vA_1-\vA_2$ and 
$\rac2(\vA) = \vA_2-\vA_3$, 
hence we will identify $\vA\in\Aa$ with
$(\rac1(\vA), \rac2(\vA) )\in\RR\times\RR$
(see Figure \ref{fig-intro Aa}).
The $\W$-invariant Euclidean norm $\normeuc{\ }$ on $\Aa$ 
(unique up to rescaling) is normalized so that
the simple roots $\rac\indN$ measure the distance 
to the corresponding singular line $\rac\indN=0$.

\begin{figure}[h]
  \includegraphics[scale=0.8]{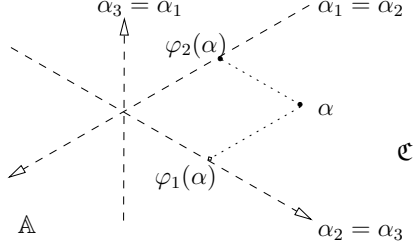} 

  \caption{Simple roots coordinates in the model flat $\Aa$.}
\label{fig-intro Aa}
\end{figure}

When $\Es$ is a (real) Euclidean building or a symmetric space
of type $\typeA_2$, i.e. with maximal flats isomorphic to $(\Aa,\W)$, 
the usual metric $\deuc:\Es\times\Es \to \RR_{\geq 0}$ 
(induced by the Euclidean norm $\normeuc{\ }$ on $\Aa$)
has a natural vector-valued
refinement,   
$$\Cd:\Es\times\Es \to \Cb$$
that we will call the {\em $\Cc$-distance}:
it is the canonical projection induced by the
 natural markings $\f:\Aa\to\Es$ of flats, whose
transition maps are in $\W$ up to translation.
The corresponding refinement of the usual (translation) length
({\em Euclidean length})
$$\elleuc(\g)=\{\deuc(\xE,\g\xE),\ \xE\in\Es\}$$
of an automorphism $\g$ of $\Es$
is the   {\em $\Cc$-length} $\ellC(\g)$ of $\g$.
It may be defined as the unique vector  
of minimal length in (the closure in $\Cb$
of) $\{\Cd(\xE,\g\xE),\ \xE\in\Es\}$, 
and we have 
$$\elleuc(\g)=\normeuc{\ellC(\g)}\;.$$
For $\g$ in $\SL_3(\KK)$ acting on its associated Euclidean building
 (for ultrametric $\KK$) or symmetric space (for $\KK=\RR$) it
corresponds to
$$\ellC(\g)=(\log\abs{\aK_\indN})_\indN$$ 
where the $\aK_\indN$ are the eigenvalues of  $\g$ (in nonincreasing order). 
The $\Cc$-length refines another notion of length 
of particular interest, 
the {\em Hilbert length}, which is 
the length of $\g$ for the Hilbert metric 
in the context of convex projective structures.
It may be defined by 
$$\ellH(\g)=\normH(\ellC(\g))$$
where $\normH$ is the {\em \hex-norm}  on $\Aa$ 
i.e. the $\W$-invariant norm
defined by $\normH(\vA)=\vA_1-\vA_3$ for $\vA$ in $\Cc$
(whose unit ball is the singular regular hexagon). 

We will here introduce the naturally associated notion of
{\em  $\Cc$-geodesics}, 
which are paths on which  the $\CC$-distance is additive.
Note that, unlike for the usual distance, 
$\Cc$-geodesics between two given points are not unique, 
and that usual 
geodesics are $\Cc$-geodesics, but the converse is not true.
The notion of weak convexity is then defined, by analogy with the
usual setting, as follows: 
we say that a subset $\Y \subset \Es$ is {\em $\Cc$-convex} if for any
two points $\xE,\yE$ in $\Y$, there exists a $\Cc$-geodesic from $\xE$ to
$\yE$ that is contained in $\Y$.

We now turn to the Fock-Goncharov parametrization 
of  representations  of the fundamental group
$\Ga$ of a compact oriented surface $\Sf$ with nonempty boundary.
More precisely,
following \cite{FoGoSPC}, 
we explain how to associate,
to an ideal triangulation $\T$ and  $8\chi(\Sf)$ parameters in
$\KK$ (one per triangle and two per edge),  
a representation $\rho:\Ga\to \PGL_3(\KK)$. 
This construction is in fact  valid any field $\KK$.  
It is based on projective geometry, 
through the action of $\PGL_3(\KK)$ on the projective plane $\PP(\KK^3)$. 
Denote by $\Bir(\aK_1,\aK_2,\aK_3,\aK_4)$ the cross ratio on $\PP(\KK^2)$,
with the convention $\Bir(\infty,-1,0,\aK)=\aK$.
Let $\MaxFlags(\PP)$ be the space of flags   
in the projective plane $\PP=\PP(\KK^3)$, that is the space of pairs
$(\p,\D)$, 
where $\p$ is a point and $\D$ a line of $\PP$, with $\p\in\D$.
%
%
Denote by  $\Farey{\Sf}$ the {\em Farey set} of the surface,
which may be defined as the  set of boundary components 
of the universal cover $\Sft$ of $\Sf$
(see section \ref{ss- prelim surf and ideal triang}), with the induced
cyclic order.
Let $\T$ be an ideal triangulation of $\Sf$. 
Denote by $\Tt$  the lift of $\T$ to the universal cover $\Sft$ of $\Sf$.
Shrinking boundary components of $\Sft$ to points, we may see $\Tt$ as
a triangulation of $\Sft$ with vertex set the Farey set $\Farey{\Sf}$.
Denote 
by   $\TrianglesT$ the set of  triangles of $\T$,
by $\orEdgesT$ the set of oriented edges of $\T$, 
which are finite sets 
of respective cardinality $2\abs{\chi(S)}$ and $6\abs{\chi(S)}$.
Fix a {\em FG-parameter} $\FGp=((\Z_\tau)_\tau,(\Ep_\e)_\e)$ 
in $(\KKol)^{\TrianglesT} \times (\KKo)^{\orEdgesT}$.
There exists then a unique (up to $\PGL(\KK^3)$ action)
associated {\em   flag map} 
$\FdevZ:\Farey{\Sf} \to \MaxFlags(\PP)$,
$\i \mapsto (\p_\i,\D_\i)$,
 equivariant with respect to a unique 
representation  $\rhoZ:\Ga\to \PGL(\KK^3)$, 
such that 
the flag map $\FdevZ$ sends  
each triangle $\taut=(\i,\j,\k)$  of $\Tt$ 
to a generic triple of flags of   triple ratio
$$\Bir(\D_\i, \p_\i \p_\j, \p_\i (\D_\j\cap \D_\k), \p_\i
\p_\k) = \Z_\tau$$
where $\tau$ is the triangle of $\T$ with lift $\taut$,
and for any two ajdacent triangles 
$(\i,\j,\k)$ and $(\k,\l,\i)$ of $\Tt$ 
with common edge $\et=(\k,\i)$ 
we have
$$\Bir
(\D_\i, \p_\i \p_\j, \p_\i \p_\k, \p_\i (\D_\k\cap \D_\l))
=\Ep_\e$$
where $\e$ is the oriented edge of $\T$ with lift $\et$,
and   $\i,\j,\k,\l$ in $\Farey{\Sf}$ are positively ordered. 
When $\KK=\RR$,
 the representations $\rhoZ$ with  positive FG-parameters
($\Z_\tau,\Ep_e\in \RR_{>0}$ for all $\tau,\e$) 
correspond to the holonomies of convex projective
structures on $\Sf$.
We now define the  {\em $\typeA_2$-complex $\CEsgZ$ associated with} 
a {\em \leftshifting{} geometric FG-parameter} $\gFGp$ in
$\RR^\TrianglesT \times \RR^\orEdgesT$.
Consider a {\em geometric FG-parameter}  $\gFGp=((\gZ_\tau)_\tau,(\gEp_\e)_\e)$ 
in $\RR^\TrianglesT \times \RR^\orEdgesT$.
We suppose that  $\gFGp$ is {\em \leftshifting} 
i.e. satisfies the following condition:
\begin{center}
(\HypLeftShift) 
  For each $\e\in\orEdgesT$,
with left and right 
  triangles $\tau$ and $\taup$, we have
$\gEp_\e     > \max\{- \gZmoins_\tau,- \gZplus_\taup\} $
\end{center}
where $\posp{\tRR}=\max(\tRR,0)$ and 
$\negp{\tRR}=\max(-\tRR,0)$ for $\tRR \in \RR$.
For each triangle $\tau$ of the triangulation $\T$,
pick a singular equilateral triangle $\cellCEsgZ^\tau$ 
in the model plane $\Aa$, 
with vertices $\vstA_1$, $\vstA_2$, $\vstA_3$,
and sides of $\Cc$-length
$\Cd(\vstA_1,\vstA_2)=(\gZplus_\tau,\gZmoins_\tau)$
in simple roots coordinates
(well-defined up to translations and action of $\W$),
see figure \ref{fig- intro-Aa-triangles}.

\begin{figure}[h]
  \includegraphics{figures/intro_Aa_triangles.fig}
  \caption{The singular triangle $\cellCEsgZ^\tau$ in $\Aa$.}
\label{fig- intro-Aa-triangles}
\end{figure}

When $\tau,\taup$ are adjacent along an edge $\e$ 
(oriented according to $\tau$), 
we connect the end
of the edge corresponding to $\e$ 
of the triangle $\cellCEsgZ^\tau$ 
to the beginning of the edge corresponding to $\e$ 
of the triangle $\cellCEsgZ^\taup$, 
by gluing  either
a segment $\cellCEsgZ^\e$ in $\Aa$ of $\Cc$-length  
$(\gEp_\eb,\gEp_\e)$,
 when $\gEp_\e,\gEp_\eb\geq 0$,
or  a flat strip $\cellCEsgZ^\e\subset \Aa$ 
such that $\cellCEsgZ^\e= [0, \gEp_\eb]\times [0, \gEp_\e]$ 
(in simple roots coordinates), 
when $\gEp_\e<0$ or $\gEp_\eb <0$,
as in figure \ref{fig-intro- gluings} 
(note that under hypothesis (\HypLeftShift) 
$\gEp_\e < 0$ implies that $\gEp_\eb \geq 0$).

\begin{figure}[h]
  \includegraphics[scale=0.75]{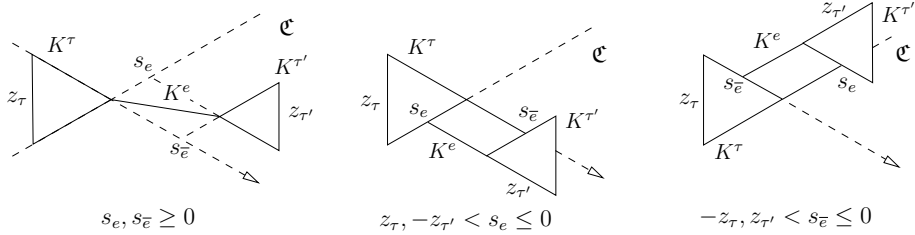}
  \caption{Gluings (local development in $\Aa$).}
  \label{fig-intro- gluings}
\end{figure}

The resulting finite $2$-dimensional  complex $\CEsgZ$ 
(see figure \ref{fig-intro- three A2-pants})
is 
a deformation retract of $\Sf$, 
and its fundamental group
has canonical identification with $\Ga=\pi_1(\Sf)$. 
The length metric on $\CEsgZ$
induced by the Euclidean $\W$-invariant metric on
$\Aa$ 
will be denoted by $\deuc$. 
Furthermore, the complex $\CEsgZ$ is endowed with 
a {\em $\typeA_2$-structure}
(charts in $\Aa$ with transition maps in $\W$).
Hence we may define the $\Cc$-length of piecewise affine paths in
$\CEsgZ$.
The $\Cc$-length 
$\ellC(\ga,\CEsgZ)$ of $\ga\in\Ga$
is then defined
as the $\Cc$-length 
of one (any) closed 
geodesic 
representing $\ga$.
We define the  $\Cc$-distance $\Cd$ on 
the universal cover $\CEstgZ$ of $\CEsgZ$
as the $\Cc$-length of the unique 
geodesic between two points. 
Note that, unlike in Euclidean buildings, in $\typeA_2$-complexes 
the $\Cc$-distance does not refine  the usual metric $\deuc$, 
in the sense
that the inequality
$\normeuc{\Cd(\xE,\yE)}\leq \deuc(\xE,\yE)$ may be strict.

There are several particular cases of special interest, 
providing a continuous transition from graphs to surfaces.
The geometric FG-parameters $\gFGp$ satisfying the condition 
\[
(\HypTree) 
\left\{
\begin{array}{l}
  \gZ_\tau = 0 \mbox{ for all  triangles } \tau \mbox{  of } \T\\ 
  \gEp_\e > 0  \mbox{ for all oriented edges }\e \mbox{  of } \T
\end{array}
\right.
\]
(which imply (\HypLeftShift)),
correspond to the case where $\CEsgZ$ is a graph 
(the $3$-valent ribbon graph dual to the ideal triangulation), 
endowed with a {\em $\Cc$-metric}.
Relaxing the hypotheses, 
 the condition
\begin{center}
(\HypTreeOfTriangles) $\gEp_\e \geq 0$ for all oriented edge $\e$ of $\T$   
\end{center}
means that all the $\cellCEsgZ^\e$ are segments 
so $\CEsgZ$ is obtained from the previous graph
by replacing vertices by triangles ({\em graph of triangles}).
At the opposite of the spectrum,
when 
\begin{center}
(\HypSf) $\gEp_\eb<0$ or $\gEp_\e<0$ for all oriented edge $\e$ of $\T$,   
\end{center}%
then $\CEsgZ$ is a \tT{} surface homeomorphic to $\Sf$.

\begin{figure}[h]
  \includegraphics[scale=0.5]{figures/A2-pantalon_intro.fig}
  \caption{Examples of $\typeA_2$-complex $\CEsgZ$ on a pair of pants,
corresponding to the conditions (\HypTree), (\HypTreeOfTriangles), and
(\HypSf) on the parameter $\gFGp$.}
  \label{fig-intro- three A2-pants}
\end{figure}

We now state the main result (see Theorem \ref{theo- Cgeod surface}).
We will need the following hypothesis:
A geometric  FG-parameter $\gFGp$ will be called {\em
\edgeseparating{}} if it satisfies the following condition.

\begin{center}
(\HypLeftShift) 
For each $\tau$ in  $\T$ and every pair of edges $\ei$,
$\eii$ of $\tau$, we have
$\left\{
\begin{array}{ll}
-\gEp_\ei-\gEp_\eii &<\gZmoins_\tau\\
-\gEp_\eib-\gEp_\eiib &<\gZplus_\tau
\end{array}
\right.
$\;.
\end{center}

\begin{theointro}
\label{theointro- Cgeod embbeded A2-complex}
Let $\FGp=((\Z_\tau)_\tau,(\Ep_\e)_\e)$ in 
$(\KKol)^\TrianglesT \times (\KKo)^\orEdgesT$, 
and 
denote by $\rho$ the representation $\rhoZ: \Ga\to \PGL_3(\KK)$ 
of FG-parameter $\FGp$.
Let $\gZ_\m=\log\abs{\Z_\m}$, $\gEp_\m=\log\abs{\Ep_\m}$
 and $\gZ=(\gZ_\m)_\m$, $\gEp=(\gEp_\m)_\m$.
Suppose that
\begin{mydescription}

\item[(\HypFlatTriangles)] For each triangle $\tau$ in  $\T$, 
we have $\abs{\Z_\tau+1}\geq 1$ ;

\item[(\HypH)] For each oriented edge $\e$  in  $\T$, 
we have 
$\abs{\Ep_\e+1}\geq 1$ ; 
\item[(\HypLeftShift)]  $\gFGp$ is \leftshifting{} ;
\item[(\HypEdgeSep)]   $\gFGp$ is \edgeseparating{} ;
\end{mydescription}
Let $\CEsgZ$ be the $\typeA_2$-complex 
of geometric FG-parameter $\gFGp$.
Then
there exists a $\rho$-equivariant map 
$$\Psi: \CEstgZ\to \Es$$ 
preserving  the $\Cc$-distance $\Cd$.
\end{theointro}

\begin{corointro}
Under the hypotheses of Theorem \ref{theointro- Cgeod embbeded
  A2-complex}, the following assertions holds.

\begin{enumerate}
\item The  $\Cc$-length spectra coincide, 
i.e. for all $\ga\in\Ga$
  \[ \ellC(\rho(\ga))=\ellC(\ga,\CEsgZ)\;.\]
In particular, the usual 
Euclidean and Hilbert length are given by
\begin{align*}
&\elleuc(\rho(\ga))=\normeuc{\ellC(\ga,\CEsgZ)}
,\\
\mbox{and } 
&\ellH(\rho(\ga))=\normH(\ellC(\ga,\CEsgZ)) 
\;.
\end{align*}

\item The map $\Psi$ is
bilipschitz.
In particular the representation $\rho$ is undistorted, 
i.e. for any fixed point $\xE$  in $\Es$, 
we have $\deuc(\xE,\rho(\ga)\xE) \simeq \norm{\ga}$
where $\norm{\ga}$ is the word length of $\ga$ in $\Ga$. 
\item The representation $\rho$ is faithfull and proper (hence discrete).
  \end{enumerate}

\end{corointro}

\begin{remas*}
  \begin{enumerate}

\item The image $\Y$ of $\Psi$ is a closed $\Cc$-convex subset of
  $\Es$ preserved by $\rho$,
  and $\Ga$ acts freely discontisnuously cocompactly on $\Y$.

\item 
  The $\Cc$-length spectrum of $\rhoZ$ depends only on
$\gZ=\log\abs{\Z}$, $\gEp=\log\abs{\Ep}$ 
(in particular it does not determine the
  representation up to conjugacy).

\item Note that, for positive representations (that is, with positive
  FG-parameters $\gZ_\tau,\gEp_\e>0$)  
in ordered fields  $\KK$, 
the hypothesis (\HypFlatTriangles) and (\HypH) are always satisfied.

\item
Note that (\HypLeftShift) and (\HypEdgeSep) are a finite system of
strict linear inequations in $\gZmoins_\tau,\gZplus_\taup$, 
in particular 
the subset $\ConeLS$ of \leftshifting{} and \edgeseparating{}
$\gFGp$ is a finite union 
of open convex polyhedral cones in $\RR^\TrianglesT \times \RR^\orEdgesT$.
It contains the non empty cone 
$\{0\}^\TrianglesT \times\RR_{>0}^\orEdgesT$ of $\gFGp$ satisfying (\HypTree).
For arbitrary fixed triangle parameters $\gZ_\tau$, conditions (\HypLeftShift)
and (\HypEdgeSep) are always
 satisfied for big enough edge parameters $\gEp_\e$.
In particular $\ConeLS$ is a nonempty open cone.

\item The  result holds in fact in a more general setting including
  exotic buildings, see Theorem \ref{theo- local Cgeod surface}.

\end{enumerate}

\end{remas*}

A special case
with much simpler hypotheses (and proof) 
is when $\FGp$ satisfies simply
\[(\HypTree')
\left\{
\begin{array}{l}
  \abs{\Z_\tau}=\abs{\Z_\tau+1} = 1 \mbox{ for all } \tau\\ 
  \abs{\Ep_\e} > 1  \mbox{ for all }\e
\;.\end{array}
\right.
\]
Then 
all hypotheses of  Theorem \ref{theointro- Cgeod embbeded A2-complex}
are satisfied, 
$\gFGp$ satisfies (\HypTree) and $\CEsgZ$ is a graph,
and the image $\Y$ of $\Psi$ is an invariant cocompact 
$\Cc$-convex (in particular bilipschitz)  tree in the building.
The hypotheses of  Theorem \ref{theointro- Cgeod embbeded
  A2-complex} are also satisfied in the other particular case
corresponding to the following open simple condition
\[(\HypTreeOfTriangles')
\left\{
\begin{array}{l}
  \abs{\Z_\tau} \neq  1 \mbox{ for all } \tau\\ 
  \abs{\Ep_\e} > 1  \mbox{ for all }\e
,\end{array}
\right.
\]
and $\gFGp$ satisfies (\HypTreeOfTriangles),
providing an invariant $\Cc$-convex  ``tree of triangles''  $\Y$.
On the other end of the spectrum, Theorem \ref{theointro- Cgeod embbeded
  A2-complex} provides (for $\gFGp$ satisfies (\HypSf)) 
examples  of representations whose image preserves
a $\Cc$-geodesic (in particular, bilipschitz)  surface $\Y$ in the building.

\medskip

In the last part of the paper, 
we use Theorem \ref{theointro- Cgeod embbeded A2-complex} 
to describe
limit of length functions (in the associated symmetric space) 
for a large family of degenerations of
representations $\Ga\to\PGL(\RR^3)$ corresponding to convex
$\RP^2$-structures on $\Sf$.

\begin{theointro}
\label{theointro- degeneration of reps}
Let $(\gFGpn)_{\n\in\NN}$ be a sequence 
in $\RR^\TrianglesT\times\RR^\orEdgesT$.
Let $\Zn_\tau=\exp(\gZn_\tau)$ and $\Epn_\e=\exp(\gEpn_\e)$.
Let $\rhon: \Ga\to \PGL_3(\RR)$ 
be the representation of FG-parameter
$\FGpn=((\Zn_\tau)_\tau,(\Zn_\e)_\e)$.
Let $(\lan)_\n$ be a sequence of real numbers going to $\pinfty$, 
such that 
the sequence $\lslan\gFGpn$ converges to a nonzero $\gFGp$ 
in $\RR^\TrianglesT\times\RR^\orEdgesT$.
Suppose that $\gFGp$ 
is \leftshifting{}  and \edgeseparating{} 
($(\HypLeftShift)$ and $(\HypEdgeSep)$).
Let $\CEsgZ$ be the $\typeA_2$-complex of FG-parameter $\gFGp$.
Then the renormalized $\Cc$-length spectrum of $\rhon$ 
converges to the $\Cc$-length spectrum of  $\CEsgZ$, 
that is: for all $\ga\in\Ga$ we have
$$\lslan \ellC(\rhon(\ga)) \to \ellC(\ga,{\CEsgZ})$$ 
in $\Cb$.
In particular for Euclidean and Hilbert lengths,
we have then:
$$
\lslan \elleuc(\rhon(\ga)) 
\to \normeuc{\ellC(\ga,{\CEsgZ})}$$
$$
\lslan \ellH(\rhon(\ga)) 
\to\normH(\ellC(\ga,\CEsgZ))
$$
 for all $\ga\in\Ga$.
\end{theointro}

A similar result holds in more general valued field $\KK$
(see Theorem \ref{theo- degeneration of  reps}).
Note that, for a given sequence $(\gFGpn)_{\n\in\NN}$ going to
infinity,  there always
exists a convenient sequence $\lan$, taking  
$\lan=\max_{\tau,\e}\abs{\gZn(\tau)},\abs{\gEpn(\e)}$.
This describes a part (corresponding to the open cone $\ConeLS$ of
FG-parameters)  of the boundary 
(constructed in \cite{ParComp})
of the  space $\SPC(\Sf)$ 
of convex real projective structures on $\Sf$ 
(see Coro. \ref{coro- bord SPC}).
Note that
D.~Cooper, K.~Delp, D.~Long and M.~Thistle\-thwaite
 announced  results similar to Theorem 
\ref{theointro- degeneration of reps}.
Our proofs involve a geometric interpretation of FG-parameters in
Euclidean buildings of type $\typeA_2$, 
relying on results from \cite{ParTriples}
describing the geometry of triples of ideal chambers 
in relation with their triple ratio as triples of
flags. 
It allows to associate with each triangle $\tau$ of the
triangulation $\Tt$ a singular flat triangle $\Delta_\tau$ 
in the building in a canonical way.
The map $\Psi$ is then defined by sending $\CEstgZ^\tau$ to
$\Delta^\tau$.
 The main technical difficulty 
is to  prove that the map $\Psi$ is globally
 $\Cc$-geodesic. 
Note that in the case (\HypTree') of trees 
the proofs are {\em much} simpler.
Application to degenerations of representations uses asymptotic cones,
and basically reduces to prove that the Fock-Goncharov
parametrization behaves well under ultralimits
(Proposition \ref{prop- FG coords goes to asymptotic cone}).

The structure of the paper is the following:
in Section \ref{s- geometric prelim},
we recall some basic facts about Euclidean buildings of type
$\typeA_2$ that will be used throughout the article, 
and we establish a criterion for a local
 $\Cc$-geodesic  to be a global $\Cc$-geodesic 
(Proposition \ref{prop- critere Cgeod avec sing}) that will be used to
prove global $\Cc$-geodesicity for $\Psi$.
In Section \ref{s- FG parametrization}, 
we explain Fock-Goncharov parametrization for
representations in any field $\KK$.
In Section \ref{s- construction A2-complexe CgZ},
we introduce the notion of $\typeA_2$-complexes,
and we construct the
$\typeA_2$-complex $\CEsgZ$ associated with a 
 \leftshifting{} geometric FG-parameter $\gFGp$ 
and discuss the special cases (from trees
to surfaces).
In Section \ref{s- actions on buildings},
 we study actions on Euclidean buildings (possibly exotic), 
introducing a purely geometric version of  FG-invariants, 
and we prove the main result (Theorem 
\ref{theointro- Cgeod embbeded A2-complex}) in this wider setting.
Finally, in Section \ref{s- degenerations of representations}, 
we study degenerations of representations
and prove Theorem \ref{theointro- degeneration of reps},
introducing asymptotic cones of projective spaces 
and studying 
the asymptotic behaviour of Fock-Goncharov parametrizations.

\paragraph{\bfseries{Aknowledgments}}
I would like to thank Frédéric Paulin 
for usefull discussions 
and comments on the preliminary version. 
I also want to thank the members of the  Institut
Fourier for their support.

\section{Geometric preliminaries}
\label{s- geometric prelim}

\subsection{Projective geometry}
\label{ss- projective spaces} 
We here collect notations for projective geometry
which will be used throughout this article.

\subsubsection*{\NondegeneratedQ{} quadruples on a projective line}
\label{s- nondegeneratedQ}
Cross ratios on projective lines will be defined on 
quadruples $(\xi_1,\xi_2,\xi_3,\xi_4)$ of points satisfying the
following nondegeneracy condition:
  ({\em no triple point}, i.e. any three of the points are not equal,
or, equivalently,
\begin{equation}
\label{eq- nondegeneratedQ}
  (\xi_1 \neq \xi_4 \mbox{ and } \xi_2 \neq \xi_3)
  \mbox{ or }
  (\xi_1 \neq \xi_2 \mbox{ and }  \xi_3 \neq \xi_4)
\;.
\end{equation}
The quadruple $(\xi_1,\xi_2,\xi_3,\xi_4)$ is 
then called {\em \nondegeneratedQ}.
\subsubsection*{Projective planes}
Let $\PP$ be a projective plane.
We denote by $\PP^*$  the dual projective plane,
i.e. the set  of lines in $\PP$.
We will denote $\p\oplus\q$ or $\p\q$ the line joining two distinct
point $\p$, $\q$ in $\PP$.

We denote by $\MaxFlags(\PP)$ the set of  (complete) flags 
$\F=(\p,\D)\in\PP\times\PP^*$,   $\p\in\D$, 
in the projective plane  $\PP$.
Two flags are called {\em opposite} if they are in generic position.

\subsubsection*{Triples of flags}
\label{s- nondegeneratedTF}
Let $\Tau=(\F_1,\F_2,\F_3)$ be a triple of flags
 $\F_\ii=(\p_\ii,\D_\ii)$ in $\PP$.
We will  denote by $\p_\ij$ the point $\D_\i\cap \D_\j$
 (resp. $\D_\ij$ the line $\p_\i\p_\j$), when defined.

%

%
The natural  nondegeneracy condition on the triple $(\F_1,\F_2,\F_3)$ 
for the triple ratios  to be well defined is the following:
\begin{center}
(\DefTR)
either for all $\ii$, $\p_\ii \notin \D_{\ii+1}$ 
or for all $\ii$, $\p_\ii \notin \D_{\ii-1}$. 
\end{center}
This condition is clearly equivalent to:
the points are pairwise distinct, the lines are pairwise distinct,
none of the points is on the three lines 
(i.e. $\D_\ii\cap \D_\jj \neq \p_\kk$ for all $\{\ii,\jj,\kk\}=\{1,2,3\}$)
and none of the lines contains the three points
(i.e. $\p_\ii \p_\jj \neq \D_\kk$ for all $\ii,\jj,\kk$). 
We will then say that the triple $(\F_1,\F_2,\F_3)$ is 
{\em \nondegeneratedTF}.

It is easy to check that the triple $\Tau$ defines then 
a \nondegeneratedQ{} quadruple of well-defined lines
$\D_\ii$, $\p_\ii\p_\jj$, $\p_\ii\p_\jk$, $\p_\ii\p_\kk$ 
through each point $\p_\ii$,
and a \nondegeneratedQ{} quadruple of well-defined points 
$\p_\ii$, $\D_\ii\cap\D_\jj$, $\D_\ii\cap\D_\jk$, $\D_\ii\cap\D_\kk$ 
on each line $\D_\ii$.

The  triple of flags  $\Tau=(\F_1,\F_2,\F_3)$ is {\em \genericTF} if
the flags $\F_\ii=(\p_\ii,\D_\ii)$ are pairwise opposite, 
the points $(\p_\ii)_\ii$ are not collinear 
and the lines  $(\D_\ii)_\ii$ are not concurrent.
In particular, $\Tau$ is then \nondegeneratedTF{}, and the induced
quadruples of points on each line (resp. of lines through each point)
are generic (pairwise distinct).

\subsection{The model finite reflection group 
\texorpdfstring{$(\Aa,\W)$ of type $\typeA_2$}
{(A,W) of type A\_2}
}
 The {\em model flat} (of type $\typeA_2$)
is  the vector space 
$\Aa=\RR^3/\RR(1,1,1)$, endowed with the action of
the {\em Weyl group} $\W=\Sym_3$ acting on $\Aa$ by permutation of
coordinates, which is a finite reflection group. 
We denote by $\Waff$ the subgroup of  affine isomorphisms of $\Aa$ with
linear part in $\W$.
We denote by 
$\classA{\vRN}$
the projection in $\Aa$ of a vector $\vRN$ in $\RR^3$.
The vector space $\Aa$ will be identified with the hyperplane 
$\{\vA=(\vA_1,\vA_2,\vA_3)\in\RR^3/\ \sum_\indN \vA_\indN=0\}$ of $\RR^3$.

Recall that a vector in $\Aa$ is called {\em singular} if it belongs
to one the three lines
$\vA_\indN=\vA_\indNp$, and {\em regular} otherwise. 
A {\em (open) (vectorial) Weyl chamber} of $\Aa$ is a connected component of
regular vectors.
The {\em model Weyl chamber} 
is $\Cc=\{\vA\in\Aa/\ \vA_1 > \vA_2 >\vA_3\}$.
 Its closure $\Cb$ is a strict fundamental domain for the action of
 $\W$ on $\Aa$, 
and we denote by $\Cabso:\Aa\to \Cb$ the canonical projection, which maps
a vector $\vA\in\Aa$ to its {\em \type} in $\Cb$.
We denote by $\bord \Aa$ the subset of unitary vectors in $\Aa$,
identified with the set $\PP^+(\Aa)=(\Aa-\{0\})/\RR_{>0}$ 
of rays issued from $0$,
and $\bord: \Aa \to \bord \Aa$ the corresponding projection.
The {\em type (of direction)}  of 
a nonzero vector $\vA \in \Aa$ 
is its canonical projection $\bord(\Cabso(\vA))$ in $\bord \Cb$.

The  {\em simple roots} (associated with $\Cc$) are the linear forms 
$$\rac1: \vA\mapsto \vA_1-\vA_2$$
$$\rac2: \vA \mapsto\vA_2-\vA_3$$
and we denote by
$\rac3: \vA \mapsto\vA_3-\vA_1 $
the root satisfying $\rac1+\rac2+\rac3=0$.

A singular vector $\vA$ is said to be {\em of type $\stypep$} 
if its type in $\Cb$ 
satisfies $\vA_1 > \vA_2=\vA_3$,
and {\em of type $\stypeD$} 
if its type 
satisfies $\vA_1=\vA_2> \vA_3$.
Recall that two nonzero vectors $\vA$ and $\vAp$ of $\Aa$ 
are called {\em opposite} if $\vAp=-\vA$.
Similarly, two Weyl chambers $\Ch$ and $\Chp$ of $\Aa$ 
are {\em opposite} if $\Chp=-\Ch$.
We denote by $\wopp$ the unique element of $\W$ sending $\Cc$ to  $-\Cc$,
and by $\vA^\opp=\wopp(-\vA)=(-\vA_3,-\vA_2,-\vA_1)$ the
image of $\vA$ by the {\em opposition involution} $\opp$ of $\Aa$.

We will normalize the $\W$-invariant Euclidean norm $\norm{\cdot}$ on
$\Aa$ by requiring that the simple roots have unit norm.  
The associated Euclidean metric on $\Aa$ is denoted by $\deuc$.
The {\em $\Cc$-distance} on $\Aa$ (or {\em $\Cc$-length} of segments)
is the canonical projection 
$\Cd:\Aa\times\Aa\to \Cb$
which is defined by $\Cd(\xA,\yA)=\Cabs{\yA-\xA}$.

\label{s- def normH}%
We will denote by $\normH$ the {\em \hex}-norm, 
that is the $\W$-invariant norm  on $\Aa$ 
defined by 
$$\normH(\vA)=\vA_1-\vA_3=-\rac3(\vA)$$
for $\vA$ in $\Cc$, 
whose unit ball is a  regular hexagon with singular sides. 

\subsection{Euclidean buildings}
The Euclidean buildings considered in this article are 
  $\RR$-buildings, in particular they are not necessarily discrete
 (have no simplicial complex structure) nor locally compact.
We refer to \cite{ParImm} for their definition and basic properties
(see also \cite{Tits86}, \cite{KlLe97}, \cite{Rousseau09}). 
Let $\Es$ be a Euclidean building of type $\typeA_2$.
Recall that 
$\Es$ is a $\CAT(0)$ metric space 
endowed with a (maximal) collection $\Apps$ of
isometric embeddings $\f:\Aa \to \Es$ called  {\em marked apartments},
or {\em marked flats} by analogy with  Riemannian symmetric spaces,
satisfying the following properties
\begin{mydescription}
  \item[{\bfseries (A1)}] $\Apps$ is invariant by precomposition by $\Waff$
    ;
  \item[{\bfseries (A2)}] If $\f$ and $\f'$ are two marked flats, then the
    transition map $\f^{-1}\circ \f'$ is 
in
 $\Waff$
    ;
  \item[{\bfseries (A3')}] Any two rays of $\Es$ are initially 
 contained in a common marked flat.
\end{mydescription}
 The {\em flats} 
(resp. the {\em Weyl chambers}) 
of $\Es$ are the images of $\Aa$ (resp. of $\Cc$)
by the marked flats.

We say that we are in the {\em algebraic case}
when $\Es$ is the Euclidean building $\Es(\V)$ associated with some
$3$-dimensional vector space $\V$ on an ultrametric field $\KK$.
We then denote by $\abs{\cdot}$ the absolute value of $\KK$.

Recall that, in Euclidean buildings, 
two (unit speed) geodesic segments issued from
a common point $\xE$ have zero angle 
if and only if 
they have same germ at $\xE$ (i.e. coincide in a neighborhood of $\xE$).
A {\em direction} at $\xE\in\Es$ is 
a germ of (unit speed) geodesic segment from $\xE$.
A direction, geodesic segment, ray or line 
has a well-defined {\em type (of direction)} in $\bord\Cb$, 
which is its canonical
projection (through a marked flat) in $\bord\Cb$.
It is called {\em singular} or {\em regular} accordingly.
The {\em space of directions} (or {\em unit tangent cone}) at $\xE$
is denoted by $\TangS_\xE\Es$. It is endowed with the angular metric.
We denote by   $\TangS_\xE:\Es-\{\xE\}\to \TangS_\xE \Es$
the associated projection.
The space of directions $\TangS_\xE\Es$  
is a spherical building of type $\typeA_2$, 
whose apartment are the germs $\TangS_\xE \App$ at $\xE$
of the flats $\App$ of $\Es$ passing through $\xE$,
and whose chambers (i.e. $1$-dimensional simplices) 
are the germs $\TangS_\xE \Ch$ at $\xE$ 
of the Weyl chambers $\Ch$ of $\Es$ with vertex $\xE$ 
(see for example \cite{ParImm}).
The {\em local projective plane  at $\xE$} $\PP_\xE=\PP_\xE(\Es)$
is the projective plane associated to the spherical
$\typeA_2$-building  $\TangS_\xE\Es$,
i.e. the projective plane whose incidence graph is $\TangS_\xE\Es$:
Its points are the singular directions of type $\stypep$
and its lines are the singular directions of type $\stypeD$ at $\xE$.
Recall that, in a spherical building, any two points (resp. chambers)
 are contained in a common apartment,
and that they are  {\em opposite} if they
are opposite in that apartment.

Two Weyl chambers $\Ch,\Chp$ of $\Es$ with common vertex $\xE$
 are {\em opposite} (at $\xE$) 
if their union contains a regular geodesic line passing
by $\xE$, 
or, equivalently, 
if they define opposite chambers $\TangS_\xE\Ch$,
$\TangS_\xE\Chp$ in the spherical building $\TangS_\xE\Es$ of
directions at $\xE$.
Then there exists a unique flat of $\Es$ containing both $\Ch$ and $\Chp$.

\subsection{The boundary of a 
\texorpdfstring{$\typeA_2$}{A\_2}%
-Euclidean building and its projective geometry}

\subsubsection{The projective plane at infinity}
We denote by $\bordinf\Es$ the $\CAT(0)$ boundary of $\Es$.
The {\em type} 
of an ideal point
$\xi\in\bordinfEs$ is the type in $\bord\Cb$ of any ray to $\xi$. 
The boundary $\bordinf\Es$ of $\Es$ is the incidence graph of a
projective plane $\PP=\PP_\infty(\Es)$ whose points are the singular
points of type $\stypep$ of $\bordinf\Es$ and lines are the singular
points of type $\stypeD$ of $\bordinf\Es$.
The set $\bordF\Es$ of chambers at infinity of $\Es$
(Furstenberg boundary) identifies then with 
the set $\MaxFlags(\PP)$ of (complete) flags 
$\F=(\p,\D)\in \PP\times\PP^*$, $\p\in\D$, 
in the projective plane $\PP$.

In the algebraic case, the projective plane $\PP$ at infinity of
$\Es=\Es(\V)$ is the classical projective plane $\PV$.

For $\xE\in\Es$, 
we denote by $\TangS_\xE:\yE \to \TangS_\xE\yE$ 
the canonical projection from $\bordinf\Es$ to the unit tangent cone
$\TangS_\xE\Es$ at $\xE$. 
The canonical projection $\TangS_\xE: \bordinf\Es \to \TangS_\xE \Es$
preserves the simplicial structure
and  the type (in $\bord\Cb$) of points, and in particular it 
induces  the {\em  canonical projection}
$\TangS_\xE:\PP \to \PP_\xE$, 
which is a surjective morphism of projective planes 
(i.e. if $\p\in\PP$ and $\D\in\PP^*$, 
then $\TangS_\xE\p \in\PP_\xE$ and $\TangS_\xE\D \in\PP_\xE^*$,
and $\p\in\D$ implies $\TangS_\xE\p \in \TangS_\xE\D$).

If $\chp$ and $\chm$ are opposite flags in $\PP$ (i.e. chambers at
infinity of $\Es$), then
we denote by $\App(\chm,\chp)$ the unique flat joining $\chm$ to $\chp$ in
$\Es$.
A basic fact is that given a generic (i.e. non collinear) 
triple of  points $\p_1,\p_2, \p_3$ 
in $\PP$ there exists a unique flat $\App(\p_1,\p_2, \p_3)$ of
$\Es$ containing them in its boundary 
(and the analog holds for lines).

\subsubsection{Transverse trees at infinity}
\label{s- transverse spaces at infinity}
(See for example \cite[\S 8]{Tits86}, 
\cite[1.2.3]{Leeb00}, \cite[\S 4]{MSVM14-I}.)
We denote by $\Es_\xi$ the transverse tree at a singular ideal point $\xi$
in $\bordinf\Es$ 
which may be defined, from the metric viewpoint,  as 
the space of classes of strongly asymptotic rays to $\xi$
the quotient space of the space of all rays to $\xi$ 
by the pseudodistance $\dxi$ given by
$$\dxi(\ray_1,\ray_2)=\inf_{\tRR_1,\tRR_2}\deuc( \ray_1(\tRR_1),
\ray_2(\tRR_2) ) \;.$$

We denote by $\projE_\xi:\Es\to\Es_\xi$ the canonical projection. 
Recall that $\Es_\xi$ is a $\RR$-tree, and that its
boundary $\bordinf\Es_\xi$ identifies with the set of singular points
of $\bordinf\Es$ adjacent to $\xi$. 
In particular,
if $\p$ is a point in $\PP$,
 then the boundary of the associated tree $\Es_\p$
 is identified with the set $\p^*$ of lines
$\D$ through $\p$ in the projective plane $\PP$.
Similarly, the boundary of the tree $\Es_\D$ associated with a
line $\D$ of $\PP$ is identified with the set $\D^*$ 
of points $\p$ of $\PP$ that belong to $\D$.

\subsubsection{The 
\texorpdfstring{$\Aa$}{A}-valued Busemann cocycle}
We denote by $\Buso_\ch:\Es\times\Es\to\Aa$ 
the $\Aa$-valued Busemann cocycle 
associated with an ideal chamber $\ch$ of $\Es$, 
which is defined by
$$\Bus{\ch}{\f(\vA)}{\f'(\vA')}=\vA'-\vA$$
for all marked flats $\f, \f': \Aa \to \Es$ sending $\bord\Cc$ to
$\ch$ and {\em very strongly asymptotic}
that is 
such that $\deuc(\f(\ray(\tRR)), \f'(\ray(\tRR)))$ goes to zero when
$\tRR\to \pinfty$ for one (all) regular ray $\ray$ in $\Cc$
(which in Euclidean buildings is equivalent to: 
$\f=\f'$ on some subchamber $\vA"+\Cc$).
\label{ss- usual Bus cocycle}
Note that in rank one (when $\dim\Aa=1$) this is the usual Busemann
cocycle, which is defined by
$$\Bus{\xi}{\xE}{\yE}=\lim_{\zE\ra \xi}\deuc(\xE,\zE)-\deuc(\yE,\zE)$$

We will use the following basic property, that describes the behaviour
of Busemann cocycle associated with ideal chamber $\ch=(\p,\D)$
upon projections to transverse trees at infinity $\Es_\p$ and $\Es_\D$.

\begin{equation}
\label{eq- projections and Busemann cocycle}
    \begin{array}{rl}
\rac1(\Bus{(\p,\D)}{\xE}{\yE})
&= \Bus{\p} {\projE_\D(\xE)} {\projE_\D(\yE)}\\
\rac2(\Bus{(\p,\D)}{\xE}{\yE})
&= \Bus{\D} {\projE_\p(\xE)} {\projE_\p(\yE)}
    \end{array}
\end{equation}

If $\chp$ and $\chm$ are opposite chambers at infinity, then
  \begin{equation}
\label{eq- Bus and opp chambers}
   \Bus{\chp}{\xE}{\yE}=-(\Bus{\chm}{\xE}{\yE})^\opp \mbox{ for } \xE,\yE
   \mbox{ in the flat }\App(\chm,\chp)
  \end{equation}

\subsubsection{Cross ratio on the boundary of a tree.}
\label{ss- geometric cross ratio on the boundary of a tree}
(See \cite[\S 7]{Tits86}, 
and for a more general setting \cite{Otal92}, \cite{Bourdon96}).
In this section, we suppose that $\Es$ is 
a $\RR$-tree, 
and we denote by $\bordinfEs$ its boundary at infinity.
Given three distinct ideal points $\xi_1,\xi_2,\xi_3$ in $\bordinfEs$,
we denote by $\centre(\xi_1,\xi_2,\xi_3)$ the {\em center} of the ideal
triple $\xi_1,\xi_2,\xi_3$, that is the unique intersection point of
the three geodesics joining two of the three points.
\noindent 
{
  {\advance\linewidth by-6.2cm
    \begin{minipage}{\linewidth}
      The {\em cross ratio} of four pairwise distinct 
      points $\xi_1$, $\xi_2$, $\xi_3$, $\xi_4$ in  $\bordinfEs$ 
      is defined 
      as the oriented distance 
      on the geodesic
      from $\xi_3$ to $\xi_1$, 
      %
      %
      %
      %
      from the center $\xT$ of the ideal triple $\xi_3,\xi_1,\xi_2$
      to the center $\yT$ of the ideal triple $\xi_3,\xi_1,\xi_4$
      \begin{equation}
        \label{eq- geombir and centers of tripods}
        \geombir(\xi_1, \xi_2, \xi_3,\xi_4) 
        = \ovra{\xT\yT}
        = \Bus{\xi_1}{\xT}{\yT}
        \;.
      \end{equation}
    \end{minipage}
  }  %
  \hfill
  \begin{minipage}{5.3cm}
    %
    \centering\includegraphics{figures/geom_crossratio_tree_centers_tripods.fig}
  \end{minipage}
}%
In the case where some of the points coincide,
the cross ratio is still defined if
the quadruple $(\xi_1,\xi_2,\xi_3,\xi_4)$ is \nondegeneratedQ{}
(see section \ref{s- nondegeneratedQ}).
It is then set to $0$ when $\xi_1 = \xi_3$ or $\xi_2 = \xi_4$,
$\minfty$ when $\xi_1 = \xi_2$ or $\xi_3 = \xi_4$,
and $\pinfty$ when $\xi_1 = \xi_4$ or $\xi_2 = \xi_3$.

We recall that the cross ratio is invariant under double
transpositions and satisfies the following properties.

\begin{prop} 
\label{prop- geombir}
We have
  \begin{enumerate}
  \item 
    $ \geombir(\xi_3, \xi_2, \xi_1, \xi_4)
    = \geombir(\xi_1, \xi_4, \xi_3, \xi_2)
    = -\geombir(\xi_1, \xi_2, \xi_3, \xi_4)$ ;
  \item 
    $ \geombir(\xi_1, \xi_2, \xi_3, \xi_4)
    +  \geombir(\xi_1, \xi_4, \xi_2, \xi_3)
    +  \geombir(\xi_1, \xi_3, \xi_4, \xi_2) 
    = 0 $ ;

\item 
if $\geombir(\xi_1, \xi_2, \xi_3,\xi_4)>0$,\\
then $\geombir(\xi_1, \xi_3, \xi_4,\xi_2)=0$ 
and  $\geombir(\xi_1, \xi_4, \xi_2,\xi_3)=-\geombir(\xi_1,
\xi_2,\xi_3,\xi_4)$ ;

%

%
  \item 
    $\geombir(\xi_1, \xi_2, \xi_3,\xi_4)
    +\geombir(\xi_1, \xi_4,\xi_3,\xi_5)
    =
    \geombir(\xi_1, \xi_2, \xi_3,\xi_5)$ \;.
  \end{enumerate}
\end{prop}

\subsubsection{Cross ratio on the boundary of a
\texorpdfstring{$\typeA_2$}{A\_2}-Euclidean building}
\label{s- geom cross ratio in building}
See \cite{Tits86}.
Let $\Es$ be a Euclidean building of type $\typeA_2$ and $\PP$ the
associated projective plane at infinity. 
We denote by $\geombir(\p_1, \p_2,\p_3, \p_4)$
the {\em (geometric) cross ratio} 
({\em projective valuation} in \cite{Tits86})
of a \nondegeneratedQ{} 
quadruple $(\p_1, \p_2,\p_3, \p_4)$
of points lying on a common line $\D$ of $\PP$. 
We recall that  it is defined 
as their cross ratio as points 
in the boundary of the transverse tree $\Es_\D$ at ideal point
$\D$ of $\Es$. 
We similarly denote by $\geombir(\D_1, \D_2,\D_3, \D_4)$
the geometric cross ratio 
of four lines $\D_1$, $\D_2$, $\D_3$, $\D_4$ 
through a common point $\p$ of $\PP$, 
which is defined as their cross ratio as
points in the boundary 
of the transverse tree $\Es_\p$ at ideal point $\p$ of $\Es$. 
Recall that perspectivities
preserve cross ratios, that is
$$\geombir(\p_1, \p_2,\p_3, \p_4)=\geombir(\q\p_1, \q\p_2,\q\p_3, \q\p_4)$$
$$\geombir(\D_1, \D_2,\D_3, \D_4)
=\geombir(\L\cap\D_1,\L\cap\D_2,\L\cap\D_3, \L\cap\D_4)$$
(when defined).

In the algebraic case, $\PP=\PP(\KK^3)$ and 
the  geometric cross ratio $\geombir$ is then obtained 
from the usual (algebraic) cross ratio $\Bir$ 
(see section \ref{ss- algebraic cross ratio} for the definition)
by
\begin{equation}
  \label{eq- link geom and alg cross ratio} 
  \begin{array}{lcl}
    \geombir(\p_1, \p_2,\p_3, \p_4)
    & = & \log \abs{ \Bir(\p_1, \p_2,\p_3, \p_4)}\\
    \geombir(\D_1, \D_2,\D_3, \D_4)
    & = & \log \abs{ \Bir(\D_1, \D_2,\D_3, \D_4)}
  \end{array}
\end{equation}
 (see for example \S 1.10 in \cite{ParTriples}).

\subsection{\texorpdfstring{$\Cc$}{C}-distance, translation lengths,
and \texorpdfstring{$\Cc$}{C}-geodesics}

\subsubsection*{The \texorpdfstring{$\Cc$}{C}-distance}
The {\em $\Cc$-distance} on $\Es$ 
is the map  $\Cd:\Es\times\Es\to \Cb$ 
defined by
$\Cd(\f(\xA),\f(\yA))=\Cd(\xA,\yA)$
for any marked flat $\f:\Aa\to \Es$ and $\xA,\yA\in\Aa$.
Note that we  have
$\Cd(\yE,\xE)=\Cd(\xE,\yE)^ \opp$.
The $\Cc$-distance may be seen as a refinement of 
the usual distance $\deuc$, since
$$\deuc(\xE,\yE)=\normeuc{\Cd(\xE,\yE)}\;.$$

\subsubsection*{The \texorpdfstring{$\Cc$}{C}-length of an autorphism}
Let $\g$ be an automorphism of $\Es$.
The usual {\em (translation) length} of $\g$ 
is $\elleuc(\g) = \inf_{\xE\in\Es} \deuc(\xE,\g\xE)$,
and 
will be called the {\em Euclidean (translation) length} of $\g$.

We will denote by $\ellC(\g)$ the   {\em $\Cc$-(translation) length} of
$\g$ (called {\em vecteur de translation} in \cite{ParComp}),
which is the unique vector of minimal length in (the closure in $\Cb$
of)
$\{\Cd(\xE,\g\xE),\ \xE\in\Es\}$.
We recall that in the algebraic case, for $\g\in\PGL_3(\KK)$, 
we have 
$$\ellC(\g)=\classA{(\log\abs{\aK_\indN})_\indN}$$ 
where the $\aK_\indN$ are the eigenvalues of  $\g$. 
The $\Cc$-length refines the Euclidean length as
$\elleuc(\g)=\norm{\ellC(\g)}$.
We will also consider the {\em Hilbert length}  
$$\ellH(\g)=\normH(\ellC(\g)$$
of $\g$, which correspond 
to the translation length for the Hilbert metric in the
case of holonomies of convex projective structures.

\subsubsection*{The \texorpdfstring{$\Cc$}{C}-geodesics}
The {\em $\Cc$-length} of a piecewise affine path $\apath$ 
with vertices  $\xE_0$, $\xE_1$,\ldots, $\xE_\IndPath$ in $\Es$
is the vector 
$$\ellC(\apath)=\sum_{\indPath}\Cd(\xE_\indPath,\xE_{\indPath+1})$$ 
in the closed Weyl chamber $\Cb$.
\begin{defi}
  A piecewise affine path $\apath: [ 0 , \sRR ] \to \Es$
will be called a {\em $\Cc$-geodesic} if there is a marked flat
$\f:\Aa\to\Es$ such that 
$\apath$ is  the image by $\f$ of a
(piecewise affine) path $\apathA: [ 0 , \sRR ]\to \Aa$ such that 
$\dot{\apathA}(\tRR)\in\Cb$ for almost all $\tRR\in [ 0 , \sRR ]$.
\end{defi}

Note that a piecewise affine path in $\Aa$
is a $\Cc$-geodesic if and
only if it is a geodesic for the \hex-metric 
(that is the metric induced by the \hex-norm $\normH$).
 The following proposition 
collects some obvious properties of $\Cc$-geodesics
that are  needed in this article 
(they actually satisfy stronger properties, see \cite{ParCd}).

\begin{prop}
Let $\apath: [ 0 , \sRR ] \to \Es$ be a $\Cc$-geodesic from $\xE$ to
$\yE$ in $\Es$. Then 
\begin{enumerate}
\item 
the $\Cc$-length $\ellC(\apath)$ of $\apath$ 
is equal  to 
the $\Cc$-distance $\Cd(\xE,\yE)$,

\item any flat containing $\xE$ and $\yE$ contains $\apath$.
\qed
\end{enumerate}
\end{prop}
\subsubsection*{A local criterion}

We say that two directions in $\TangS_\xE\Es$ are 
{\em $\Cc$-opposite} 
if they are contained in opposite closed chambers of $\TangS_\xE\Es$.
For $\yE\neq\xE$ in $\Es$, 
we denote by $\fac_\xE(\yE)$ 
the minimal closed simplex  of  $\TangS_\xE\Es$  
containing $\TangS_{\xE}\yE$.

\begin{prop}
\label{prop- critere 3 points et facettes croissantes}
Let $\xE,\yE,\zE\in\Es$, with $\yE\neq\xE,\zE$.
The following are equivalent:
  \begin{enumerate}

  \item  The path ($\xE$, $\yE$, $\zE$) is $\Cc$-geodesic ;

\item The directions $\TangS_{\yE}\xE$ and  $\TangS_\yE\zE$ 
are $\Cc$-opposite in $\TangS_\yE\Es$. 
  \end{enumerate}
Then $\xE\neq \zE$ and  $\TangS_{\xE}(\yE)$ belongs to $\fac_\xE(\zE)$.
\end{prop}

\begin{proof}
This follows from the fact that two opposite Weyl chambers at $\yE$ are
contained in a flat.
\end{proof}

\begin{rema}
A key difficulty is that, unlike in the usual cases, 
a path may be locally $\Cc$-geodesic but not globally $\Cc$-geodesic,
even for arbitrary close deformations.
Easy examples can be found in products of two trees, taking 
 in any flat identified with $\RR\times\RR$ a ``$U$''-path: for
 instance the piecewise affine path with successive 
vertices $\xE_0=(0,1)$, $\xE_1=(0,0)$, $\xE_2=(1,0)$, $\xE_3=(1,1)$.
In Euclidean buildings of type $\typeA_2$, an example is the piecewise affine
path $\apath$ in $\Aa$ with vertices $\xE_0=\classA{(-1,2,-1)}$,
$\xE_1=0$, $\xE_2=\classA{(2,-1,-1)}$ and $\xE_3=\classA{(3,0,-3)}$,
which is a local $\Cc$-geodesic but not globally $\Cc$-geodesic
 (see Figure \ref{fig- local but not glob Cgeod}).
\begin{figure}[h]
  \includegraphics[scale=0.8]{figures/Aa_local_but_not_global_Cgeod.fig} 

  \caption{A local, but not global, $\Cc$-geodesic in $\Aa$.}
\label{fig- local but not glob Cgeod}
\end{figure}
This phenomenon makes it hard to prove global preservation of the
$\Cc$-distance for maps between subset of Euclidean buildings, 
since it is not enough to check it locally.
\end{rema}

\subsubsection*{A local to global criterion}
For piecewise regular $\Cc$-geodesic paths, 
we have the following  fundamental local-to-global property:
 
\begin{coro}
Let $(\xE_{\indPath})_\indPath$ be a  (finite or not) sequence  in $\Es$.
Suppose that for all $\indPath$ 
the segment $[\xE_{\indPath},\xE_{\indPath+1}]$ is regular,
and 
the path $(\xE_{\indPath-1}$, $\xE_{\indPath}$, $\xE_{\indPath+1})$ is $\Cc$-geodesic.
Then the whole path $(\xE_{\indPath})_\indPath$ is $\Cc$-geodesic.
\end{coro}
We now state a criterion 
for a general locally $\Cc$-geodesic piecewise affine path
to be $\Cc$-geodesic, 
which will be used in the proof of the main theorem 
(Section \ref{ss- proof  global Cisom}).

\begin{prop}
\label{prop- critere Cgeod avec sing}
Suppose that $\dim \Aa=2$.
Let $(\xE_{\indPath})_\indPath$ be a  (finite or not) sequence  in $\Es$, such that
for all $\indPath$ the point $\xE_{\indPath}$ is not in the segment 
$[\xE_{\indPath-1},\xE_{\indPath+1}]$.
Suppose that: 
\begin{enumerate}
\item (local $\Cc$-geodesic)
\label{it- hyp local Cgeod}
For all $\indPath$ the directions
$\TangS_{\xE_{\indPath}}\xE_{\indPath-1}$ and  $\TangS_{\xE_{\indPath}}\xE_{\indPath+1}$ 
are $\Cc$-opposite in $\TangS_{\xE_{\indPath}}\Es$.
\item For all $\indPath$ such that $[\xE_{\indPath-1},\xE_{\indPath}]$ is singular, 
$\TangS_{\xE_{\indPath}}\xE_{\indPath-2}$ and  $\TangS_{\xE_{\indPath}}\xE_{\indPath+1}$ 
are $\Cc$-opposite in $\TangS_{\xE_{\indPath}}\Es$.
\end{enumerate}
Then $(\xE_{\indPath})_\indPath$ is $\Cc$-geodesic.
\end{prop}

Note that all involved directions are well defined, 
since we have 
$\xE_\indPath\neq \xE_{\indPath-1},\xE_{\indPath+1}$ for all
$\indPath$, 
and hypothesis (\ref{it- hyp local Cgeod}) implies that
$\xE_{\indPath-1} \neq \xE_{\indPath+1}$ for all $\indPath$.

\begin{proof}
Suppose that  $(\xE_{0}, \xE_{1},\ldots, \xE_{\indPath})$ 
is $\Cc$-geodesic for some $\indPath \geq 2$.
In the spherical building $\TangS_{\xE_\indPath}\Es$ of directions at
$\xE_\indPath$,
Proposition \ref{prop- critere 3 points et facettes croissantes}
implies the following inclusions of simplices:
$\fac_{\xE_\indPath}(\xE_{\indPath-1})
\subset \fac_{\xE_\indPath}(\xE_{\indPath-2})
\subset \fac_{\xE_\indPath}(\xE_{0})$. 
Note that, since $\xE_{\indPath-1}$ is not in $[\xE_{\indPath},\xE_{\indPath-2}]$,
the segment $[\xE_{\indPath},\xE_{\indPath-2}]$ is necessarily regular, hence 
$\fac_{\xE_\indPath}(\xE_{\indPath-2})$ 
is a closed chamber (i.e. a maximal simplex),
and then
$\fac_{\xE_\indPath}(\xE_{\indPath-2})=\fac_{\xE_\indPath}(\xE_{0})$.

If the segment $[\xE_{\indPath-1},\xE_{\indPath}]$ is regular, 
then
$\fac_{\xE_\indPath}(\xE_{\indPath-1})=\fac_{\xE_\indPath}(\xE_{\indPath-2})=\fac_{\xE_\indPath}(\xE_{0})$.
By hypothesis $\TangS_{\xE_\indPath}\xE_{\indPath+1}$ is in a closed chamber
opposite to the closed chamber $\fac_{\xE_\indPath}(\xE_{\indPath-1})=\fac_{\xE_\indPath}(\xE_{0})$,
hence $\TangS_{\xE_\indPath}\xE_0$ is $\Cc$-opposite to $\TangS_{\xE_\indPath}\xE_{\indPath+1}$.

If the segment $[\xE_{\indPath-1},\xE_{\indPath}]$ is singular,
then by hypothesis $\TangS_{\xE_\indPath}\xE_{\indPath+1}$ is in a closed chamber
opposite to the closed chamber
$\fac_{\xE_\indPath}(\xE_{\indPath-2})=\fac_{\xE_\indPath}(\xE_{0})$, hence
 $\TangS_{\xE_\indPath}\xE_0$ is also $\Cc$-opposite to $\TangS_{\xE_\indPath}\xE_{\indPath+1}$.
 
Then in all cases $\xE_{0}$, $\xE_{\indPath}$, $\xE_{\indPath+1}$ is $\Cc$-geodesic
(Proposition \ref{prop- critere 3 points et facettes croissantes}),
and it follows that $(\xE_{0}, \xE_{1},\ldots, \xE_{\indPath+1})$
 is $\Cc$-geodesic.
\end{proof}

\section{Fock-Goncharov parameters for surface group representations}
\label{s- FG parametrization}
In this section, following Fock and Goncharov \cite{FoGoSPC}, 
we explain in detail how to build representations of a
punctured surface group in $\PGL_3(\KK)$ for any field $\KK$ using
ideal triangulations and projective geometry. The goal is to define
the representation $\rho_{\FGp}$ associated with a {\em FG-parameter}
$\FGp= ((\Z_\tau)_\tau,(\Ep_\e)_\e)$.
Note that our edge parameters $\Ep_\e$ 
are in fact a slight modification of those in
 \cite{FoGoSPC}, and are more symmetric with respect to natural
 point-line duality (see \S \ref{ss- relation with usual FG} for the
 precise relationship).

In this section,  $\KK$ is any field and $\PP=\PP(\KK^3)$.

\subsection{Surfaces and ideal triangulations}
\label{ss- prelim surf and ideal triang}
Consider a compact oriented connected surface  $\Sf$ 
with non empty boundary 
and negative Euler characteristic $\chi(\Sf)<0$. 
Boundary components
of $\Sf$ are oriented in such a way that  
the surface lies to their right. 
They will also be  seen  as punctures.
Let $\Ga=\pi_1(\Sf)$ be the fundamental group of $\Sf$.
We denote by  $\Farey{\Sf}$ 
the {\em Farey set} of $\Sf$,
which may be defined as the set of boundary components of the universal cover
$\Sft$ of $\Sf$ (see \cite[\S 1.3]{FoGoIHES}). 
This set inherits a cyclic order from the orientation of the surface.
For each $\i\in \Farey{\Sf}$, 
we denote by $\ga_\i$ the
corresponding element of $\Ga$, i.e the primitive element 
translating the boundary component $\i$ in the positive direction.
Then for the induced order on $\Farey{\Sf} -\{\i\}$, we have
$\ga_\i(\j) > \j$ for all $\j \neq \i$.
The fundamental group $\Ga=\pi_1(\Sf)$ acts on the Farey set
$\Farey{\Sf}$, and $\ga_\i$ fixes $\i$ for each $\i\in\Farey{\Sf}$.
Let $\T$ be an ideal triangulation of $\Sf$, 
i.e
a triangulation  with vertices the boundary components,
considered as punctures.
We denote by  
$\Triangles(\T)$ the set of  triangles of $\T$ 
and by
$\orEdges(\T)$ the set of oriented edges of $\T$.
Lift $\T$ to an ideal triangulation $\Tt$ of the universal cover $\Sft$ of $\Sf$.
 The set of vertices of $\Tt$ 
then identifies with the Farey set $\Farey{\Sf}$ of $\Sf$.
We will identify the oriented  edges $\e$ of $\Tt$ with the corresponding 
pairs $(\i,\j)$ of points in $\Farey{\Sf}$ (vertices of $\e$).
A {\em marked}  triangle of $\Tt$ is a triple
$(\i,\j,\k)$ of points in $\Farey{\Sf}$ that are the common vertices
of a triangle of $\Tt$.
\subsection{Cross ratio} 
\label{ss- algebraic cross ratio}
We use the following convention for cross ratios (following
Fock-Goncharov \cite{FoGoSPC}).
When $\V$ is a two dimensional vector space over a field $\KK$, 
the cross ratio of a four points $\aK_1,\aK_2,\aK_3,\aK_4$  
in the projective line $\PV$
is defined by
\begin{equation}
\label{eq- def cross ratio}  
\Bir(\aK_1,\aK_2,\aK_3,\aK_4)
=\frac{(\aK_1 -\aK_2 )(\aK_3 -\aK_4 )}
      {(\aK_1 -\aK_4 )(\aK_2 -\aK_3 )}
\end{equation}
in any affine chart $\PV\isomto \KK\cup\{\infty\}$,  
that is in order that
$\Bir(\infty,-1,0,\aK)=\aK$.
It is well-defined (in $\KK\cup\{\infty\}$) 
when the quadruple is \nondegeneratedQ{},
i.e. when either the numerator or the denominator is nonzero
(see section \ref{s- nondegeneratedQ}).

We now recall the natural symmetries.
For a permutation $\sigma$ in  $\Sym_4$,
we denote
$$(\sigma\cdot \Bir)(\aK_1, \aK_2, \aK_3,\aK_4) 
=\Bir(\aK_{\sigma(1)},
\aK_{\sigma(2)},\aK_{\sigma(3)},\aK_{\sigma(4)})\;.$$
Recall that 
$\sigma\cdot\Bir=\Bir$ 
when $\sigma$ is any the double transpositions,
that $\sigma\cdot\Bir=\Bir^{-1}$ 
when $\sigma$ is $(13)$, $(24)$, $(1234)$ or $(1432)$ ;
and that 
\label{ss- sym bir 1342}
$(234)\cdot\Bir=-(1+\Bir^{-1})$
and
\label{ss- sym bir 1423}
$(243)\cdot\Bir=-(1+\Bir)^{-1}$.
The cocycle identity is 
\begin{equation}
\label{eq- cross ratio - cocycle identity}
-\Bir(\aK_1, \aK_2, \aK_3,\aK_4)\Bir(\aK_1, \aK_4,\aK_3,\aK_5)=
\Bir(\aK_1, \aK_2, \aK_3,\aK_5)
\end{equation}

\subsection{Triple ratio of a triple of flags} 
\label{s- triple ratio}
We refer the reader to \cite[\S 9.4 p128]{FoGoIHES}.
Let $\F_\ii=(\p_\ii,\D_\ii)$, $\ii=1,2,3$, be 
a  triple of flags in $\PP=\PP(\KK^3)$.
The {\em triple  ratio} of the triple $(\F_1,\F_2,\F_3)$
is defined   by
\[\Tri(\F_1,\F_2,\F_3)
= \frac{\Dt_1(\pt_2)\Dt_2(\pt_3)\Dt_3(\pt_1)}{\Dt_1(\pt_3)\Dt_2(\pt_1)\Dt_3(\pt_2)}\] 
where $\pt_\ii$ is any vector in $\KK^3$ representing $\p_\ii$ and
$\Dt_\ii$ is any linear form in $(\KK^3)^*$ representing $\D_\ii$.
It is well defined (in $\KK\cup\{\infty\}$) when the triple $(\F_1,\F_2,\F_3)$
is \nondegeneratedTF{}, i.e. when either the numerator or denominator are
nonzero (see section \ref{s- nondegeneratedTF}).
Note that
$\Tri(\F_1,\F_2,\F_3) = \infty $ if and only if 
there exists $\ii$ such that $\p_\ii \in \D_{\ii+1}$ 
and that
$\Tri(\F_1,\F_2,\F_3) = 0$ if and only if 
there exists $\ii$ such that $\p_\ii \in \D_{\ii-1}$.
In particular, the three flags are pairwise opposite 
if and only if 
their triple ratio is not $0$ or $\infty$.
The triple ratio is invariant under cyclic permutation of the flags:
and reversing the order inverses the triple ratio:
\begin{align*}
\Tri( \F_2, \F_3, \F_1 )&=\Tri(\F_1,\F_2,\F_3)\\
\Tri( \F_3, \F_2, \F_1 )&=\Tri(\F_1,\F_2,\F_3)^{-1} 
\;.
\end{align*}
The triple ratio may be expressed as 
the following  cross ratio on the naturally induced
quadruples of lines at $\p_1$ (which is \nondegeneratedQ{}, see
section\ref{s- nondegeneratedTF})

\begin{equation}
\label{eq- triple ratio as a cross ratio}
\Tri(\F_1,\F_2,\F_3)=\Bir(\D_1, \p_1 \p_2, \p_1 \p_{23}, \p_1 \p_3)
\end{equation}
or on the line $\D_1$
\begin{equation*}
  \Tri(\F_1,\F_2,\F_3)=\Bir(\D_1\cap \D_2, \D_1\cap \D_{23}, \D_1\cap \D_3, \p_1)
\end{equation*}

\begin{figure}[h]
\includegraphics{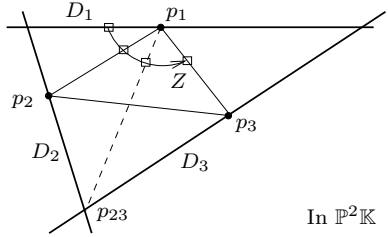}
\caption{The triple ratio $\Z=\Tri(\F_1,\F_2,\F_3)$ as a cross ratio.}
\end{figure}

Generic triples may be characterized by triple ratio:
$(\F_1,\F_2,\F_3)$ is generic 
if and only if 
$\Tri(\F_1,\F_2,\F_3)\neq \infty,0,-1$.
The triple ratio parametrize the generic triples of flags in the
projective plane, more precisely
for each $\aK\in\KKol$ there exists  a generic triple of flags
in $\PP$ with triple ratio $\aK$,
and $\PGL(\KK^3)$ acts $1$-transitively on the set of  generic triples of flags
of given triple ratio
(see also Lemma \ref{lemm- D_3 from triple ratio}).

\subsection{FG-invariants of a \transverse{} flag map}
\label{ss- invariants of Fdev}
Consider  a  {\em flag map}
$$\Fdev:\Farey{\Sf} \mapsto \MaxFlags(\PP)\;.$$
We denote by $\p_\i$ (resp. by $\D_\i$)
the point (resp. the line) of the flag $\F_\i=\Fdev(\i)$, 
for $\i\in\Farey{\Sf}$.
Let $\Tt$ an ideal triangulation  of $\Sft$.
We suppose that  $\Fdev$ and $\Tt$ are {\em \transverse}
that is 
that $\Fdev$ sends each triangle in $\Tt$ to a generic triple of flags.
We denote by $\p_\ij$ the point $\D_\i\cap \D_\j$
(resp. by $\D_\ij$ the line $\p_\i\p_\j$) (when defined).

To each  triangle $\tau$ of $\Tt$ 
with vertices $(\i, \j , \k)$ in $\Farey{\Sf}$, 
we  associate a {\em triangle invariant}:
the triple ratio 
\[\Z_\tau
=\Tri(\Fdev_\i, \Fdev_\j, \Fdev_\k)
= \Bir(\D_\i, \p_\i\p_\j, \p_\i\p_\jk, \p_\i\p_\k)
\]
of the triple of flags $\Fdev(\tau)$ 
(where $\i,\j,\k$ are  cyclically ordered 
accordingly to the orientation of the
surface).
It is well defined and  in $\KKol$ 
as  $\Fdev(\tau)$ is a generic triple of flags.
\label{ss- alg  edge invariants}
To each an oriented edge $\e=(\k,\i)$ in $\Tt$, 
we associate an {\em edge invariant}: the cross ratio
\[\Ep_\e
=\Bir
(\D_\i, \p_\i \p_\j, \p_\i \p_\k, \p_\i \p_\kl)
=\Bir
(\p_\k, \D_\k \cap \D_\l, \D_\k \cap \D_\i, 
\D_\k \cap \D_\ij)
\]
where $\i,\j,\k,\l$ in $\Farey{\Sf}$ are 
the vertices of  the two adjacent triangles
$\tau=(\i, \j , \k)$  and $\taup=(\k, \l , \i)$, 
cyclically ordered 
accordingly to the orientation of the surface 
(see figure \ref{fig- edge inv}).
Since $\Fdev(\tau)$ and $\Fdev(\taup)$ are generic, this is well defined and 
in $\KKo$. 

\begin{figure}[h]
\noindent{\hfill
\begin{minipage}{4cm}
\centering
\includegraphics{figures/adjacent_triangles.fig}\\
in $\Sft$
\end{minipage}
\hfill
\begin{minipage}{7cm}
\centering
\includegraphics[scale=0.8]{figures/PP2K_FG_quadruple_alter_edge_inv.fig}
\end{minipage}
\hfill }

  \caption{The invariant $\Ep_\e$ associated with an oriented edge
    $\e$.}
  \label{fig- edge inv}
\end{figure}

Note that the edge parameters are symmetric with respect to natural
duality,
as reversing the orientation of $\e$ (i.e. applying 
the half-turn $(\i\k)(\j\l)$) we get
\[\Ep_\eb
=\Bir
(\D_\k, \p_\k \p_{\l}, \p_\k \p_\i, \p_\k \p_\ij)\]
so exchanging the roles of points and lines 
correspond to exchange $\Ep_\e$ and $\Ep_\eb$.

Thus we have a well-defined  {\em FG-invariant} 
$$\FGp=((\Z_\tau)_{\tau}, (\Ep_\e)_{\e})$$ 
in $(\KKol)^{\Triangles(\Tt)} \times (\KKo)^{\orEdges(\Tt)}$
of the flag map $\Fdev$ {\em with respect to the triangulation $\Tt$}.

\subsection{Construction of flag maps from FG-parameters}
\label{s- construction of F}
We now show that 
FG-invariants $\FGp=((\Z_\tau)_\tau,(\Ep_\e)_\e)$ in 
$(\KKol)^{\Triangles(\Tt)} \times (\KKo)^{\orEdges(\Tt)}$ 
parametrize $\Tt$-\transverse{}  
flag maps $\Fdev:\Farey{\Sf}\to\MaxFlags(\PP)$ up to the action of
$\PGL(\KK^3)$.

Fix a base triangle $\tauo$ in the triangulation $\Tt$ with
(positively ordered) vertices $(\o_1,\o_2,\o_3)$ in $\Farey{\Sf}$.

\begin{prop}\cite{FoGoSPC}
\label{prop- existence unicite Fdev}
Let $\FGp=((\Z_\tau)_\tau,(\Ep_\e)_\e)$ be a {\em FG-parameter},
i.e. an element of
$(\KKol)^{\Triangles(\Tt)} \times (\KKo)^{\orEdges(\Tt)}$.
Fix a generic triple $\FFp=(\F_1,\F_2,\p_3)$, 
where $\F_1$, $\F_2$ are two flags in $\PP^2 \KK$ 
and $\p_3$ is a point in $\PP^2 \KK$.
There exists a unique map 
$\Fdev:\Farey{\Sf} \mapsto\MaxFlags(\PP)$,
\transverse{} to $\Tt$,
such that the FG-invariant 
of $\Fdev$ relatively to $\Tt$ is $\FGp$, 
and sending the points  $\o_1,\o_2$ to the flags $\F_1,\F_2$, and the
point $\o_3$ to  some flag through point $\p_3$.
\end{prop}
In order to normalize, 
we will denote by 
$\FdevZ$
the flag  map $\Fdev$ with FG-invariant $\FGp$ such that 
the triple $\FdevZ(\tauo)=(\F_1,\F_2,\F_3)$ is in {\em canonical
  form}, 
that is
 $\p_1=[1:0:0]$, $\p_2=[0:1:0]$,  $\D_1\cap \D_2=[0:0:1]$, $\p_3=[1:1:1]$ 
is the canonical projective frame.

\begin{proof}
  Since the dual graph of the
  triangulation $\Tt$ has no cycle (ie, is a tree), existence and unicity
  of $\Fdev$ comes from the following basic facts,
  by induction on adjacent triangles. 

  \begin{lemm}
    \label{lemm- D_3 from triple ratio}
    Let $\F_1=(\p_1,\D_1)$, $\F_2=(\p_2,\D_2)$ be two flags in
    $\PP$ 
    and $\p_3$ be a point in $\PP$.
    Suppose that $\F_1, \F_2$ and $\p_3$ are in generic position.
    Let $\aK\in\KKol$. 
    Then there exists a unique flag $\F_3=(\p_3,\D_3)$
    such that
    the  triple of flags $(\F_1,\F_2,\F_3)$ is generic 
    and 
    $\Tri(\F_1,\F_2,\F_3)=\aK$.
    \qed\end{lemm}

\begin{lemm}
  Let $(\F_1,\F_2,\F_3)$ be a generic triple of flag. For all
  $\Ep$, $\Ep'$ in $\KKo$ and $\Z'$ in $\KKol$, there exists a
  unique flag $\F_4$ such that
   $$\Ep=\Bir(\D_1, \p_1 \p_2, \p_1 \p_3, \p_1 (\D_3\cap\D_4) )$$
   $$\Ep'=\Bir(\D_3, \p_3 \p_4, \p_3 \p_1, \p_3 (\D_2\cap\D_1))$$
  and the triple of flags $(\F_1,\F_2,\F_3)$ is generic 
  and has triple ratio $\Z'$. 
\end{lemm}
\begin{proof}
Since $\F_1$,$\F_2$, and  $\F_3$ are in generic position, they define three
pairwise distinct points $\D_3 \cap \D_1$, $\D_3 \cap (\p_1\p_2)$, and
$\p_3$ on the line $\D_3$. 
So there exists a unique point
$\p$ on $\D_3$ such that     
$\Bir(\D_3 \cap \D_1, \D_3 \cap (\p_1\p_2),  \p_3, \p)=\Ep$.

Similarly, we have three pairwise distinct lines $\D_3$, $\p_3(\D_2\cap\D_1)$, $\p_3\p_1$
through point $\p_3$, hence there exists a unique line $\Delta$ through
$\p_3$ such that
$\Bir(\D_3, \Delta, \p_3\p_1, \p_3(\D_2\cap\D_1))=\Ep'$,
and $\p_1\notin\Delta$ as $\Ep'\neq\infty$.

Since $\Ep \neq 0,\infty$, we have $\p \neq \p_3$ and $\p\notin \D_1$, hence
we have  three pairwise distinct lines $\D_1,\p_1\p_3,\p_1\p$ at
$\p_1$, and there exists a unique line $\Delta'$ through $\p_1$
satisfying 
$\Bir(\D_1,\p_1\p_3,\p_1\p,\Delta')=\Z'$, 
and $\p_3\notin\Delta'$ as $\Z'\neq -1$.
We have $\Delta \neq \Delta'$ (else $\p_1\in\Delta$) 
so $\Delta$ and $\Delta'$
intersects in a unique point $\p_4$ 
with $\p_4\notin \D_1, \D_3$, and $\p_4\notin \p_1\p_3$.
Then $\p\neq \p_4$ 
(else $\p\in \Delta'$ and $\Delta'=\p_1\p$ and $\Z'=0$)
so we may define $\D_4=\p_4\p$, and then $\D_4\neq \D_1,\D_3$. 
We have $\p_3\notin \D_4$ as $\D_4\cap \D_3=\p\neq \p_3$. 
Since $\Delta'=\p_1\p_4$ is different from $\p_1\p$ (since $\Z'\neq 0$),
we have that $\p_1\notin \D_4$.
Since $\p=\D_4\cap\D_3$ is different from $\D_1\cap\D_3$,
we have that $\D_1\cap\D_3$ is not on $\D_4$.
Therefore the triple $(\F_1,\F_3,\F_4)$ is generic
and its triple ratio is
$\Bir(\D_1,\p_1\p_3,\p_1\p,\p_1\p_4)=\Z'$ as $\p_1\p_4=\Delta'$.
\end{proof}
\end{proof}

\label{s- construction of rhoZ}

\subsubsection{Equivariance and construction of representations}
We now suppose that $\Tt$ is the lift of an ideal triangulation $\T$ of
$\Sf$ and that $\FGpt$ is a FG-parameter on $\Tt$ invariant under
$\Ga=\pi_1(\Sf)$, i.e. lifting a FG-parameter $\FGp$ on $\T$. 
We denote $\FdevZ=\FdevZt$.
We now show that, 
since $\PGL(\KK^3)$ acts $1$-transitively on generic triples of
flags of given triple ratio, by rigidity of the construction,
we have an associated {\em holonomy} representation.

\begin{prop}
Let $\FGp=((\Z_\tau)_\tau,(\Ep_\e)_\e)$ in
$(\KKol)^{\Triangles(\T)} \times (\KKo)^{\orEdges(\T)}$,
and let $\Fdev=\FdevZ$.
There exists a unique  representation
$\rho:\Ga \to \PGL(\KK^3)$ such that
$\Fdev$ is $\rho$-equivariant, i.e. 
$\rho(\ga)\FdevZ(\i)=\Fdev(\ga\i)$
for all $\ga\in\Ga$, $\i\in\Farey{\Sf}$.
We will denote $\rho=\rhoZ$ and call it 
the representation with FG-parameter $\FGp$.
\end{prop}

In particular $\FdevZ(\i)$ is a flag fixed by $\rhoZ(\ga_i)$.
Note that different choices of $\FGp$ may lead to the same
representation $\rhoZ$.

\begin{proof}
Let $\ga\in\Ga$. The triples of
flags $\Fdev(\ga\tauo)$ and $\Fdev(\tauo)$ 
have same triple ratio $\Zt_{\ga\tauo}=\Zt_\tauo\neq -1$,
so there exists a unique $\g$ in $\PGL(\KK^3)$ such that 
$\g\Fdev(\tauo)=\Fdev(\ga\tauo)$. 
We set then $\rho(\ga)=\g$.
The maps $\rho(\ga)\circ\Fdev$ and $\Fdev\circ\ga$ from $\Farey{\Sf}$ to
$\MaxFlags(\PP)$ 
have same FG-invariant $\FGpt=\FGpt\circ\ga :
\Triangles(\Tt)\cup\orEdges(\Tt)\to\KK$ with respect to $\Tt$, 
and send the base triangle $\tauo$
to the same generic triple of flags,
hence they coincide 
by Proposition \ref{prop- existence unicite Fdev}.
The fact that $\rho$ is a morphism follows then from
$1$-transitivity on generic triples of flags, since:
$$\rho(\ga_1\ga_2)\Fdev(\tauo)
=\Fdev(\ga_1\ga_2\tauo)
=\rho(\ga_1)\Fdev(\ga_2\tauo)
=\rho(\ga_1)\rho(\ga_2)\Fdev(\tauo)\; .$$
\vskip -3.5ex
\end{proof}

\subsection{Other edge invariants and relation with 
\texorpdfstring{\cite{FoGoSPC}}{[FG]}}
\label{ss- relation with usual FG}
Note that our edge invariants $\Ep_\e$ differ sligthly from those of
\cite{FoGoSPC}. We here describe the relationship in detail.
We use the setting of section \ref{ss- invariants of Fdev}.

Let $\i,\j,\k,\l$ in $\Farey{\Sf}$ be
the vertices of two adjacent triangles
$\tau=(\i, \j , \k)$  and $\taup=(\k, \l , \i)$
with common edge $\e=(\k,\i)$. 
The associated invariants $X,Y,Z,W$ of \cite{FoGoSPC} are in our settings
$X   = \Z_\tau$, $Y   = \Z_\taup$, $Z   = \Z_\e$, and $W   = \Z_\eb$,
where $\Z_\e$ denotes   the following cross-ratio
\[\Z_\e
=\Bir
(\D_\i, \p_\i \p_{\j}, \p_\i \p_{\k}, \p_\i \p_\l)
\;.
\]

\label{ss- alg Dual edge invariants}
The edge invariant $\Z_\e$ is not symmetric under duality, yet
exchanging the roles of points and lines
provide another natural invariant 
$$\Zet_\e
=\Bir(\p_\i, \D_\i \cap \D_\j, \D_\i \cap \D_\k,
\D_\i \cap \D_\l)
\;.
$$
Our edge invariants $\Ep_\e$ are 
then easily related to the original $\Z_\e$ by (using the cocycle
identity):
\begin{equation}
\label{eq- usual FG from alter FG}
\begin{array}{ll}
\Z_\e   & = \Ep_\e(1+\Z_\taup)\\
\Zet_\e & = \Ep_\eb(1+\Z_\taup^{-1})
\end{array}
\;.
\end{equation}
In particular, when $\KK$ is an ordered field, 
then  if the triangle invariants are positive, 
our edge invariants are positive if and only if 
the usual edge invariants are positive.

Note that the relation linking usual FG-invariants 
of two adjacent triangles
\begin{equation}
\label{eq- relation between Z and Z*}
\Zet_\e
=\Z_\eb\frac{1}{1+\Z_\tau}(1+\Z_\taup^{-1})
\end{equation}
(compare \cite[2.5.3]{FoGoSPC})
 follows from \eqref{eq- usual FG from alter FG} 
and from  the autoduality
 of the $\Ep_\e$,
 since reversing the edge $\e$ we get 
\begin{align*}
\Ep_\e=\Z_\e(1+\Z_\taup)^{-1}=\Zet_\eb(1+\Z_\tau^{-1})^{-1}\\
\Ep_\eb=\Z_\eb(1+\Z_\tau)^{-1}=\Zet_\e(1+\Z_\taup^{-1})^{-1}
\;.
\end{align*}

\section{The \texorpdfstring{$\typeA_2$}{A\_2}-complex 
\texorpdfstring{$\CEsgZ$}{K} 
associated with
a  \leftshifting{} 
\texorpdfstring{$\gFGp$}{(z,s)}%
}
\label{s- construction A2-complexe CgZ}

\subsection{\texorpdfstring{$(\Aa,\W)$}{(A,W)}-complexes and 
\texorpdfstring{\tTsurfaces{}}{1/3-translation surfaces}} 
In this section, we introduce the notion of 
$\W$-translation surfaces, generalizing
translation and half-translation surfaces,
and  the more general notion of $(\Aa,\W)$-complexes.
Natural examples are subcomplexes 
of Euclidean buildings with model flat $(\Aa,\W)$.
We show that, like Euclidean buildings, these spaces are naturally
endowed with a $\Cc$-valued metric and  associated $\Cc$-distance
(where $\Cc$ is a standard fixed Weyl chamber in $\Aa$). 
\subsubsection{\Wsurfaces{}}
\label{sss- tTsurf}
Let $\Aa$ be a Euclidean vector plane and
let  $\W$ be a finite subgroup of isometries of $\Aa$.
A {\em \WTsurface} consists of  a compact surface $\WSurf$ 
possibly with boundary,
   a finite set  of interior points $\Sing\subset \WSurf$ ({\em singularities})
and a $(\Waff,\Aa)$-structure 
on $\WSurf-\Sing$
i.e. an atlas of charts $\chart_\indChart:U_\indChart\to \Aa$
with transition maps in $\Waff=\W\ltimes \Aa$.
This atlas induces in particular a flat metric on $\WSurf-\Sing$,
and we require that each singular point $\x\in \Sing$ has a
neighborhood $U$ such that $U-\{\x\}$ is isometric to a punctured cone.

For $\W=\{\id\}$ (resp. for $\W=\{\pm\id\}$)
 it corresponds to the classic notion 
of {\em  translation surface} 
(resp. of {\em  half-translation surface}) (see for example
\cite{Masur}, \cite{Yoccoz}). 

By analogy, we will call a {\em \tTsurface} a \WTsurface{} with $\W$ the 
subgroup of  rotations of angle in $\frac{2\pi}{3}\ZZ$.

\subsubsection{\texorpdfstring{$(\Aa,\W)$}{(A,W)}-complexes}
In this section,  $(\Aa,\W)$ is 
a finite reflection group of dimension two.
We recall that $\Waff$ is the subgroup of affine isomorphisms of
$\Aa$ with linear part in $\W$.

Intuitively speaking,
 a {\em $(\Aa,\W)$-complex} (or {\em $\W$-complex}, 
or {\em $\typeA_2$-complex} when $\W$ is of type $A_2$)
is a space $\CEs$ obtained by gluing polygons of $\Aa$
along boundary segments  by elements of $\Waff$.
We now give a precise definition of $(\Aa,\W)$-simplicial complexes 
following the definition of Euclidean simplicial complexes 
in \cite[I.7.2]{BrHa}.

\begin{defi}($(\Aa,\W)$- simplicial complex)
Let $\{ \SimplA^\indSimpl,\ \indSimpl\in \IndSimpl \}$ be a family of
affine simplices $\SimplA^\indSimpl \subset\Aa$. 
Let 
$\E=\sqcup_{\indSimpl\in \IndSimpl} 
\SimplA^\indSimpl \times \{\indSimpl\}$ 
denote their disjoint union.
Let $\simeq$ be an equivalence relation on $\E$  and let 
$\CEs=\E/\simeq$ denote the quotient space.
Let $\projCEs:\E\to\CEs$ denote the corresponding projection,
define $\projCEs_\indSimpl : \SimplA^\indSimpl \to \CEs$ by
$\projCEs_\indSimpl(\xA)=\projCEs(\xA,\indSimpl)$, and
denote by $\SimplCEs^\indSimpl\subset \CEs$ the image $\projCEs_\indSimpl(\SimplA^\indSimpl)$.

The space $\CEs$ is called a  {\em $(\Aa,\W)$-simplicial complex} if
\begin{enumerate}

\item for every $\indSimpl\in\IndSimpl$, 
the map $\projCEs_\indSimpl$  is injective.

\item If $\SimplCEs^\indSimpl\cap \SimplCEs^\indSimplp \neq\emptyset$,
  then 
there is an element $\w_{\indSimpl,\indSimplp}$ of  $\Waff$ such that
for all $\xA\in \SimplA^\indSimpl$ and $\xAp\in \SimplA^\indSimplp$
we have
$\projCEs(\xA,\indSimpl)= \projCEs(\xAp,\indSimplp)$
if and only if 
$\xAp=\w_{\indSimpl,\indSimplp}(\xA)$,
and
$\SimplA^{\indSimpl,\indSimplp}=
\SimplA^\indSimpl \cap \w_{\indSimpl,\indSimplp}^{-1}(\SimplA^\indSimplp)$   
is a face of $\SimplA^\indSimpl$.
\end{enumerate}
\end{defi}
In particular, $\CEs$ is 
a Euclidean simplicial complex of dimension $2$.
We will  suppose from now on that 
$\CEs$ is connected 
and that 
the set 
of isometry classes of simplices of $\CEs$ is finite.
We denote by $\deuc$ the associated
metric, which is
a complete geodesic length metric (see \cite[I.7]{BrHa}).
We denote by $\TangS_\xCE \CEs$ the geometric link  of $\CEs$ at a point $\xCE$,
which is a spherical $1$-dimensional complex (hence a metric graph) 
endowed with the angular length metric $\lengthangle$ 
(see \cite[I.7.15]{BrHa}). 
{From now on, we will  suppose that $\CEs$ has non positive curvature,
 that is for all points $\xCE\in\CEs$
each injective loop in the link $\TangS_\xCE \CEs$ has  length at least
$2\pi$.
}
If $\CEs$ is simply connected,  $(\CEs,\deuc)$ is then a $\CAT(0)$ metric space
(see Theorem I.5.4 and Lemma I.5.6 of \cite{BrHa}).

\subsubsection{\texorpdfstring{$\Cc$}{C}-distance}
Germs of non trivial segments at a point $\xCE\in\CEs$
have a well-defined projection in $\bord\Cb$ (their ({\em type
  (of direction)}).
In particular the notions of regular and singular
directions still make sense in $\TangS_\xCE\CEs$.
Note that a geodesic segment is not necessarily of constant
type of direction, unlike in Euclidean buildings.
The {\em $\Cc$-length} $\ellC(I)$ of a segment $I=[\xCE,\yCE]$ 
contained  in a simplex $\SimplCEs^\indSimpl$ of
$\CEs$ is defined as the $\Cc$-length in $\Aa$ of
the segment $\projCEs_\indSimpl^{-1}(I)$ 
(note that it does not depend on the choice of $\indSimpl$, 
because the transition maps are in $\Waff$). 
The {\em $\Cc$-length} of 
a piecewise affine path
$\apath: [0,\sRR] \to \CEs$ in $\Es$ 
 is defined by
$\ellC(\apath)=\sum_{\indPath}\ellC([\xCE_\indPath,\xCE_{\indPath+1}])$
for one (any) subdivision 
  $\tRR_0=0 <\tRR_1< \cdots < \tRR_\IndPath=\sRR$ of
$[0,\sRR]$ such that the restriction of $\apath$ to
$[\tRR_\indPath,\tRR_{\indPath+1}]$ 
is an affine segment $[\xCE_\indPath,\xCE_{\indPath+1}]$
conatined in some simplex of $\CEs$.
It is a vector in the closed Weyl chamber $\Cb$.
It is invariant under subdivisions of the simplicial complex $\CEs$.
\label{s- def C-distance in A2-complex}%
When $\CEs$ is simply connected (hence $\CAT(0)$),
we define the {\em $\Cc$-distance} 
from $\xCE$ to $\yCE$ in $\CEs$ as the 
$\Cc$-length $\Cd(\xCE,\yCE)=\ellC(\apath)$
of the geodesic $\apath$ from $\xCE$ to $\yCE$.
We then have
$$\Cd(\yCE,\xCE)=\Cd(\xCE,\yCE)^ \opp$$
and
$$\norm{\Cd(\xCE,\yCE)}\leq \deuc(\xCE,\yCE)$$

\begin{rema}
\label{rem- Cd do not refine deuc in C-complexes}
Note that, unlike in Euclidean
buildings, the inequality may well be strict. 
Thus the $\Cc$-distance is no longer a refinement of the
distance $\deuc$.
A basic   example is
given by  non convex subsets  $\CEs$ of $\Aa$.

\end{rema}

\subsubsection{\texorpdfstring{$\Cc$}{C}-Length of automorphisms}
An {\em automorphism} $\g$ of $\CEs$ is a bijection 
preserving $\Cd$.
In particular it preserves the distance $\deuc$.
The  {\em $\Cc$-length} of  $\g$ 
of $\CEs$ translating some geodesic $\geod$ is defined by
$$\ellC(\g)=\Cd(\xCE,\g\xCE)$$
for one (any) $\xCE$  on $\geod$
(it does not depend on the choice of $\geod$ as 
two different translated geodesics bound a flat strip, 
and may be  developped as parallel geodesics in $\Aa$).
Note that, in contrast to the case of Euclidean buildings,
 the $\Cc$-length do no longer refine the {\em Euclidean length}
$$\elleuc (\g)=\{=\{\deuc(\xE,\g\xE),\ \xE\in\Es\}\;.$$
We have $$\norm{ \ellC (\g)} \leq \elleuc (\g)$$
but  the inequality  may be strict.

\subsection{Abstract geometric FG-parameters and \LeftShift{}}

Let $\T$ be an ideal triangulation of a punctured surface $\Sf$, with
set of triangles $\TrianglesT$ and set of oriented edges $\orEdgesT$.
Consider an {\em geometric FG-parameter}  on $\T$, 
i.e. an element $\gFGp=((\gZ_\tau)_\tau,(\gEp_\e)_\e)$ 
in $\RR^\TrianglesT \times\RR^\orEdgesT$.
We now introduce the class of abstract geometric FG-parameters $\gFGp$
to which we are going to associate an $\typeA_2$-complex $\CEsgZ$.
\label{ss- hyp (LeftShift) sur gZ}
Let  $\e$ be an oriented edge in $\T$, 
with left and right adjacent triangles $\tau$ and $\taup$.

\begin{defi}
\label{def- hyp (LeftShift) sur gZ}
We say that  
$\gFGp$ is {\em \leftshiftingedge{}} $\e$ 
if we have
$\gEp_\e     > - \gZmoins_\tau,- \gZplus_\taup $
and $\gEp_\eb     > - \gZplus_\tau,-\gZmoins_\taup$.
\end{defi}
\begin{rema}
\label{rema- leftshiftingedge e : three cases}

Then 
we are in  one (and only one) of the three following cases:
\begin{enumerate}
\item $\gEp_\e>0$ and $\gEp_\eb >0$

\item 
$\gZ_\tau<0$, $\gZ_\taup>0$,
$\gEp_\eb>0$ and $ \gZ_\tau, -\gZ_\taup <\gEp_\e \leq 0 $ ;
\item
 $ \gZ_\tau>0$, $\gZ_\taup < 0$, 
$\gEp_\e >0$ and $ \gZ_\taup, -\gZ_\tau <\gEp_\eb \leq 0 $ ;
\end{enumerate}

\end{rema}

We  say that $\gFGp$ is {\em \leftshifting} if it is
\leftshiftingedge{} $\e$ for all $\e$, and we denote this property by 
(\HypLeftShift).
Note that the subset $\ConeL$ of \leftshifting{} $\gFGp$ 
in $\RR^\TrianglesT \times \RR^\orEdgesT$
is an non empty open cone 
(in fact a finite union of open convex polyhedral cones).

\subsection{Construction of the
  \texorpdfstring{$\typeA_2$}{A\_2}-complex 
\texorpdfstring{$\CEsgZ$}{K} 
associated with 
\texorpdfstring{$\gFGp$}{(z,s)}}
\label{ss- construction of the A2-complex from gZ}
Consider
 a \leftshifting{} geometric FG-parameter
$\gFGp=((\gZ_\tau)_\tau,(\gEp_\e)_\e)$ 
in $\RR^\TrianglesT \times \RR^\orEdgesT$
(see Definition \ref{def- hyp (LeftShift) sur gZ}).
Lift $\gFGp$ on the universal cover $\Tt$ in a $\Ga$-invariant
\leftshifting{} geometric FG-parameter, again denoted by $\gFGp$.

For each marked 
triangle $\tau=(\i,\j,\k)$ of the triangulation $\Tt$ 
let $\cellAgZ^\tau\subset \Aa$ be the singular equilateral triangle
 with vertices $\vstAi^\tau=0$, 
$\vstAj^\tau
=(-\gZmoins_\tau,-\gZplus_\tau)$ and
$\vstAk^\tau
=(-\gZplus_\tau,-\gZmoins_\tau)$ (in simple roots coordinates). 
Note that $\cellAgZ^\tau$ lies in the chamber $-\Cb$.
\begin{figure}[h]
  \includegraphics{figures/Aa_triangles_gZtau.fig}
  \caption{Singular triangle $\cellAgZ^\tau$ in $\Aa$.}
\end{figure}

For each oriented edge  $\e=(\k,\i)$ of $\Tt$
let $\cellAgZ^\e\subset \Aa$ be 
either,  when $\gEp_\e,\gEp_\eb\geq 0$,
 the closed segment from $0$ to the point $(\gEp_\eb,\gEp_\e)$,
or,
when $\gEp_\e<0$ or $\gEp_\eb <0$,
the parallelogram 
given (in simple roots coordinates)
by $\cellAgZ^\e= [0, \gEp_\eb]\times [0, \gEp_\e]$ (intersection of two
Weyl chambers of opposite direction).

We now describe formally how $\CEstgZ$ is constructed,
 gluing the polygons $\cellAgZ^\m$.
We define $\CEstgZ= \sqcup_{\m\in \orMPt} \cellAgZ^\m \times
\{\m\} / \sim$, where $\orMPt=\mkTriangles(\Tt)\cup\orEdges(\Tt)$ 
and $\sim$ is the equivalence relation generated by
the following identifications:
For every oriented edge $\e=(\k,\i)$ of $\Tt$,
with positively oriented adjacent triangle $\tau=(\i,\j,\k)$, 
Remark \ref{rema- leftshiftingedge e : three cases}
implies that 
the convex polygons $\cellAgZ^\tau$ and $\cellAgZ^\e$ intersects on the
subsegment (maybe reduced to a point)
$[\bEdgeAki^\tau,\vstAi^\tau]$ of $[\vstAk^\tau,\vstAi^\tau]$,  
with $\bEdgeAki^\tau=(\min(0,\gEp_\eb), \min(0,\gEp_\e))$.
We then  glue $\cellAgZ^\tau \times \{\tau\}$ along $\cellAgZ^\e \times \{\e\}$ 
along this segment
(i.e. by $(\vA,\tau)\sim (\vA,\e)$ for  $\vA\in \cellAgZ^\tau \cap \cellAgZ^\e$).
 
 \begin{figure}[h]
  \includegraphics[scale=0.8]{figures/gluings_marked_tau_e.fig}
  \caption{Gluings.}
  \label{fig- gluings}
\end{figure}

If $\taup$ is a permutation of 
a marked triangle $\tau=(\i,\j,\k)$ in $\Tt$,
 we identify $\cellAgZ^{\taup} \times \{\taup\}$ with $\cellAgZ^\tau \times \{\tau\}$ by
the unique affine isomorphism $\w_{\tau,\taup}: \cellAgZ^{\tau}\isomto
\cellAgZ^{\taup}$ sending $\vstAs^\tau$ to $\vstAs^\taup$ for $\s=\i,\j,\k$,
which is in $\Waff$.
If $\eb=(\i,k)$ is the opposite edge of
$\e=(\k,\i)$, then 
we identify $\cellAgZ^\eb\times \{\eb\}$ with $\cellAgZ^\e\times \{\e\}$ by
the unique affine isomorphism 
$\w_{\e,\eb}:\cellAgZ^\e\isomto \cellAgZ^\eb$
with linear part $\wopp\in\W$ (which sends $0$ to $(\gEp_\e,\gEp_\eb)$).
We denote by $\chartCEs : \sqcup_\m\cellAgZ^\m \to \CEstgZ$ the canonical
projection.
We denote by  $\cellCEstgZ^\m$ the image 
$\chartCEs(\cellAgZ^\m)$ in $\CEstgZ$.
The {\em canonical charts} are the restrictions
 $\chartCEs_\m :\cellAgZ^\m \to \cellCEstgZ^\m$  of
$\chartCEs$ to $\cellAgZ^\m$. 
We thus  obtain a two dimensional $(\Aa,\W)$-complex $\CEstgZ$ 
endowed with a free and cocompact action of $\Ga$ by automorphisms. 
We denote by $\Cd$ the natural $\Cc$-valued distance on $\CEstgZ$
(see Section \ref{s- def C-distance in A2-complex}).
The quotient $\CEsgZ$  of $\CEstgZ$ under $\Ga$ is the
{\em $\typeA_2$-complex associated with FG-parameter $\gFGp$} on the
triangulation $\T$. It is a finite $2$-dimensional complex 
homotopy equivalent to $\Sf$.
We denote by  $\cellCEsgZ^\m$ the image of $\cellCEstgZ^\m$ in $\CEsgZ$.

We remember the special points in the above construction for later
use:
For each marked triangle $\tau=(\i,\j,\k)$ of $\Tt$
we denote by $\vstCEi(\tau)$ and $\bEdgeCEki$  
the points of $\CEstgZ$ given by
$\vstAi^\tau$ and $\bEdgeAki^\tau$.
If $\gZ_\tau=0$, then the corresponding singular triangle
$\CEstgZ^ \tau$ is reduced to a point $\vstCE_\tau$, and $\CEstgZ$ is
locally the union of three non trivial edges at $\vstCE_\tau$.

Note that the $\typeA_2$-complex $\CEstgZ$ is naturally {\em
  oriented},  in the sense
that the space of directions at each point  $\TangS_\x\CEstgZ$ 
and the boundary at infinity $\bordinf\CEstgZ$  
inherit a natural cyclic order from the orientation 
of the surface $\Sf$. 

\subsection{Particular cases: from trees to surfaces}

\subsubsection{Tree}
A case of special interest  is  when 
all singular flat triangles $\cellCEsgZ^\tau$
are reduced to a point.
The corresponding condition on
$\gFGp$ in $\RR^\TrianglesT \times\RR^\orEdgesT$ 
is
\[
(\HypTree) 
\left\{
\begin{array}{l}
  \gZ_\tau = 0 \mbox{ for all  triangles } \tau \mbox{  of } \T\\ 
  \gEp_\e > 0  \mbox{ for all oriented edges }\e \mbox{  of } \T
\end{array}
\right.
\]
%
Then $\CEstgZ$ is a $3$-valent ribbon  tree
isomorphic to the dual tree of the triangulation $\Tt$, 
and $\CEsgZ$ is a
graph isomorphic to the dual graph of the triangulation $\T$.
Both are endowed with a
$\typeA_2$-structure or {\em $\Cc$-metric},
i.e. a $\Cc$-valued function $\e\mapsto \ellC(\e)$ on oriented edges
 satisfying
$\ellC(\eb)=\ellC(\e)^\opp$.

\subsubsection{Tree of triangles}
Another - more general - particular case of interest is  when 
is  when all the $\cellCEsgZ^\e$ are segments. 
Then $\CEstgZ$ is a ``tree
of triangles'', obtained from  the dual tree of the triangulation
by replacing vertices by triangles.
The corresponding condition on the \leftshifting{} FG-parameter 
$\gFGp$ in $\RR^\TrianglesT \times\RR^\orEdgesT$
is
\begin{center}
(\HypTreeOfTriangles) $\gEp_\e \geq 0$ for all oriented edge $\e$   
\end{center}

Note that 
(\HypLeftShift) is automatic if $\gEp_\e > 0$ for all $\e$,
and that 
(\HypTreeOfTriangles) implies (\HypEdgeSep).

\subsubsection{Surface}
At the other end of the spectrum,
another particular case of special interest 
is  when %
\begin{center}
(\HypSf) $\gEp_\eb<0$ or $\gEp_\e<0$ for all oriented edge $\e$   
\end{center}
Then $\CEstgZ$ is a surface, homeomorphic to $\Sft$ 
(the thickening of the ribbon tree dual to
the triangulation), hence $\CEsgZ$ is a \tT{} surface homeomorphic to
 $\Sf$ (see \S \ref{sss- tTsurf}), 
with piecewise singular geodesic boundary and no interior
singularities.

\begin{rema}
The subset 
of \leftshifting{} $\gFGp$
in $\RR^\TrianglesT \times \RR^\orEdgesT$
satisfying (\HypSf) is not empty  if and only if 
the triangulation $\T$  is $2$-colourable 
(since  $\gZ_\tau$ and $\gZ_\taup$ have then  opposite sign for adjacent
triangles $\tau$ and $\taup$). It is a finite union of non empty 
open convex polyhedral cones, one for each $2$-coloration of $\T$
(choice of prescribed signs for the triangle parameters).
\end{rema}
\section{Nice invariant complexes  for actions on buildings}
\label{s- actions on buildings}

\subsection{Geometric FG-invariants for actions on buildings}
\label{ss- geom FG-inv}

In this section,
we introduce an analog of FG-invariants
for actions on Euclidean buildings $\Es$ of type $\typeA_2$.
These invariants take values in $\RR$ and are defined 
by geometric cross ratios in the projective plane
at infinity of $\Es$.
In the algebraic case (i.e. when $\Es$ is the Euclidean building of
$\PGL(\KK^3)$ for some ultrametric field $\KK$), 
these {\em geometric  FG-invariants} 
are obtained from
the $\KK$-valued usual FG-invariants 
 by  taking logarithms of absolute values.
Note that the geometric invariants are substantially weaker than the
algebraic invariants (since we take absolute values). 
The principal advantage is that 
geometric FG-invariants still make sense when the building $\Es$ is exotic
(non algebraic), whereas the usual FG-invariants are not defined
anymore.
We recall that 
the set $\bordF\Es$ of chambers at infinity (Furstenberg boundary) 
of the Euclidean building $\Es$ is identified with the set
$\MaxFlags(\PP)$ of flags in the projective plane $\PP$ at infinity of
$\Es$, and that $\geombir$ denotes the  $\RR$-valued cross
ratio on $\PP$ induced by $\Es$ 
(see Section \ref{s- geom cross ratio in building}).

\medskip

Let $\Fdev:\Farey{\Sf} \to \bordF\Es$
be a {\em flag map}.
We denote by  $\p_\i$ (resp. by $\D_\i$)
the  point (resp. the line) 
of the flag $\Fdev_\i=\Fdev(\i)$ 
for every $\i \in \Farey{\Sf}$.
We suppose that the map $\Fdev$ is equivariant 
under an action $\rho$ of $\Ga=\pi_1(\Sf)$ on $\Es$.
Let  $\T$ be  an ideal triangulation  of $\Sf$, 
and $\Tt$ be the lift  of $\T$ to $\Sft$.  

We suppose that $\Fdev$ is {\em \transverse{} to $\T$}
i.e. sends each triangle of $\Tt$ on a generic triple of ideal chambers. 
We recall that we denote by $\p_\ij$ the point $\D_\i\cap \D_\j$
(resp. by $\D_\ij$ the line $\p_\i\p_\j$) (when defined).

\subsubsection{Triangle invariants}
To each marked triangle $\tau=(\i,\j,\k)$ of the triangulation $\Tt$ 
we associate the {\em geometric triple ratio} (see \cite{ParTriples})
of the generic triple of chambers 
$\Fdev(\tau)=(\F_\i,\F_\j, \F_\k)$, which is 
the following triple of  geometric cross ratios, taking values in $\RR$, 
obtained from permutations of the four lines 
$\D_\i$, $\p_\i\p_\j$, $\p_\i\p_\jk$, $\p_\i\p_\k$ in $\PP$
(cyclically permuting the three last ones)
(see \cite{ParTriples} for details)
\[
\left\{
\begin{array}{rcl}
\gZ_\tau=\geomtrio(\F_\i,\F_\j, \F_\k) 
&:=& \geombir(\D_\i, \p_\i\p_\j, \p_\i\p_\jk, \p_\i\p_\k)\\
\gZprime_\tau=\geomtrip(\F_\i,\F_\j, \F_\k)
&:=& \geombir(\D_\i, \p_\i\p_\k, \p_\i\p_\j, \p_\i\p_\jk)\\
\gZpprime_\tau=\geomtripp(\F_\i,\F_\j, \F_\k)
&:=& \geombir(\D_\i, \p_\i\p_\jk, \p_\i\p_\k, \p_\i\p_\j)
\end{array}
\right.
\;.\] 
We recall from \cite{ParTriples} the following basic properties. 
Each of
$\gZ_\tau$, $\gZprime_\tau$ and $\gZpprime_\tau$ 
is invariant under cyclic permutation of $\tau$, 
and reversing the order gives 
$\gZ_\taub=-\gZ_\tau$, $\gZprime_\taub=-\gZpprime_\tau$
(denoting $\taub=(\k,\j,\i)$).
We have  $\gZ_\tau+\gZprime_\tau+\gZpprime_\tau=0$, 
and moreover the triple $(\gZ_\tau,\gZprime_\tau,\gZpprime_\tau)$  
is of the form
 $(0,\gZ, -\gZ)$, $(-\gZ,0,\gZ)$ or $(\gZ,-\gZ,0)$ with $\gZ\geq 0$.
Note that,   if $\gZprime_\tau \leq 0$, then  $\gZprime_\tau=-\gZmoins_\tau$ 
and  $\gZpprime_\tau=\gZplus_\tau$. 

(by the properties of the geometric cross ratio 
under $3$-cyclic permutation, see Prop. \ref{prop- geombir}).

In the algebraic case,
in terms of the algebraic triangle FG-invariant (triple ratio)
$\Z_\tau=\Tri(\F_\i,\F_\j, \F_\k)$ in $\KKol$, we have,
by the symmetries of the usual algebraic cross ratio under
$3$-cyclic permutations,
\begin{equation}
  \begin{array}{rcl}
    \gZ_\tau &=&\log\abs{\Z_\tau}\\
    \gZprime_\tau &=&\log\abs{({1+\Z_\tau})^{-1}}\\
    \gZpprime_\tau &=& \log\abs{1+\Z_\tau^{-1}}
    \;.
  \end{array}
\end{equation}

\subsubsection{Edge invariants}
To each oriented edge $\e=(\k,\i)$ between 
 two   adjacent triangles
$\tau=(\i,\j,\k)$ and $\taup=(\k,\l,\i)$, 
where $\i,\j,\k,\l$ are cyclically ordered
accordingly to orientation of the surface,
we  associate the triple of geometric cross ratios
 at the point $\p_\i$ in $\PP$ 
associated with the four lines 
$\D_\i, \p_\i \p_\j, \p_\i \p_\k, \p_\i \p_\kl$ by cyclic permutation
of the three last ones:
\[
\left\{
\begin{array}{rcl}
\gEp_\e
&=& \geombir(\D_\i, \p_\i \p_\j, \p_\i \p_\k, \p_\i \p_\kl)\\
\gEprime_\e
&=& \geombir(\D_\i, \p_\i \p_\kl, \p_\i \p_\j, \p_\i \p_\k)\\
\gEpprime_\e
&=& \geombir(\D_\i, \p_\i \p_\k, \p_\i \p_\kl, \p_\i \p_\j)
\end{array}
\right.
\;.\]
As for triangle invariants, we have
$\gEp_\e+\gEprime_\e+\gEpprime_\e=0$
and moreover the triple $(\gEp_\tau,\gEprime_\tau,\gEpprime_\tau)$  
is in  $\RR^+(0,1,-1)$, $\RR^+(-1,0,1)$ or $\RR^+(1,-1,0)$.

In the algebraic case,
the link with the algebraic edge invariants 
$\Ep_\e\in\KKo$ (defined in \S \ref{ss- alg edge invariants}) is:
\begin{equation}
  \begin{array}{rcl}
    \gEp_\e &=&\log\abs{\Ep_\e}\\
    \gEprime_\e &=&\log\abs{({1+\Ep_\e})^{-1}}\\
    \gEpprime_\e &=& \log\abs{1+\Ep_\e^{-1}}
\;.  \end{array}
\end{equation}

\begin{figure}[h]
\noindent{\hfill
\begin{minipage}[b]{4cm}
\centering
\includegraphics{figures/adjacent_triangles.fig}\\
in $\Sft$
\end{minipage}
\hfill
\begin{minipage}[b]{7cm}
\centering
\includegraphics[scale=.7]{figures/PP_FG_quadruple_alter_edge_inv.fig}\\
in $\PP$
\end{minipage}
\hfill }

  \caption{The edge invariant $\gEp_\e$ associated with an oriented edge
    $\e$.}
  \label{fig- geom edge inv}
\end{figure}

As $\Fdev$ is $\rho$-equivariant, 
the triangle and edge invariants are invariant under the action
of $\Ga$ on $\Tt$, hence induce 
 well-defined invariants associated to triangles and oriented edges of
 $\T$, we will call the  {\em geometric FG-invariants of $\Fdev$ relatively to $\T$}.

\subsection{Main result}
We refer the reader to 
Sections \ref{ss- hyp (LeftShift) sur gZ}
and  \ref{ss- construction of the A2-complex from gZ}
for the notion of \leftshifting{} (\HypLeftShift) 
geometric FG-parameter $\gFGp$,
and the associated $\typeA_2$-complex $\CEsgZ$.

\begin{theo}
\label{theo- local Cgeod surface}
Let $\rho$ be an action of $\Ga=\pi_1(\Sf)$ on $\Es$, 
and $\Fdev : \Farey{\Sf} \to \bordF\Es$ 
be a $\rho$-equivariant map.
Let $\T$ be an ideal triangulation of $\Sf$.
Suppose that $\Fdev$ is transverse to $\T$.
Let $\gZ_\tau$,
$\gZprime_\tau$, 
$\gZpprime_\tau$, 
with $\tau\in \TrianglesT$, 
and
$\gEp_\e$, 
$\gEprime_\e$, 
$\gEpprime_\e$, 
with $\e$  in $\orEdgesT$,
 be the geometric FG-invariants of $\Fdev$ relatively to $\T$.
Suppose that
\begin{mydescription}
\item[(\HypFlatTriangles)] 
for each triangle $\tau$ in  $\T$, 
we have $\gZprime_\tau \leq 0$,
\item[(\HypLeftShift)]  $\gFGp$ is \leftshifting{}.
\end{mydescription}
Let $\CEsgZ$ be the $\typeA_2$-complex 
of FG-parameter $\gFGp=((\gZ_\tau)_\tau, (\gEp_\e)_\e)$.
Then
there exists a $\rho$-equivariant map 
$\Psi: \CEstgZ\to \Es$, locally preserving  the $\Cc$-distance $\Cd$.
\end{theo}

\begin{theo}
\label{theo- Cgeod surface}
\label{theo- global Cgeod surface}
Under the hypotheses and notations of Theorem \ref{theo- local Cgeod surface},
suppose furthermore that 
\begin{mydescription}
\item[(\HypH)] for each oriented edge $\e$  in  $\T$, 
we have 
$\gEprime_\e
\leq 0$ ;

\item[(\HypEdgeSep)] for each triangle $\tau$ in  $\T$ 
and every pair of edges $\ei$,
$\eii$ of $\tau$ 
 (oriented after $\tau$), we have 
$-\gEp_\ei-\gEp_\eii <\gZmoins_\tau$%
and $-\gEp_\eib-\gEp_\eiib <\gZplus_\tau$.
\end{mydescription}
Then the  map $\Psi: \CEstgZ\to \Es$
 preserves globally the $\Cc$-distance $\Cd$.
In particular
\begin{enumerate}
\item for all $\ga\in\Ga$
  \[ \ellC(\rho(\ga))=\ellC(\ga,\CEsgZ) \]
  (in particular the length spectrum of $\rho$ depends only on
  $\gFGp$)
and for usual lengths
  \[ \elleuc(\rho(\ga))=\norm{(\ellC(\ga,\CEsgZ)}, \]
\[ \ellH(\rho(\ga))=\normH(\ellC(\ga,\CEsgZ))  
\;. \]

\item The map $\Psi$ is bilipschitz. 
  The action
$\rho$ is undistorted, 
  faithfull and proper (hence discrete).
%
  %
  %
\end{enumerate}

\end{theo}
Note that in general  we do not have 
$\elleuc(\rho(\ga))=\elleuc(\ga,\CEsgZ)$.

\begin{rema}
  \label{rema- hyp (Edgeseparating)}
  Hypothesis (\HypEdgeSep) means geometrically that, 
  in the $\typeA_2$-complex $\CEsgZ$,
  a singular segment
  entering a triangle $\cellCEsgZ^\tau$ from one adjacent edge cell
  $\cellCEsgZ^\e$ does not extend outside $\cellCEsgZ^\tau$ 
  (see left side of the figure \ref{fig- edgeseparating}).
  We will then say that $\gFGp$ is {\em \edgeseparating{}}.

  \begin{figure}[h]
    \includegraphics[scale=0.9]{figures/CEsgZ_edgeseparating_tau_pos.fig}
    \hskip 5em
    \includegraphics[scale=0.9]{figures/CEsgZ_nonedgeseparating.fig}
    \caption{\Edgeseparating{} (on the left) vs \Nonedgeseparating{}
      (on the right)
      (in the case where $\gZ_\tau    >0$)} 
    \label{fig- edgeseparating}
  \end{figure}

\end{rema}
\subsection{Proof of Theorem \ref{theo- local Cgeod surface}}
Let $\rho:\Ga \to \Aut(\Es)$ be an action, and
$\Fdev : \Farey{\Sf} \to \bordF\Es$ be a $\rho$-equivariant map.
Suppose that $\Fdev$ is transverse to the  
ideal triangulation $\T$ of $\Sf$.
For each pair $(\i,\j)$ of distinct points in the Farey set
$\Farey{\Sf}$ of $\Sf$, we denote by $\App_\ij$ the flat in $\Es$ joining
$\Fdev_\i$ and $\Fdev_\j$.
Suppose first that 
(\HypFlatTriangles) holds, that is: 
$\gZprime_\tau
\leq 0$
for each triangle $\tau$ in  $\T$.

Let $\tau=(\i,\j,\k)$ be a triangle in $\Tt$.
Since $\gZprime_\tau=\geomtrip(\F_\i,\F_\j,\F_\k)
\leq 0$, 
we have the following properties for the triple  $(\F_\i,\F_\j,\F_\k)$
which are proved 
in \cite{ParTriples}
(Theorem 2), and depicted in Figure \ref{fig- the flat singular triangle}. 

\begin{theo}
  \label{theo- triangle sing type triangle - local} 
  The intersection of the two flats $\App(\p_\i,\p_j,\p_k)$ and
  $\App(\D_\i,\D_j,\D_k)$ is a flat singular triangle $\Delta=\Delta_\tau$
  with vertices 
  $\vstEi=\vstEi(\tau)$, $\vstEj=\vstEj(\tau)$ and $\vstEk=\vstEk(\tau)$
  such that:
  \begin{enumerate}
  \item
    %
    The Weyl chamber from $\vstEi$ to $\F_\i$ is $\Aij\cap\Aik$ ;

  \item In any marked flat $\f_\ij:\Aa\to \Aij$ sending
    $\bord\Cc$ to $\F_\j$, we  have  in simple roots coordinates
    $$\ovra{\vstEi\vstEj}=(\gZplus_\tau,\gZmoins_\tau)\; ;$$

  \item When $\Delta$ is not reduced to a point (i.e. when $\gZ_\tau
    \neq 0$),
    then $\Delta$ and $\F_\i$ define
    opposite chambers $\TangS_{\vstEi} \Delta$ and $\TangS_{\vstEi} \F_\i$ at $\vstEi$.

  \end{enumerate}
\end{theo}

\begin{figure}[h]
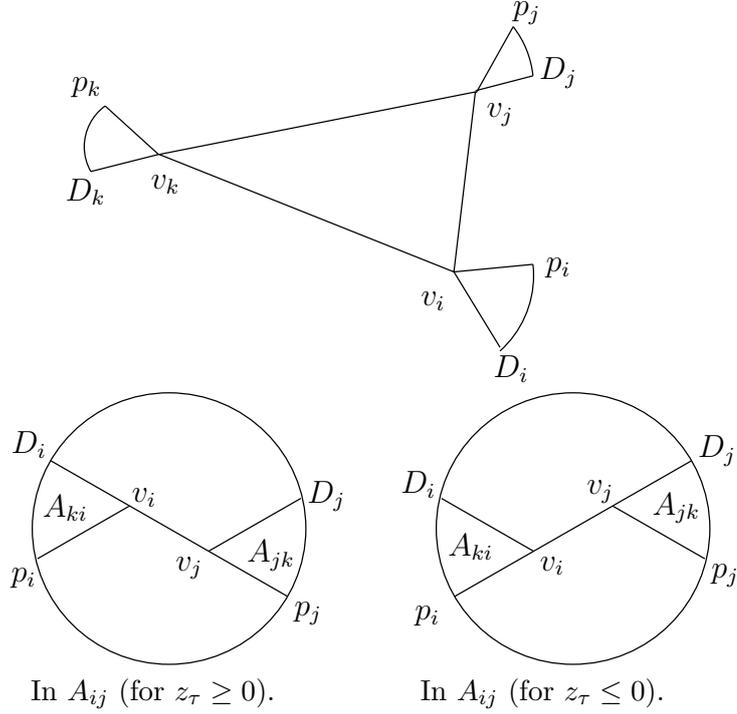

\includegraphics{figures/Es_triple_type_triangle_simplif.fig}

\hfill
\hfill
  \begin{minipage}[b]{0.3\linewidth}
\centering
\includegraphics{figures/Es_Aij_triangle_pos.fig}

{In $\Aij$ (for $\gZ_\tau \geq 0$).}
  \end{minipage}
\hfill
  \begin{minipage}[b]{0.3\linewidth}
\centering
\includegraphics{figures/Es_Aij_triangle_neg.fig}

{In $\Aij$ (for $\gZ_\tau \leq 0$).}
  \end{minipage}
\hfill
\hfill

\caption{The flat singular triangle
  $\Delta_\tau=(\vstE_\i,\vstE_\j,\vstE_k)$ associated with $\tau$.}
\label{fig- the flat singular triangle}
\end{figure}

\medskip

We now study the behaviour of two adjacent triangles, in particular
how edge invariants  measure  the shift between 
$\Delta_\tau$, $\Delta_\taup$
along the common flat.
Let $\tau=(\i,\j,\k)$ and $\taup=(\k,\l,\i)$, 
be a pair of adjacent triangles in $\Tt$
(where $(\i,\j,\k,\l)$ are positively ordered), and denote by $\e$ the
common edge $(\k,\i)$. 
Denote by $\vstEi=\vstEi(\tau)$, $\vstEk=\vstEk(\tau)$, 
$\vstEpi=\vstEi(\taup)$ and $\vstEpk=\vstEk(\taup)$
the particular points 
in the flat $\App_\ki$ joining $\F_\k$ to $\F_\i$
defined by each of the two adjacent triples.
Let  $\f_\e:\Aa\to \App_\ki$ be a marking of the flat $\App_\ki$ 
sending $\bord\Cc$ to $\Fi$. 
By Theorem \ref{theo- triangle sing type triangle - local}
and the invariance by cyclic permutation, 
in the marked flat $\f_\e$, we have
in simple roots coordinates
$$\ovra{\vstEk\vstEi}=(\gZplus_\tau,\gZmoins_\tau)
\mbox{ and }
\ovra{\vstEpk\vstEpi}=(\gZmoins_\taup,\gZplus_\taup)$$
in particular $\ovra{\vstEk\vstEi}$ and $\ovra{\vstEpk\vstEpi}$ 
belong to $\Cb$.

\begin{prop}[Geometric interpretation of edge parameters]
  \label{prop- geometry of  adjacent triangles of type triangle}
  In any marked flat $\f_\e:\Aa\to \App_\ki$ 
  sending $\bord\Cc$ to $\Fi$,  we have
  $$\ovra{\vstEi\vstEpk} =(\gEp_\eb, \gEp_\e)$$
  in the basis of simple roots of $\Aa$.
\end{prop}

\begin{proof}
  We project in the transverse tree at infinity $\Es_{\p_\i}$
  in direction $\p_\i$ by
  $\projE_{\p_\i}:\Es\to\Es_{\p_\i}$.
  We denote by
  $\oT$ and $\oTp$  the respective projections of 
  $\p_\i\p_\j$ and $\p_\i\p_{\kl}$
  (seen as points of $\bordinf\Es_{\p_\i}$) 
  on the line from $\D_\i$ to $\p_\i\p_k$ in $\Es_{\p_\i}$. 
  Then we have $\projE_{\p_\i}(\vstEi)=\oT$ and $\projE_{\p_\i}(\vstEpk)=\oTp$ 
  by 
  Lemma 17 of \cite{ParTriples}.
  Thus we have 
  (by \eqref{eq- projections and Busemann cocycle}
  and \eqref{eq- geombir and centers of tripods})
  $$\begin{array}{ccl}
    \rac2( \ovra{\vstEi \vstEpk} )
    &=&\Bus{\D_\i}{ \projE_{\p_\i}(\vstEi) }{ \projE_{\p_\i}(\vstEpk) }\\
    &=&\Bus{\D_\i}{\oT}{\oTp}\\
    &=&\geombir(\D_\i,\p_\i\p_\j,\p_\i\p_\k,\p_\i\p_\kl)\\
    &=&\gEp_\e
 \;. \end{array}$$

  Similarly, projecting in the transverse tree $\Es_{\D_\i}$ 
  and denoting by $\oTet$, $\oTetp$
  the respective projections of 
  $\D_\i\cap\D_\j$ and $\D_\i\cap\D_{\kl}$
  (seen as points of $\bordinf\Es_{\D_\i}$) 
  on the line from $\p_\i$ to $\D_\i\cap\D_k$ in the tree $\Es_{\D_\i}$,
  we  have
  $$\begin{array}{ccl}
    \rac1(\ovra{\vstEi \vstEpk})
    &=&\Bus{\p_\i}{\projE_{\D_\i}(\vstEi)}{\projE_{\D_\i}(\vstEpk)}\\
    &=&\Bus{\p_\i}{\oTet}{\oTetp}\\
    &=&\geombir(\p_\i,\D_\i\cap\D_\j,\D_\i\cap\D_\k,\D_\i\cap\D_\kl)\\
    &=&\geombir(\D_\k,\p_\k\p_\l,\p_\k\p_\i,\p_\k\p_\ij)
    =\gEp_\eb
\;. \end{array}
  $$
\end{proof}

In particular, we have the following geometric interpretation in the
building $\Es$ of the
hypothesis ``\leftshiftingedge{} $\e$''.

\begin{coro}
  \label{coro-  geom interpretation of leftshiftingedge}
  The three following assertions are equivalent
  \begin{enumerate}
  \item  $\gFGp$ is \leftshiftingedge{} $\e$ ;
  \item In any marked flat $\Aa\to \App_\ki$ 
    sending $\bord\Cc$ to $\Fi$,
    the vectors $\ovra{\vstEi\vstEpi}$ and $\ovra{\vstEk\vstEpk}$
    are in $\Cc$; 

  \item
    $\vstEi$ is in the open Weyl chamber from $\vstEpi$ to
    $\Fdev_\i$,\\
    and
    $\vstEpk$ is in the open Weyl chamber from $\vstEk$ to $\Fdev_\i$.
  \end{enumerate}
\end{coro}

The following lemma establishes that, if 
$\gFGp$ is \leftshiftingedge{} $\e$
then the associated adjacent singular triangles $\Delta_\tau$,
$\Delta_\taup$ 
lie in a common flat. 

\begin{lemm}
  \label{lemm- two adjacent sing triangles lie in a flat}%
  %
  Let $\tau=(\i,\j,\k)$ and $\taup=(\k,\l,\i)$
  be a pair of adjacent triangles in $\Tt$
  (where $(\i,\j,\k,\l)$ are positively ordered), and denote by $\e$ the
  common edge $(\k,\i)$. 
  %
  %
  Suppose that $\gFGp$ is \leftshiftingedge{} $\e$.
  %
  %
  Let $\Ch$ be any Weyl chamber with tip $\vstEi$ containing $\Delta_\tau$, 
  and 
  $\Chpr$ be any Weyl chamber with tip $\vstEpk$ containing
  $\Delta_\taup$. 
  There exists a marked flat $\f:\Aa\to \Es$ 
  such that $\f(0)=\vstEi$, 
  $\f(\aA)=\vstEpk$ with $\aA=(\gEp_\eb,\gEp_\e)$ in simple roots coordinates, 
  $\f(-\Cb)=\Ch$ and $\f(\aA+\Cb)=\Chpr$.
  %
  %
  %
\end{lemm}

\begin{proof}
  Let $\ch$, $\chpr$ be the boundaries at infinity of $\Ch$ and $\Chpr$.
  Let $\f_\e:\Aa\to \App_\ki$ be the marked flat sending $\bord\Cc$ to
  $\Fi$ and $0$ to $\vstEi$. 
  Let $\f_\e^{-1}(\vstEk)=\vstAk$, $\f_\e^{-1}(\vstEpi)=\vstApi$ 
  and $\f_\e^{-1}(\vstEpk)=\vstApk$.
  By 
  Proposition \ref{prop- geometry of  adjacent triangles of type triangle}
  and Theorem \ref{theo- triangle sing type triangle - local},
  we have that
  $\vstAk\in-\Cb$, $\vstApi\in\Cc$, $\vstApk=(\gEp_\eb,\gEp_\e)$ is in $\vstAk+\Cc$, 
  and $[0,\vstApk]$ is contained in  $(\vstAk+\Cb) \cap (\vstApi-\Cb)$.
  Since $\F_\i$ and $\Ch$ are  opposite at $\vstEi$
  (Theorem \ref{theo- triangle sing type triangle - local}),
  there exists a marked flat $\f_1:\Aa\to \App_1$ 
  sending $-\Cb$ to $\Ch$ and $\bord \Cc$ to $\F_\i$.
  Since  $\f_1$ and $\f_\e$ both sends $0$ to $\vstEi$ and $\bord \Cc$ to
  $\F_\i$, we have $\f_1=\f_\e$ on the convex subset
  $f_e^{-1}(f_1(\Aa))$, 
  which contains $\Cb$.
  Since $\vstEk=\f_\e(\vstAk)$ belongs to $\Delta_\tau$, 
  hence to $\App_1=f_1(\Aa)$, it implies that $\f_1=\f_\e$ on
  $\vstAk+\Cb$.
  Since $\vstApk \in \vstAk+\Cc$, $\f_1$ and $\f_\e$ coincide on a germ of 
  the Weyl chamber $\vstApk-\Cc$ at $\vstApk$.
  Then  the Weyl chambers 
$\Chpp=\f_1(\vstApk-\Cc)$ and $\f_\e(\vstApk-\Cc)$ 
  define the same chamber
  $\TangS_{\vstEpk}\ch=\TangS_{\vstEpk}\F_\k$ 
  in the space of directions at $\vstEpk$. 
  Therefore $\Chpp$ and $\Chpr$ are opposite at $\vstEpk$
  by Theorem \ref{theo- triangle sing type triangle - local}.
  Hence there exists a marked flat $\f:\Aa\to \App$ 
  sending $\vstApk+\Cb$ to $\Chpr$ and $-\bord \Cc$ to $\ch$.
  Since $\vstEpi=\f_\e(\vstApi)=\f_1(\vstApi)$ belongs to
  $\App$ and $\f$ and $\f_1$ are very strongly asymptotic on $-\Cc$,
  we have $\f=\f_1$ on $\vstApi-\Cb$. 
  In particular, $\f(-\Cc)=\f_1(-\Cc)=\Ch$. 
  Moreover $\f=\f_\e$ on $(\vstAk+\Cb) \cap (\vstApi-\Cb)$, 
  which contains  $[0,\vstApk]$.
\end{proof}

\label{s- Psi loc C isom}
From now on, we suppose that  $\gFGp$ is \leftshifting{}.
The next lemma formalizes the construction of the map $\Psi$, 
and is a straigthforward consequence of 
Theorem \ref{theo- triangle sing type triangle - local}
and  Proposition \ref{prop- geometry of  adjacent triangles of type
  triangle}.
We refer to Section \ref{ss- construction of the A2-complex from gZ}
for the definition of charts of $\CEstgZ$, and
we recall that $\vstCEi(\tau)$ is the $\i$-vertex of the
singular triangle associated with $\tau$
in the associated $\typeA_2$-complex $\CEstgZ$.

\begin{lemm}
  There exists a unique $\rho$-equivariant  map
  $\Psi : \CEstgZ \to \Es$ such that
  \begin{itemize}
  \item The map $\Psi$ sends $\vstCEi(\tau)$ to $\vstEi(\tau)$ 
    for all marked triangle $\tau=(\i,\j,\k)$ of $\Tt$ ;
    %
  \item For every chart $\chartCEst_\m : \cellAgZ^\m \to \cellCEstgZ^\m$ of $\CEstgZ$,
    the map $\Psi\circ \chartCEst_\m:\cellAgZ^\m \to \Es$ is the restriction
    of a marked flat. \qed
  \end{itemize}
\end{lemm}

We now check that $\Psi$ is a local $\Cc$-isometry.
Let $\xCE$ be a point in $\CEstgZ$. Then either
there is a neighbourhood of $\xCE$ contained in some
$\cellCEstgZ^\tau\cup\cellCEstgZ^\e$ with $\e$ adjacent to $\tau$, on which 
$\Psi$ is a $\Cc$-isometry 
by Lemma \ref{lemm- two adjacent sing triangles lie in a flat},
or  $\xCE$ is the vertex $\vstCEtau$ 
of a singular triangle $\cellCEstgZ^\tau$ reduced to a point 
(i.e. with invariant $\gZ_\tau=0$). 
\smallskip

\noindent {\advance\linewidth by-5.5cm
\begin{minipage}{\linewidth}
    \hskip1em 
    In that case, 
    denote  by $(\i,\j,\k)$ the vertices of $\tau$ 
    and by $\es$ the oriented edge of $\tau$ 
    with terminal vertex $\s$, for $\s=\i,\j,\k$.
    A neighbourhood of $\xCE$ in $\CEstgZ$ is then given by 
    the union of the three segments $\cellCEstgZ^\es$, $\s=\i,\j,\k$.  
    The image by  $\Psi$ of $\cellCEstgZ^\es$ is then 
    a non trivial segment $[\vstEtau,\vsteE_\s]$, 
    contained in the Weyl chamber $\Ch_\s$ with vertex $\vstEtau$
    and boundary $\Fs$.
    The chambers $\Ch_\s$  are pairwise opposite at $\vstEtau$
    by Theorem \ref{theo- triangle sing type triangle - local}.
    Hence $\Psi$ is a local $\Cc$-isometry on the union of the three
    segments $\cellCEstgZ^\es$.
    \qed
  \end{minipage}
  \hfill
  \begin{minipage}{4cm}
    \centering
    \includegraphics[scale=0.5]{figures/%
      Es-shrinked-triangle-and-adjacent-edges.fig}
  \end{minipage}
}

\subsection{Proof of Theorem \ref{theo- global Cgeod surface}}
\label{ss- proof global Cisom}
Let $\xCE$, $\xCEp$ be two points of $\CEstgZ$. 
Let $\geod$ be the unique  geodesic 
from $\xCE$ to $\xCEp$ in $\CEstgZ$.
We are going to prove that the image
$\eta=\Psi\circ\geod$ of $\geod$ by $\Psi$ is a $\Cc$-geodesic path in
$\Es$, 
using the criterion 
in Proposition \ref{prop- critere Cgeod avec sing}.
Then we will have $\Cd(\Psi(\xCE),\Psi(\xCEp))=\ellC(\eta)$,
which is equal to $\ellC(\geod)$
since $\Psi$ preserves the $\Cc$-length of paths,
hence equal to $\Cd(\xCE,\xCEp)$
by definition of the $\Cc$-distance  $\Cd$ in $\CEstgZ$, which
concludes.

Let $\tRR_0=0 <\tRR_1< \cdots < \tRR_\IndPath=1$ be a minimal subdivision of
$[0,1]$ such that $\restrict{\geod}{[\tRR_\indPath,\tRR_{\indPath+1}]}$ 
has constant type of direction in $\bord\Cb$, 
and let $\xCE_\indPath=\geod(\tRR_\indPath)$ and $\yE_\indPath=\Psi(\xCE_\indPath)$.
For $0<\indPath<\IndPath$,
since $\lengthangle_{\xCE_{\indPath}}({\xCE_{\indPath-1}},{\xCE_{\indPath+1}})>\pi$, 
the point $\xCE_{\indPath}$ is a singularity of $\CEstgZ$, 
hence by construction of $\CEstgZ$ it is a boundary point of the form
$\xCE_{\indPath}=\bCE_{\e}$ for some  oriented edge $\e$ of $\Tt$.
Suppose that  for some $0<\indPath<\IndPath$ 
the (constant) type of the segment $[\xCE_\indPath, \xCE_{\indPath+1}]$ 
is singular. 
We have to prove that the directions 
$\TangS_{\yE_{\indPath}}\yE_{\indPath+2}$ and
$\TangS_{\yE_{\indPath}}\yE_{\indPath-1}$ are $\Cc$-opposite, 
i.e. contained in
two opposite closed chambers at $\yE_\indPath$.

We are first going to show that 
the \edgeseparating{} hypothesis (\HypEdgeSep) 
allows us to reduce our study to the case of two adjacent triangles.

\begin{lemm}
 There exist two adjacent triangles $\tau=(\i,\j,\k)$ and
$\taup=(\k,\l,\i)$ in $\Tt$ such that
the segment $[\xCE_{\indPath},\xCE_{\indPath+1}]$ is contained in
$\cellCEstgZ^\tau\cup \cellCEstgZ^\ki\cup \cellCEstgZ^\taup$,
and,  
up to exchanging $\xCE_{\indPath}$ and $\xCE_{\indPath+1}$,
denoting $\vstCEs=\vstCEs(\tau)$ and $\vstCEps=\vstCEs(\taup)$,
we are in one of the following cases
\begin{enumerate}
\item 
\label{i- case - contained in one triangle}
$\bEdgeCEki=\vstCEi$, and
$\xCE_{\indPath}=\bEdgeCEij$ and
$\xCE_{\indPath+1}=\bEdgeCEki$ ; 
\item 
\label{i- case - contained in triangle U edge strip}
$\xCE_{\indPath}=\bEdgeCEij$ and
$\xCE_{\indPath+1}=\bEdgeCEik$ ;
\item 
\label{i- case - contained in shrinked edge strip}
$\bEdgeCEki=\vstCEi$ and $\bEdgeCEik=\vstCEpi$, 
and
 $\xCE_{\indPath}=\bEdgeCEki$ and $\xCE_{\indPath+1}=\bEdgeCEik$ ;

\item 
\label{i- case - contained in triangle U shrinked edge strip U adj triangle}
$\bEdgeCEki=\vstCEi$ and $\bEdgeCEik=\vstCEpi$, 
and $\xCE_{\indPath}=\bEdgeCEij$ and
$\xCE_{\indPath+1}=\bEdgeCEkl$.

\end{enumerate}
\end{lemm}

\begin{figure}[h]
\hfill
\begin{minipage}[b]{.4\linewidth}
\centering  
\includegraphics[scale=0.5]{figures/CEs-geod-sing-segment-tri.fig}\\
(\ref{i- case - contained in one triangle})
\end{minipage}
\hfill
\begin{minipage}[b]{.4\linewidth}
\centering  
  \includegraphics[scale=0.5]{figures/CEs-geod-sing-segment-triUedge.fig}\\
\vskip-4ex
(\ref{i- case - contained in triangle U edge strip})
\end{minipage}
\hfill\hfill

\hfill
\begin{minipage}[b]{.4\linewidth}
\centering  
  \includegraphics[scale=0.5]{figures/CEs-geod-sing-segment-incl-shrinked-edge.fig}\\
(\ref{i- case - contained in shrinked edge strip})
\end{minipage}
\hfill
\begin{minipage}[b]{.4\linewidth}
\centering  
  \includegraphics[scale=0.5]{figures/CEs-geod-sing-segment-triUshrinked-edgeUtri.fig}\\
\vskip-2ex
(\ref{i- case - contained in triangle U shrinked edge strip U adj triangle})
\end{minipage}
\hfill\hfill

  \caption{Singular segment in a geodesic in $\CEstgZ$: the four cases.}
  \label{fig- the four cases for a singular segment in a geod in CEstgZ}
\end{figure}

\begin{proof}
  Since $\sphericalangle_{\xCE_{\indPath}}({\xCE_{\indPath-1}}, {\xCE_{\indPath+1}})>\pi$,
the singular direction $\TangS_{\xCE_{\indPath}}\xCE_{\indPath+1}$ must be 
in the  boundary of $\CEstgZ$ at $\xCE_{\indPath}$. 
Similarly
the singular direction $\TangS_{\xCE_{\indPath+1}}\xCE_{\indPath}$ must be 
in the  boundary of $\CEstgZ$ at $\xCE_{\indPath+1}$.

Denote by $(\i,\j)$ the oriented edge such that $\xCE_\indPath=\bEdgeCEij$, 
and let $\tau=(\i,\j,\k)$ be the (marked) left adjacent triangle in $\Tt$.%

Suppose first that the segment from $\xCE_{\indPath}=\bEdgeCEij$ to 
$\xCE_{\indPath+1}$ starts in direction of the point $\vstCEi$.
Then it contains $[\bEdgeCEij,\vstCEi]$.
If  $\xCE_{\indPath+1}=\vstCEi$, then $\vstCEi=\bEdgeCEki$, and we are done.
If $\xCE_{\indPath+1}\neq \vstCEi$,
the segment $[\bEdgeCEij,\vstCEi]$ extends by $[\vstCEi,\bEdgeCEik]$ in a constant
type segment, and $[\xCE_{\indPath},\xCE_{\indPath+1}]$ contains the segment 
$[\bEdgeCEij,\bEdgeCEik]$.
If $\xCE_{\indPath+1}\neq \bEdgeCEik$, then we now show that,
by hypothesis (\HypEdgeSep),
we must have $\bEdgeCEik=\vstCEpk$ and $\xCE_{\indPath+1}=\bEdgeCEkl$:
The constant type geodesic ray $\ray$ in $\cellCEstgZ^\taup$ from $\bEdgeCEik$
parallel to the side $[\vstCEpk,\vstCEpl]$   hits the boundary of
$\cellCEstgZ^\taup$ at the point $\bCE$ of $[\vstCEpi,\vstCEpl]$ 
at distance $d(\bEdgeCEik,\vstCEpk)= \larg_\ik=\max(-\gEp_\ki,-\gEp_\ik)$ 
from  $\vstCEpl$, 
and cannot be extended outside $\cellCEstgZ^\taup$ 
since $\bCE$ is on $[\vstCEpl,\bEdgeCEli[$ 
since $\larg_\ik+\larg_\li <\abs{\gZ_\taup}$
by (\HypEdgeSep) (see remark \ref{rema- hyp (Edgeseparating)}).
There is only one singular point that may then be on $\ray$, which is
$\bEdgeCEkl$ in the case where $\bEdgeCEik=\vstCEpk$.

Suppose now that 
$\TangS_{\xCE_{\indPath}}\xCE_{\indPath+1} = \TangS_\bEdgeCEij \vstCEppi$, where
$\vstCEppi=\vstCEi(\taupp)$, 
with $\taupp=(\j, \i, \h)$ the triangle adjacent to $\tau$ 
along edge $(\i, \j)$.
If have $\xCE_{\indPath+1}=\vstCEppi=\bEdgeCEji$,
 then  $\xCE_\indPath=\bEdgeCEij=\vstCEj$, and we are in case (\ref{i- case -
   contained in shrinked edge strip}), 
up to replacing the pair of adjacent triangles
$\tau,\taupp$ by the pair $\taup,\tau$.

If $\xCE_{\indPath+1}\neq\vstCEppi$, then $\xCE_{\indPath+1}$ must be the next
singular point on the same side of the adjacent triangle cell
$\cellCEstgZ^\taupp$ 
(since the ray from $\vstCEppi=\vstCEi(\taupp)$ to $\bEdgeCEih$ to no extend
outside $\cellCEstgZ^\taupp$). 
We are then reduced to  the
previous case $\xCE_{\indPath}=\bEdgeCEij$, $\xCE_{\indPath+1}=\bEdgeCEik$ 
by exchanging the roles of $\xCE_\indPath$ and $\xCE_{\indPath+1}$.

If $\TangS_{\xCE_{\indPath}}\xCE_{\indPath+1}$ is neither $\TangS_\bEdgeCEij \vstCEi$
nor $\TangS_\bEdgeCEij \vstCEppi$, then there is a third boundary direction in
$\CEstgZ$ at $\bEdgeCEij$, which means that $\bEdgeCEij=\vstCEj$ and
$\TangS_{\xCE_{\indPath}}\xCE_{\indPath+1}=\TangS_{\vstCEj} \vstCEk$.
Then as $[\vstCEj, \vstCEk]$ is not extendable, we must have
$\xCE_{\indPath+1}=\bEdgeCEjk$.
We are then reduced to  the
previous case $\xCE_{\indPath+1}=\bEdgeCEij$, $\xCE_{\indPath+1}=\bEdgeCEki$ 
by exchanging the roles of $\xCE_\indPath$ and $\xCE_{\indPath+1}$.
\end{proof}
Since $\Psi$ is a local $\Cc$-isometry 
(Section \ref{s- Psi loc C isom}), 
the path $\eta$ is a local $\Cc$-geodesic in $\Es$, 
and its restriction to $[\tRR_\indPath,\tRR_{\indPath+1}]$ is
is the affine  segment $[\yE_\indPath,\yE_{\indPath+1}]$ (since it is of
constant type of direction in $\Cb$).

\noindent{\bfseries 
Case (\ref{i- case - contained in one triangle}): 
$\xCE_{\indPath}=\bEdgeCEij$ and $\xCE_{\indPath+1}=\bEdgeCEki$.}
Then $\bEdgeCEki=\vstCEi$.
We then have $\yE_{\indPath+1}=\Psi(\vstCEi)=\vstEi$,
and  $\TangS_{\yE_{\indPath+1}}\yE_{\indPath+2}$
is in $\TangS_\bEdgeEki \Psi(\cellCEstgZ^\ki)$, 
hence in the closed chamber $\TangS_\bEdgeEki \F_\i$.
Since $\bEdgeEki=\vstEi$ is in the closed Weyl chamber from $\bEdgeEij$ to $\F_\i$ 
by Theorem \ref{theo- triangle sing type triangle - local},
we have that 
$\TangS_{\yE_{\indPath+1}}\yE_{\indPath+2}$ is in $\TangS_\bEdgeEij\F_\i$.
Since $\bEdgeEij$ is in the flat $\App_\ij$,
it proves that
$\TangS_{\yE_{\indPath}}\yE_{\indPath+2}$ is $\Cc$-opposite 
to $\TangS_{\yE_{\indPath}}\yE_{\indPath-1}$.

\noindent{\bfseries 
Case (\ref{i- case - contained in triangle U edge strip}):
$\xCE_{\indPath}=\bEdgeCEij$ and $\xCE_{\indPath+1}=\bEdgeCEik$.}
At $\xCE_{\indPath+1}=\bEdgeCEik$, 
the direction $\TangS_{\xCE_{\indPath+1}}\xCE_{\indPath+2}$ 
is in the unique closed chamber of $\TangS_{\bEdgeCEik}\CEst$
containing $\TangS_{\bEdgeCEik}\vstCEpi$, 
where $\taup=(\k,\l,\i)$ is the adjacent triangle in $\Tt$.
In the building $\Es$, we then have 
 $\yE_{\indPath}=\bEdgeEij$, 
$\yE_{\indPath+1}=\bEdgeEik$, 
We now prove that  we then have 
 $\TangS_{\yE_\indPath}\yE_{\indPath+2}\in\TangS_\bEdgeEij\F_\i$.
Let $\Ch$ be a closed Weyl chamber 
with tip $\yE_\indPath=\bEdgeEij$ 
containing a germ at $\yE_{\indPath+1}=\bEdgeEik$ 
of the segment $[\yE_{\indPath+1},\yE_{\indPath+2}]$.
Then $\Ch$ contains
a germ at $\bEdgeEik$ of the segment $[\bEdgeEik,\vstEpi]$. 
Let $\Ch_\i=\Ch(\bEdgeEij,\F_\i)$ 
be  the closed Weyl chamber from $\bEdgeEij$ to
$\F_\i$.
The closed Weyl chambers $\Ch(\bEdgeEik,\F_\i)$,
 and $\Ch(\bEdgeEik,\F_\k)$ are opposite at $\bEdgeEik$ 
(because $\bEdgeEik \in \Aik$),
and respectively contain  $\vstEpi$ and $\bEdgeEij$, 
therefore $\Ch_i$ contains the segment $[\bEdgeEik,\vstEpi]$.
Since $[\bEdgeEij,\bEdgeEik]$ and $[\bEdgeEik,\vstEpi]$ 
are singular segments of different type of direction in $\bord\Cb$, 
we have then
$\TangS_\bEdgeEij\Ch=\TangS_\bEdgeEij\Ch_i$, 
hence $\TangS_{\yE_\indPath}\yE_{\indPath+2}\in\TangS_\bEdgeEij\F_\i$.
Since $\TangS_\bEdgeEij \F_\j$ and $\TangS_\bEdgeEij \F_\i$ 
are opposite closed chambers of $\TangS_\bEdgeEij \Es$ 
(because $\bEdgeEij \in \Aij$), 
it proves that
$\TangS_{\yE_{\indPath}}\yE_{\indPath+2}$ is $\Cc$-opposite to $\TangS_{\yE_{\indPath}}\yE_{\indPath-1}$.

\noindent{\bfseries 
Case (\ref{i- case - contained in shrinked edge strip}):
$\xCE_{\indPath}=\bEdgeCEki=\vstCEi$ and $\xCE_{\indPath+1}=\bEdgeCEik=\vstCEpk$.}
In the building $\Es$ we have 
$\yE_\indPath=\vstEi$,
$\yE_{\indPath+1}=\vstEpk$,
$\TangS_{\yE_{\indPath}}\yE_{\indPath-1}\in \Delta_\tau$,
and $\TangS_{\yE_{\indPath+1}}\yE_{\indPath+2}\in \Delta_\taup$.
Lemma \ref{lemm- two adjacent sing triangles lie in a flat}
then implies that
there then exists two opposite Weyl chamber with tip $\vstEi$ containing
respectively $\Delta_\tau$ and $[\vstEi,\vstEpk]\cup \Delta_\taup$, 
so 
$\TangS_{\yE_{\indPath}}\yE_{\indPath+2}$ is $\Cc$-opposite to
$\TangS_{\yE_{\indPath}}\yE_{\indPath-1}$ at $\yE_\indPath=\vstEi$.
\noindent{\bfseries 
Case 
(\ref{i- case - contained in triangle U shrinked edge strip U adj triangle}): 
$\xCE_{\indPath}=\bEdgeCEij$ and $\xCE_{\indPath+1}=\bEdgeCEkl$.}
Then we have $\bEdgeCEik=\vstCEpk$, i.e. $\gEp_\e= 0$.
In the building $\Es$, we then have 
$\yE_{\indPath}=\bEdgeEij$, 
$\yE_{\indPath+1}=\bEdgeEkl$, 
and
in the spherical building $\TangS_{\bEdgeEij}\Es$ of directions at
$\yE_{\indPath}=\bEdgeEij$,
we have that
$\TangS_{\yE_{\indPath}}\yE_{\indPath+1}$ 
belongs to the chamber $\TangS_\bEdgeEij \F_\j$, 
and
$\TangS_{\yE_{\indPath+1}}\yE_{\indPath+2}$ 
belongs to the chamber $\TangS_\bEdgeEkl \F_\l$.
Since $\gEprime_\e\leq 0$,
the following Lemma \ref{lemm- gEprime leq 0 implies that Fj Fl opp en x_i}
implies  
that the ideal chambers $\F_\j$ and $\F_\l$ are then opposite at $\vstEi$, 
so, since $\vstEi$ is on the singular segment $]\bEdgeEij,\bEdgeEkl[$, 
it implies that
$\TangS_{\yE_{\indPath}}\yE_{\indPath+2}$ is $\Cc$-opposite to
$\TangS_{\yE_{\indPath}}\yE_{\indPath-1}$
as needed.

\begin{lemm}
\label{lemm- gEprime leq 0 implies that Fj Fl opp en x_i}
Let $\tau=(\i,\j,\k)$ and $\taup=(\k,\l,\i)$, 
be a pair of adjacent triangles in $\Tt$
(where $(\i,\j,\k,\l)$ are positively ordered), 
and denote by $\e$ the common edge $(\k,\i)$. 
Denote $\vstEi=\vstEi(\tau)$, $\vstEk=\vstEk(\tau)$, 
$\vstEpi=\vstEi(\taup)$ and $\vstEpk=\vstEk(\taup)$.
Suppose that 
$\gFGp$ is  \leftshiftingedge{} $\e$
and $\gEp_\e=0$. 
Then
\begin{enumerate}
\item 
\label{i- singular geod D_j x_i _p_l}
There is a geodesic from $\D_\j$ to $\p_\l$ through $\vstEj$, $\vstEi$, $\vstEpk$, and $\vstEpl$ ;

\item $\F_\j$ and $\F_\l$ are opposite 
at $\vstEi$
if and only if
\begin{equation*}
\gEprime_\e
=\geombir(\D_\i, \p_\i \p_\kl, \p_\i \p_\j, \p_\i \p_\k)
\leq 0
\end{equation*}

\end{enumerate}
\end{lemm}

\begin{figure}[h]
  \includegraphics[scale=0.5]{figures/%
Es-adjacent-triples-shrinked-edge-type1.fig}
  \caption{Adjacent triples with $\gEp_\e=0$, $\gEp_\eb>0$, $\gEprime_\e<0$.}
  \label{fig- adjacent triple with singular edge}
\end{figure}

\begin{proof}%
[Proof of Lemma \ref{lemm- gEprime leq 0 implies that Fj Fl opp en x_i}]
Since $\gFGp$ is  \leftshiftingedge{} $\e$ and $\gEp_\e=0$, we must have
$\gZ_\tau\leq 0$, and   $\gEp_\eb>0$
Lemma \ref{lemm- two adjacent sing triangles lie in a flat} then
implies that 
the path $(\vstEj, \vstEi, \vstEpk, \vstEpl)$ is a geodesic segment of 
singular type $\stypep$.
It extends in a geodesic $\geod$ from $\D_\j$ to $\p_\l$, 
because $\Delta_\tau$ is opposite to $\F_\j$ at $\vstEj$ 
and  $\Delta_\taup$ is opposite to $\F_\l$ at $\vstEpl$ 
(Theorem \ref{theo- triangle sing type triangle - local}),
and (\ref{i- singular geod D_j x_i _p_l}) is proven.

 By point (\ref{i- singular geod D_j x_i _p_l}) 
and Proposition \ref{prop- geometry of adjacent triangles of type triangle},
the directions $\D_\j$ and $\p_\l$ are opposite at $\xE=\vstEi$, 
and we have $\TangS_\xE\p_\l=\TangS_\xE\vstEpk=\TangS_\xE\p_\i$. 
Thus
$\F_\j$ and $\F_\l$ are opposite at $\xE$.
if and only if
$\p_\j$ and $\D_\l$ are opposite at $\xE$, 
i.e. $\TangS_\xE\p_\j \notin \TangS_\xE\D_\l$.

We now prove that
$\TangS_\xE\p_\j \in  \TangS_\xE\D_\l$ if and only if
$\TangS_\xE (\p_\i \p_\j)=\TangS_\xE (\p_\i \p_\kl)$.
First observe that
 $\TangS_\xE \p_\i$ is different from $\TangS_\xE \p_\j$,
 as $\xE$ is
on the flat $\App(\p_\i,\p_\j,\p_\k)$. 
We also have $\TangS_\xE \p_\i\neq \TangS_\xE \p_\kl$, since 
$\TangS_\xE \p_\kl \in \TangS_\xE \D_\k$ and 
$\TangS_\xE \p_\i \notin \TangS_\xE \D_\k$ 
since $\xE$ is in the flat  $\App(\F_\k,\F_\i)$.
Then
 $\TangS_\xE\p_\j \in  \TangS_\xE\D_\l$
if and only if
$\TangS_\xE \p_\j \oplus \TangS_\xE \p_\l = \TangS_\xE \D_\l$
(since $\TangS_\xE \p_\j \neq \TangS_\xE \p_\l$).
We have
$\TangS_\xE \p_\j \oplus \TangS_\xE \p_\l
=\TangS_\xE \p_\j \oplus \TangS_\xE \p_\i
=\TangS_\xE (\p_\i \p_\j)$
(since $\TangS_\xE \p_\j \neq \TangS_\xE \p_\i$).
On the other hand, since $\TangS_\xE \p_\i\neq \TangS_\xE \p_\kl$,
we have $\TangS_\xE \D_\l
=\TangS_\xE \p_\i \oplus \TangS_\xE \p_\kl
=\TangS_\xE (\p_\i \p_\kl)$, and we are done.

Projecting in the transverse tree at infinity $\Es_{\p_\i}$
we now show 
that $\TangS_\xE (\p_\i \p_\j)\neq\TangS_\xE (\p_\i \p_\kl)$
is equivalent to
$\geombir(\D_\i, \p_\i \p_\kl, \p_\i \p_\j, \p_\i\p_\k)
=\gEprime_\e\leq 0$.
\begin{figure}[h]
  \includegraphics[scale=0.7]{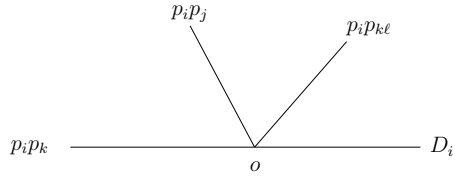} 
  \caption{In the transverse tree at infinity $\Es_{\p_\i}$.}
\end{figure}
Indeed,
since the projection of $\xE$ is
 the center $\oT$ of the ideal tripod
$\D_\i$, $\p_\i \p_\j$, $\p_\i \p_\k$ 
(by Lemma 17 of \cite{ParTriples}),
the directions $\TangS_\xE (\p_\i \p_\j)$ and  $\TangS_\xE (\p_\i \p_\kl)$ 
are distinct 
if and only if 
the two geodesic rays in $\Es_{\p_\i}$ 
from $\oT$ to the ideal points $\p_\i \p_\j$ and
$\p_\i\p_\kl$ have distinct germs at $\oT$. 
This is equivalent to 
$\geombir(\D_\i, \p_\i \p_\kl, \p_\i \p_\j, \p_\i\p_\k)
=\gEprime_\e\leq 0$,
which concludes.
\end{proof}
We have proven that in all cases 
 $\TangS_{\yE_{\indPath}}\yE_{\indPath+2}$ is $\Cc$-opposite to $\TangS_{\yE_{\indPath}}\yE_{\indPath-1}$.
Therefore the piecewise affine path
$(\yE_0, \yE_1, \ldots, \yE_\IndPath)$ is a global $\Cc$-geodesic in $\Es$ by
Proposition \ref{prop- critere Cgeod avec sing}, which concludes the
proof of Theorem \ref{theo- global Cgeod surface}.

\section{Degenerations of representations
and convex 
\texorpdfstring{$\RP^2$}{RP\^2}-structures}
\label{s- degenerations of representations}
In this section, 
we use Theorem \ref{theo- Cgeod surface} to describe 
a large family of degenerations of convex $\RP^2$-structures on $\Sf$,   
corresponding to 
a part of the boundary of the moduli space of 
convex $\RP^2$-structures on $\Sf$ 
constructed in \cite{ParComp}.
Let $\KK$ be any valued field. 
Starting from \S \ref{ss- thm on degenerations of rep}, the field $\KK$ will be
supposed to be either equal to $\RR$ or $\CC$ or ultrametric.

\subsection{Background on asymptotic cones}
\label{ss- asympt cones}
In this section, we gather definitions and tools about the various
notions of ultralimits and asymptotic cones that will be used in 
what follows.
We first fix notations  about usual ultralimits of metric spaces 
(see for example \cite{KlLe97},
or \cite[\S  2.3]{ParComp} for more details).
Then we briefly recall various notions of asymptotic cones of
algebraic objects  introduced in  \cite[\S 3]{ParComp}:
asymptotic cones of valued fields, 
normed vector spaces, 
linear group, 
ultralimits of representations, and their links. 
Finally we introduce the notion of 
{\em asymptotic cones of projective spaces} 
and establish some basic properties of  
asymptotic cones in projective geometry.

Fix a (non principal) ultrafilter $\om$ on $\NN$,
and  a {\em scaling sequence} $(\lan)_{\n\in\NN}$, 
that is a sequence  of real numbers  such that $\lan\geq 1$ and 
$\lan \to \infty$.

A point $\xom$ in a Hausdorff topological space $\E$ 
is the {\em $\om$-limit} of a sequence $(\xn)_\n$ in $\E$
if it is its limit  with respect to the filter $\om$. 
We will then  denote $\limom \xn=\xom$. 
Note that  $\limom \xn$ is then a cluster value of the sequence $(\xn)_\n$. 
Recall that any sequence contained in a compact (Hausdorff) space has a
(unique) $\om$-limit. 
The $\om$-limits of sequences of real numbers are taken
in the compact space $[\minfty,\pinfty]$.

Given a sequence of pointed metric spaces
$(\Esn,\dEn,\on)_{\n\in\NN}$,
a sequence $(\xn)_\n$ in $\prod_\n \Esn$ is called 
{\em \ombounded{}} when $\limom \dEn(\on,\xn) <\infty$.

The {\em ultralimit} of 
$(\Esn,\dEn,\on)_{\n\in\NN}$
is the quotient $\Esom$ of 
the subspace of \ombounded{} sequences in $\prod_\n \Esn$
by the pseudo-distance $\dEom$ given by
$$\dEom((\xn),(\yn))=\limom \dEom(\xn,\yn) \;.$$
It is a complete metric space $(\Esom,\dEom)$.
The class in $\Esom$ of a \ombounded{} sequence
$(\xn)$ will be called its {\em ultralimit} and 
be denoted by $\omulim \xn$.
Given a sequence $\Yn \subset\Esn$, we denote by
$\omulim{\Yn}$ the subset of $\Esom$ 
consisting of ultralimits of  
\ombounded{} sequences $(\xn)_\n$ 
such that $\xn \in \Yn$ for all $\n$, 
and call it the {\em ultralimit} of the sequence $(\Yn)_\n$.

Let $(\KKom,\absom{\ })$
be the asymptotic cone of the valued field $\KK$ 
with respect to the scaling sequence $(\lan)$,
that is the ultralimit of the sequence of valued fields 
$(\KK, \abs{\ }^{\lslanh})$ 
(base points are at $0$, see \cite[\S 3.3]{ParComp}).
It is an ultrametric field. 
Note that its absolute value $\absom{\ }$ 
takes all values in $\RR_{\geq 0}$.
Given a sequence $(\aKn)$ in $\KK$, 
we denote $\omulim{\aKn}=\infty$ when
$\limom\abs{\aKn}^\lslanh=\infty$, 
so that every sequence in $\KK$ has a well defined 
ultralimit in $\KKom\cup\{\infty\}$.
Denote by $\normVo$ the canonical norm on $\V=\KK^\N$.
Let $(\Vom,\normVoom)$
be the asymptotic cone of the normed vector space $(\V,\normVo)$
with respect to the scaling sequence $(\lan)$,
i.e. the ultralimit of the sequence of normed vector spaces 
$(\V, \normVo^{\lslanh})$ (see \cite[\S 3.4]{ParComp}).
It is a normed vector space over the valued field $\KKom$, 
canonically isomorphic to $\KKom^\N$, 
with canonical basis  the ultralimit $\eVoomb=(\eVoom_\indN)$ 
of the canonical basis $\eVob=(\eVo_\indN)$ of $\KK^\N$.

Denote by $\NormEndVo$ the norm on $\End(\V)$ associated with $\normVo$.
The ultralimit $(\End(\V)_\om,\NormEndVoom)$ 
of the sequence of normed algebra $(\End(\V), \NormEndVo^{\lslanh})$ 
 is a normed algebra over the valued field $\KKom$, 
(see \cite[\S 3.5]{ParComp}).

We now describe the asymptotic cone of the linear group $\GL(\V)$.
Let $\GL(\V)_\om$ be the  subgroup of invertible elements of
$\End(\V)_\om$. 
Note that the ultralimit of a  sequence $(\un)_\n$ in $\GL(\V)$ which
is \ombounded{} in $\End(\V)$ 
(that is $\limom\NormEndVo(\un)^\lslanh <\infty$)
may be not invertible, 
so the definition  of $\GL(\V)_\om$ in \cite{ParComp}  
is incorrect (definition 3.16) 
(with no incidence on the remaining of the paper). 
The following proposition 
describes the invertible elements in $\End(\V)_\om$.

\begin{prop}
  \label{prop- GL(V)_om}
  Let $\uom=\omulim \un$ be an element of $\End(\V)_\om$.
  Then 
  $\uom$ is invertible in $\End(\V)_\om$
  if and only if
  $\un$ is in $\GL(\V)$ for $\om$-almost all $\n$
  and 
  $(\un^{-1})$ is \ombounded{} in $\End(\V)$, 
  i.e. $\limom\NormEndVo(\un^{-1})^\lslanh <\infty$.
  Then $\uom^{-1}=\omulim \un^{-1}$. 
  \qed  
\end{prop}

\begin{proof}
  If  $(\un^{-1})$ is \ombounded{} in $\End(\V)$, 
  then clearly 
  $$(\omulim \un)\circ(\omulim(\un^{-1}))=\omulim (\un\circ\un^{-1})=1$$
  hence $\uom$ is invertible with inverse $\omulim(\un^{-1})$.

  Conversely, suppose that $\uom$ is invertible in $\End(\V)_\om$, and let
  $\uom'=\omulim \un'$ be its inverse.
  Then $1-\uom\circ \uom'=0$ in $\End(\V)_\om$, hence
  $\limom\NormEndVo(\id-\un\circ \un')^\lslanh =0$.
  Then for $\om$-almost all $\n$ we have 
  $\NormEndVo(\id- \un\circ \un')<1$, 
  so $\an=\un\circ \un'$ is invertible in $\GL(\V)$ 
  with 
  $\NormEndVo(\an^{-1})\leq ( 1-\NormEndVo(\id-\an) )^{-1}$.
  Then $\un$ is invertible 
  with inverse $\un^{-1}=\un'\circ\an^{-1}$.
  Since $\NormEndVo(\un^{-1}) 
  \leq \NormEndVo(\un')(1-\NormEndVo(\an-\id))^{-1}$,
  we have
  $\limom\NormEndVo(\un^{-1})^\lslan < \infty$, 
  that is $(\un^{-1})$ is \ombounded.
  Then $\uom'=\omulim \un^{-1}$ by uniqueness of inverses.
\end{proof}

A sequence $(\un)_{\n\in\NN}$ in $\GL(\V)$ will be called 
{\em \ombounded{} (in $\GL(\V)$)} if
 $$\limom\NormEndVo(\un)^\lslanh <\infty 
\text{ and } \limom\NormEndVo(\un^{-1})^\lslanh <\infty \;.$$

A \ombounded{} sequence $(\un)_{\n\in\NN}$ in $\End(\V)$ induces
an endomorphism $\uomup$ of $\Vom$ defined by 
$$ \uomup(\omulim\vVn)=\omulim\un(\vVn)$$
for all \ombounded{} sequence $(\vVn)$ in $\V$.
This endomorphism depends only on the ultralimit $\uom$ in
$\End(\V)_\om$ of the sequence $(\un)$. 
The following results allows us to identify $\End(\V)_\om$ with 
$\End(\Vom)$, and $\GL(\V)_\om$ with $\GL(\Vom)$.
\begin{prop}\cite[Corollaire 3.18]{ParComp}
  \label{prop- End(V)_om isom End(Vom)}
  The map 
  $$
  \begin{array}{rcl}
\End(\V)_\om& \to &\End(\Vom)\\
\uom & \mapsto & \uomup  
  \end{array}
  $$ 
  is an isomorphism of $\KKom$-normed algebras identifying 
  $\GL(\V)_\om$ with $\GL(\Vom)$.
\end{prop}

\subsection{Ultralimits of projective spaces}
One verifies easily that 
the ultralimit of any sequence of vector subspaces of $\V$ (of fixed dimension) 
is a vector subspace of $\Vom$ (of same dimension). 
Then any sequence of points in the projective space $\PP\V$ 
has a well defined {\em ultralimit}  in the projective space $\PP\Vom$.
This induces a canonical identification of $\PP\Vom$ 
with the ultralimit of the metric
spaces $(\PP\V, \dPV^\lslanh)$, where $\dPV$ 
is the 
distance on $\PP\V$ induced by the norm $\normVo$.
Let $\po_\indN\in\PP V$, $\indN=0,\ldots,\N$ 
be the {\em canonical projective frame} of $\PP\V$,
which is defined by 
$\po_\indN=\classPV{\eVo_\indN}$ for $\indN=1,\ldots,\N$ and 
$\po_0=\classPV{\eVo_1+\cdots+\eVo_\N}$.
Let $\poom_\indN$ be the ultralimit of the constant sequence $(\po_\indN)_{\n\in\NN}$.
Then $(\poom_\indN)_{\indN=0,\ldots,\N}$ 
is the canonical projective frame of $\PP\Vom$.

\label{s- def gn ombounded in PGL(V)}
A sequence  $(\gn)_\n$ in  $\PGL(V)$  is {\em \ombounded} 
if it has a \ombounded{} lift $(\un)_\n$ in $\GL(\V)$.
Then the ultralimit $\gom\in\PGL(\Vom)$ of $(\gn)$ 
is well defined by 
$$ \gom(\omulim\p_\n)=\omulim\gn(\p_\n)$$
and it coincides with the class in $\PGL(\Vom)$ 
of the ultralimit $\uom\in\GL(\Vom)$ of the sequence $\un$.

Here is a useful criterion to see 
if a sequence  $(\gn)_\n$ in  $\PGL(V)$  is \ombounded, in terms of
the action on the projective space.

\begin{prop}
  \label{prop- gn ombounded iff projective frame}
  Let $(\gn)_{\n\in\NN}$ be a sequence in  $\PGL(\V)$. 
  let $(\qn_\indN)_{0 \leq \indN\leq \N}$ be the image by $\gn$ of
  the canonical projective frame $(\po_\indN)_{0 \leq \indN\leq \N}$.
  Denote by 
  $\qom_\indN$ the ultralimit of the sequence $(\qn_\indN)_{\n\in\NN}$ in $\PP\Vom$.
  %
  %
  The following assertions are equivalent:
  \begin{enumerate}
  \item 
    The points $\qom_\indN$ form a projective frame of $\PP\Vom$ ;
  \item
    The sequence $(\gn)_\n$ is \ombounded{} in $\PGL(\V)$, 
  \end{enumerate}
  Then the ultralimit  of $(\gn)_{\n\in\NN}$ in $\PGL(\Vom)$ 
  is the unique map $\gom\in\PGL(\Vom)$  sending $(\poom_\indN)_{0 \leq \indN\leq \N}$ to $(\qom_\indN)_{0 \leq \indN\leq \N}$.
  \qed
\end{prop}

\begin{proof}
  Suppose that the points $\qom_\indN$ form a projective frame of
  $\PP\Vom$, 
  and let $\gom$ be the projective map in $\PGL(\Vom)$
  sending the frame $(\pom_\indN)$ to the frame $(\qom_\indN)_\indN$. 
  Let $\uom$ be a lift of $\gom$ in $\GL(\Vom)$. 
  There exists a \ombounded{} sequence $(\un)_{\n\in\NN}$ in $\GL(\V)$ with
  ultralimit $\uom$ (Propositions \ref{prop- End(V)_om isom End(Vom)}
  and  \ref{prop- GL(V)_om}).

  For each fixed $\indN$, 
  let $\vV^\n_\indN$ be the image of $\eVo_\indN$ by $\un$. 
  Then $(\vV^\n_\indN)_\n$ is a \ombounded{} sequence in $\V$ 
  and its ultralimit is $\vV^\om_\indN=\uom(\eVo_\indN)$, which is 
  a non zero vector in $\Vom$ representing 
  the point $\qom_\indN$ of $\PP\Vom$.

  Let $\wV^\n_\indN$ be a vector in $\qn_\indN$ (seen as a line of $\V$) 
  at minimum distance from  $\vV^\n_\indN$. 
  Then $\normVo{(\wV^\n_\indN-\vV^\n_\indN)}^\lslanh
  \leq \dPV(\un(\pn_\indN),\qn_\indN)^\lslanh$.
  Since the sequence  $(\vV^\n_\indN)_\n$ is \ombounded{}, 
  and we have $\limom\dPV(\un(\pn_\indN), \qn_\indN)^\lslanh=0$, 
  it follows that the sequence $(\wV^\n_\indN)_\n$ is \ombounded{} 
  with ultralimit $\wV^\om_\indN=\vV^\om_\indN$.

  Let $\hn\in\GL(\V)$ be the linear map sending
  the canonical basis $(\eVo_\indN)_{\indN=1\ldots\N}$ 
  to the basis $(\wV^\n_\indN)_{\indN=1\ldots\N}$, 
  which is a  lift in $\GL(\V)$ of $\gn$. 
  The sequence $(\hn)_\n$ is \ombounded{} in $\End(\V)$, 
  and
  its ultralimit  in $\End(\Vom)$ is $\uom$ 
  (since it sends $\eVoom_\indN$ to $\vV^\om_\indN$).
  Since $\uom=\omulim\hn$ is invertible in $\End(\Vom)$,
  the sequence $(\hn)^{-1}$ is \ombounded{}, as wanted.
\end{proof}

We recall that the ultralimit   $\rhoom:\Ga\to \PGL(\Vom)$  
of a sequence of representations $\rhon:\Ga \to \PGL(\V)$  
is well defined
when $\rhon$ is {\em \ombounded}, 
that is when for all $\ga\in\Ga$ (or just for a generating set), 
the sequence $(\rhon(\ga))_{\n\in\NN}$ is \ombounded{} in $\PGL(\V)$ 
(see \S\ref{s- def gn ombounded in PGL(V)}). It is then
defined by $\rhoom(\ga)=\omulim \rhon(\ga)$.

The cross ratio is easily seen to behave well under ultralimit.

\begin{prop}
  \label{prop- omulimit of cross ratios}
  For $\n\in\NN$, let $\pn_1$, $\pn_2$, $\pn_3$, $\pn_4$ 
  be four points in $\PP(\V)$ in a common line $\Dn$. 
  Let $\pom_\indQuadP=\omulim \pn_\indQuadP$ be the ultralimit in $\PP(\Vom)$ 
  of the sequence $(\pn_\indQuadP)_{\n\in\NN}$.
%
  Suppose that the quadruple 
  $(\pom_1,\pom_2,\pom_3,\pom_4)$ is \nondegeneratedQ{}, i.e. 
  has no triple point.
  Then
  $(\pn_1,\pn_2,\pn_3,\pn_4)$ is  \nondegeneratedQ{}
  for $\om$-almost all $\n$, and
  \[\Bir(\pom_1,\pom_2,\pom_3,\pom_4)=\omulim \Bir(\pn_1,\pn_2,\pn_3,\pn_4)\]
  in $\KKom\cup\{\infty\}$.
  \qed
\end{prop}

\subsection{Asymptotic cones and Fock-Goncharov parameters}
\label{ss- FG-param and asympt cones}
In this section, we show that
FG-parametrization of representations  
behaves well with respect to ultralimits, that is
the two constructions commute. 
We use the hypotheses and notations 
of Section  \ref{s- FG parametrization}, 
from which 
we recall that,
for $\FGp=((\Z_\tau)_\tau,(\Ep_\e)_\e)$ 
in  $(\KKol)^\TrianglesT \times (\KKo)^\orEdgesT$
the  $\T$-\transverse{} map with FG-parameter $\FGp$
is denoted by $\FdevZ:\Farey{\Sf} \to \MaxFlags(\PP)$ 
and $\rhoZ$ denotes the associated representation from $\Ga$ to
$\PGL(\KK^3)$. 
We use the  notations and hypotheses of the previous sections 
for ultralimits and asymptotic cones.

\begin{prop}
  \label{prop- FG coords goes to asymptotic cone}
  Let $(\FGpn)_\n$  be a sequence in
 $(\KKol)^\TrianglesT \times (\KKo)^\orEdgesT$
  and let 
  $\Zn=(\Zn_\tau)_\tau$ and $\Epn=(\Epn_\e)_\e$.
  %
   Denote by  $\Fdevom:\Farey{\Sf}\to \MaxFlags(\KKom^3)$ the ultralimit of
   the sequence of maps $\FdevZn:\Farey{\Sf} \to \MaxFlags(\PP)$.
  For each triangle $\tau$ and oriented edge $\e$ of $\T$,
  denote by  $\Zom_\tau=\omulim \Zn_\tau$ $\Epom_\e=\omulim \Epn_\e$
  the ultralimits in $\KKom \cup\{\infty\}$ of the sequence
  $(\Zn_\tau)_\n$ and $(\Epn_\e)_\n$.
  %
  Suppose that 
  $\Zom_\tau\notin \{\infty, -1,0\}$ for all triangle $\tau$ of $\T$, 
  and 
  $\Epom_\e\notin \{\infty,0\}$ for all oriented edge $\e$ of $\T$.
  %
  Then
  \begin{enumerate}
  \item $\Fdevom=\FdevZom$ ;
    %

  \item 
    The ultralimit    $\rhoom:\Ga\to \PGL_3(\KKom)$ 
    of the sequence of representations $\rhoZn$
    is well defined and 
    $\rhoom=\rhoZom$.
  \end{enumerate}

\end{prop}

\begin{proof}
  %
  %
  Denote $\Fdevn=\FdevZn$.
  Note that, for each $\i\in\Farey{\Sf}$ 
  the ultralimit of the sequence of
  flags $\Fdevn(\i)=\Fn_\i=(\pn_\i,\Dn_\i)$
  is a well-defined flag $\Fom_\i=(\pom_\i,\Dom_\i)$  in $\MaxFlags(\KKom^3)$.
  The ultralimit $\Fdevom:\Farey{\Sf}\to \MaxFlags(\KKom^3)$ of the maps $\Fdevn$ is thus always well defined.
  We first prove that $\Fdevom=\FdevZom$.
  %
  %
  Since the canonical basis
  of $\KKom^3$ is the ultralimit of the canonical basis of $\KK^3$,
  it is clear that the image $(\Fom_1,\Fom_2,\Fom_3)$ 
  of the base triangle $\tauo$ by $\Fdevom$ 
  remains  in canonical form, i.e.
  %
  $\pom_1=[1:0:0]$, $\pom_2=[0:1:0]$,  
  $\Dom_1\cap \Dom_2=[0:0:1]$, $\pom_3=[1:1:1]$ 
  is the canonical projective frame.
  So it is enough to prove
  the two next lemmas, ensuring that  
  $\Fdevom$ is $\Tt$-\transverse{} and of FG-invariant $\Zom$
  by induction on adjacent triangles, 
  following the construction of the map $\FdevZom$ 
  in Section \ref{s- construction of F}.

  \begin{lemm}
    Let   $\tau$ be a marked triangle in $\Tt$
    with ordered vertices $(\i,\j,\k)$ in $\Farey{\Sf}$.
    Suppose that $\Fom_\i, \Fom_\j$ and $\pom_\k$ are in generic position.
    Then the triple of flags $(\Fom_\i, \Fom_\j,\Fom_\k)$ is generic
    and its triple ratio is $\Zom_\tau$.
  \end{lemm}
  %
  %
  %
  %
  %

  %

  \begin{proof}
    Denote by 
    $\pn_\ij$ the point $\Dn_\i\cap\Dn_\j$
    and  by $\pom_\ij=\omulim \pn_\ij$ its ultralimit.
    Denote by $\Dn_\ki$ the line $\pn_\k\pn_\i$,
    by $\Dn_\ki$ the line $\pn_\k\pn_\i$, 
    and  by $\Dom_\ki, \Dom_\kj$ their ultralimits.
    Since $\Fom_\i, \Fom_\j$ and $\pom_\k$ are in generic position,
    the points 
    $\pom_\i$, $\pom_\j$, $\pom_\k$ and $\pom_\ij=\Dom_\i\cap\Dom_\j$ 
    are pairwise distinct and 
    $\pom_\k\pom_\i
    =\Dom_\ki$, 
    $\pom_\k\pom_\i
    =\Dom_\kj$ 
    and 
    $\pom_\k\pom_\ij=\omulim \pn_\k\pn_\ij$
    are three distinct lines.
    We have
    $$\Tri(\Fn_\i,\Fn_\j,\Fn_\k)=
    \Bir(\Dn_\kj, \pn_\k\pn_\ij, \Dn_\ki, \Dn_\k)
    =\Zn_\tau$$
    and by hypothesis $\Zom_\tau=\omulim\Zn_\tau$ is distinct from
    $\infty, 0,-1$.
    %
    Hence by Proposition \ref{prop- omulimit of cross ratios} 
    taking ultralimits, the line $\Dom_\k$ is distinct from the 
    three lines $\Dom_\ki$, $\Dom_\kj$
    and $\pom_\k\pom_\ij$, 
    so the triple of flags $(\Fom_\i, \Fom_\j,\Fom_\k)$ is generic,
    and 
    $\Tri(\Fom_\i, \Fom_\j,\Fom_\k)
    =\Bir(\Dom_\kj ,  \pom_\k\pom_\ij, \Dom_\ki, \Dom_\k)
    =\Zom$.
  \end{proof}

  \begin{lemm}
    Let   $\tau$ be a marked triangle in $\Tt$ with ordered vertices $(\i,\j,\k)$,
    and $\taup=(\k,\l,\j)$ be the adjacent triangle.
    Suppose that the triple of flags 
    $\Fom(\tau)$ is generic. 
    Then $\Fom(\taup)$ is generic and
    $$\Bir(\Dom_\i, \pom_\i \pom_\j, \pom_\i \pom_\k, \pom_\i
    (\Dom_\k\cap\Dom_\l ))
    =\Epom_\e$$
    $$\Bir(\Dom_\k, \pom_\k \pom_\l, \pom_\k \pom_\i, \pom_\k
    (\Dom_\i\cap\Dom_\j) )
    =\Epom_\eb \;.$$
    \qed
  \end{lemm}
  \begin{proof}
    %
    %
    %
    %
    Denote $\pom =\ulim (\Dn_\k\cap\Dn_\l)$. 
    Then $\pom\in \Dom_\k$ and $\pom\in\Dom_\l$.
    Since $\Fom_\k$, $\Fom_\i$ and $\Fom_\j$ are in generic position,
    we have
    $\Dom_\k\cap\Dom_\i =\ulim (\Dn_\k\cap\Dn_\i) $,
    %
    $\Dom_\k \cap (\pom_\i\oplus \pom_\j)= \ulim 
    (\Dn_\k \cap (\pn_\i\oplus \pn_\j)) $
    and these two points are distinct and distinct from $\pom_\k$.
    It follows then from Proposition \ref{prop- omulimit of cross ratios}
    that the cross ratio
    $\Bir(\Dom_\k \cap \Dom_\i, \Dom_\k \cap (\pom_\i\oplus \pom_\j),  \pom_\k, \pom)$
    is the ultralimit of 
    $\Bir(\Dn_\k \cap \Dn_\i, \Dn_\k \cap (\pn_\i\oplus \pn_\j), \pn_\k, \pn_\kl)
    =\Epn_\e$, which is $\Epom_\e$.  
    Since $\Epom_\e\neq 0,\infty$, it follows that the point $\pom$ (which is
    on the line $\Dom_\k$) is
    distinct from the two points $\pom_\k$, $\Dom_\i \cap \Dom_\k$.

    Similarly
    the three lines $\pom_\k\pom_\i$, $\pom_\k\pom_\j$ and 
    $\Dom_\k$
    are paiwise distinct, hence the ultralimit $\Deltom$ 
    of the line $\pn_\k\pn_\l$ satisfies
    $\Bir(\pom_\k \pom_\i, \pom_\k\pom_\j, \Dom_\k, \Deltom)=\Epom_\eb$. 
    The line $\Deltom$ passes through $\pom_\k$ and
    is distinct from the lines   
    $\Dom_\k$ and $\pom_\k\pom_\i$, 
    since $\Epom_\eb\neq 0,\infty$. 
    In particular $\pom_\i\notin \Deltom$, so $\pom_\l\neq \pom_\i$.


    We have  three pairwise distinct lines 
    $\Dom_\i$,$\pom_\i\pom_\k$ and $\pom_\i\pom$, 
    hence the cross ratio  
    $\Bir(\Dom_\i,\pom_\i\pom_\k,\pom_\i\pom, \pom_\i\pom_\l)$
    is the ultralimit of 
    $\Bir(\Dn_\i,\pn_\i\pn_\k,\pn_\i\pn_\kl, \pn_\i\pn_\l)=\Zn_\taup$,
    which is $\Zom_\taup$.
    Since $\Zom_\taup\neq \infty$, we have $\pom_\l \notin \Dom_\i$.
    Since $\Zom_\taup\neq -1$, we have 
    $\pom_\i\pom_\l\neq \pom_\i\pom_\k$, so $\pom_\l \notin \Dom_\k$, in
    particular $\pom_\l\neq \pom$, 
    and $\pom_\l\notin \pom_\i\pom_\k$.
    So $\Fom_\i$,$\Fom_\k$ and $\pom_\l$ are in generic position.
    Since $\pom_\l,\pom\in\Dom_\l$ and  $\pom_\l\neq \pom$, 
    we have $\Dom_\l=\pom_\l\pom$.
    Since $\Zom_\taup\neq 0$, we have 
    $\pom_\i\pom_\l\neq \pom_\i\pom$,
    so the line $\Dom_\l=\pom_\l\pom$ do not contain $\pom_\i$.
    We also have $\pom_\k\notin\Dom_\l$ (since $\pom\neq \pom_\k$) and
    $\Dom_\l$ do not pass through $\Dom_\i\cap\Dom_\k$ (since 
    $\pom\neq \Dom_\i\cap\Dom_\k$).  
    The triple of flags
    $(\Fom_\k,\Fom_\l,\Fom_\i)$
    is then generic and of triple ratio 
    $\Bir(\Dom_\i,\pom_\i\pom_\k,\pom_\i(\Dom_\k\cap\Dom_\l),
    \pom_\i\pom_\l)=\Zom_\taup$ as $\pom=\Dom_\l\cap\Dom_\k$.
  \end{proof}

  We may now conclude the proof of Proposition \ref{prop- FG coords goes to asymptotic cone}.
  Let $\ga\in\Ga$.
  Then $\rhon(\ga)$ sends the canonical projective frame 
  $\Fo_1,\Fo_2,\po_3$
  to the frame  
  $\Fn(\ga\o_1)$, $\Fn(\ga\o_2)$, $\pn(\ga\o_3)$
  whose  ultralimit
  is $\Fom(\ga\o_1)$, $\Fom(\ga\o_2)$, $\pom(\ga\o_3)$,
  which is generic hence a projective frame in $\PP(\KKom^3)$.
  Then by Proposition \ref{prop- gn ombounded iff projective frame}
  $\rhon(\ga)$ is \ombounded{}
  and its ultralimit  $\rhoom(\ga)=\omulim \rhon(\ga)$ 
  sends the canonical frame 
  $\Foom_1,\Foom_2,\poom_3$ to 
  $\Fom(\ga\o_1)$, $\Fom(\ga\o_2)$, $\pom(\ga\o_3)$
  %
  hence the flags $\Fom(\tauo)$
  to $\Fom(\ga\tauo)$ (since they have the same triple ratio).
  So $\rhoom(\ga)=\rhoZom(\ga)$ as wanted.
\end{proof}

\subsection{Main result}
\label{ss- thm on degenerations of rep}
We suppose now that  $\KK$ is either equal to $\RR$ or $\CC$ or
ultrametric.
We are now able to describe
a large family of degenerations 
of representations of $\Ga$ in $\PGL(\KK^3)$
as (length spectra of) $\typeA_2$-complexes  of the form $\CEsgZ_{\gFGp}$.
using degenerations of FG-parameters.

We denote by $\Es$ the $\CAT(0)$ metric space (symmetric
space or Euclidean building)   associated with $\PGL_3(\KK)$.

\begin{theo}
\label{theo- degeneration of reps}
Let $(\Zn,\Epn)_{\n\in\NN}$ be a sequence 
in $\KKol^\TrianglesT\times\KKo^\orEdgesT$.
Let $\rhon: \Ga\to \PGL_3(\KK)$ 
be the representation of FG-parameter
$\FGpn=((\Zn_\tau)_\tau,(\Epn_\e)_\e)$.

Let $\gZn_\tau=\log\abs{\Zn_\tau}$,
$\gZn_\e=\log\abs{\Epn_\e}$ 
and $\gZn=(\gZn_\tau)_\tau$,
$\gEpn=(\gEpn_\e)_\e$.
Consider a sequence of real numbers $\lan\geq 1$ 
going to $\pinfty$, such that 
the sequence $\lslan\gFGpn$ converges to a nonzero $\gFGp$ 
in $\RR^\TrianglesT\times\RR^\orEdgesT$.
Suppose that: 
\begin{mydescription}
\item[(\HypFlatTriangles')] 
For each triangle $\tau$ of $\T$,
$\liminf\lslan \log\abs{\Zn_\tau+1} \geq 0$ ;

\item[(\HypH')] For each oriented edge $\e$  in  $\T$, 
$\liminf\lslan \log\abs{\Epn_\e+1} \geq 0$ ;

\item[(\HypLeftShift)]
$\gFGp$ is \leftshifting{},
 
\item[(\HypEdgeSep)] $\gFGp$ is \edgeseparating{}. 
\end{mydescription}

Let $\CEsgZ$ be the $\typeA_2$-complex of FG-parameter $\gFGp$.
Then the renormalized $\Cc$-length spectrum of $\rhon$ 
converges to the $\Cc$-length spectrum of  $\CEsgZ$:
 for all $\ga\in\Ga$ we have
$$\lslan \ellC(\rhon(\ga)) \to \ellC(\ga,{\CEsgZ})$$ 
in $\Cb$.

In particular, 
the usual Euclidean  length  spectrum 
of $\rhon$ converges
to the Euclidean norm 
of the $\Cc$-length spectrum of $\CEsgZ$:
$$\lim_{\n\to \infty}\lslan \elleuc(\rhon(\ga)) 
=\normeuc{\ellC(\ga,{\CEsgZ})}$$
and the analogous claim  holds for the  Hilbert length:
 $$\lim_{\n\to \infty}\lslan \ellH(\rhon(\ga)) 
 =\normH(\ellC(\ga,\CEsgZ))$$
for all $\ga\in\Ga$.
\end{theo}

\begin{rema*}
 The hypotheses
 (\HypFlatTriangles') and (\HypH') are automatic for $\KK=\RR$
  (and more generally for $\KK$ ordered) and  positive
  FG-parameters
(since for positive $\aK\in\KK$ we then have
$\abs{\aK+1}\geq \abs{1}=1$).

\end{rema*}

\begin{proof}[Proof of Theorem \ref{theo- degeneration of reps}]
 The idea of the proof is 
first to pass to the ultralimit in 
an appropriate asymptotic cone associated with
 the scaling sequence $(\lan)_\n$,
 and then to apply Theorem \ref{theo- Cgeod surface} to show that the
 ultralimit representation preserves a $\Cc$-geodesic copy of the 
 $\typeA_2$-complex $\CEsgZ$ in the associated Euclidean building, 
hence has same marked $\Cc$-length spectrum, and to use the continuity
properties of $\Cc$-length spectrum with respect to asymptotic cones of
\cite{ParComp} to conclude.

Let $(\KKom,\absom{\ })$ be 
the asymptotic cone of the valued field $\KK$ 
with respect to the scaling sequence $(\lan)_\n$ 
(see Section \ref{ss- asympt cones}).
We first check that the ultralimits behave well.
For all $\tau\in \TrianglesT$, by hypothesis 
$\limom\lslan \log\abs{\Zn_\tau} =\gZ_\tau < \pinfty$, hence
the ultralimit $\Zom_\tau$ of the sequence $\Zn_\tau$ in $\KKom$ 
is well defined, and  
$\limom\lslan \log\abs{\Zn_\tau} =\gZ_\tau > \minfty$
hence $\Zom_\tau\neq 0$. 
Similarly, the ultralimit $\Epom_\e$ of the sequence $\Epn_\e$ in $\KKom$ 
is well defined and non zero as 
$\absom{\Epom_\e}= \exp{\gEp_\e}$.
For all triangle $\tau$, we also have 
$\absom{\Zom_\tau+1}=\limom \abs{\Zn_\tau+1}^\lslanh \geq 1$,
in particular $\Zom_\tau\neq -1$.
Then by Proposition \ref{prop- FG coords goes to asymptotic cone}
the ultralimit    $\rhoom:\Ga\to \PGL(\KKom^3)$ 
of the sequence of representations $\rhon$ is well defined
 and  is the representation $\rhoZom$ associated with the FG-parameter
 $\FGpom=((\Zom_\tau)_\tau, (\Epom_\e)_\e)$.
The FG-parameter $\FGpom$ clearly satisfies the hypotheses of 
Theorem \ref{theo- Cgeod surface}.
Hence Theorem \ref{theo- Cgeod surface} applies, and $\rhoom$,
acting on the Euclidean building $\Esom$ associated with $\PGL_3(\KKom)$,
preserves an equivariant $\Cc$-geodesically embedded copy of the 
$\typeA_2$-complex $\CEsgZ$, hence the length spectra coincide:
$$\ellC(\rhoom(\ga))=\ellC(\ga,{\CEsgZ})
\mbox{\ \ for all }\ga \mbox{ in }\Ga\;.$$
Now fix $\ga$ in $\Ga$. 
Since $\Esom$ is the asymptotic cone of the metric space $\Es$ for the
rescaling sequence $\lan$ (see Theorem 3.21 of \cite{ParComp})
and by continuity properties of the $\Cc$-length 
with respect to asymptotic cones 
(by Theorem 3.21 and Proposition 4.4 of \cite{ParComp}), 
the sequence  $\lslan \ellC(\rhon(\ga))$
 has ultralimit $\ellC(\rhoom(\ga))$ in $\Cb$.

This proves that  the sequence $\lslan \ellC(\rhon(\ga))$ converges
(in the usual sense)  to
$\ellC(\ga,{\CEsgZ})$ in $\Cb$, 
since every subsequence of the sequence 
$\lslan\ellC(\rhon(\ga))$ 
has $\ellC(\ga,{\CEsgZ})$ as cluster value in $\Cb$.
\end{proof}

We now suppose that $\KK=\RR$ and we apply this result to describe a part of 
the compactification of the moduli space of representations
constructed in \cite{ParComp}.
 We first recall briefly the compactification.
Denote $\G=\PGL_3(\RR)$.
Let $\Xsep(\Ga,\G)=\Hom(\Ga,\G)//\G$ 
be the biggest Hausdorff quotient of $\Hom(\Ga,\G)$ under $\G$, 
which identifies with 
the locally compact subspace of $\Hom(\Ga,\G)/\G$ 
consisting of completely reducible (i.e. semisimple) representations 
 (see Section 5.1 of \cite{ParComp} for more details).
The space $\CGa$ of functions from $\Ga$ to $\Cb$ is
endowed with the product topology, and let $\PCGa$ denote 
the quotient space  of $\CGa-\{0\}$ by $\RR_{>0}$.
In  \cite{ParComp} 
we constructed a metrizable compactification
$\Xsept(\Ga,\G)$ of $\Xsep(\Ga,\G)$, 
with boundary contained in $\PCGa$ and 
endowed with a natural action of the modular group $\Out(\Ga)$, 
with following sequential characterization:
a sequence $[\rhon]_\n$ in $\Xsep(\Ga,\G)$ 
converges in $\Xsept(\Ga,\G)$  
to a  boundary point  $[w]$ in $\PCGa$ 
if and only if the two following conditions are satisfied
\begin{enumerate}
\item $[\rhon]_\n$ eventually gets out of any compact subset of
$\Xsep(\Ga,\G)$ ;
\item $[\ellC\circ \rhon]$ converges to $[w]$ in $\PCGa$.
\end{enumerate}
 (see Section 5.3 of \cite{ParComp} for more details).
Let $\MP=\TrianglesT\cup\orEdgesT$ and
denote by  $\PP^+\RR^\MP$ the space of rays in $\RR^\MP$, 
that is the quotient of $\RR^\MP-\{0\}$ by $\RR_{>0}$, 
which is the standard sphere of dimension
$8\abs{\chi(\Sf)}-1$, 
and let $\PP^+ :\RR^\MP-\{0\}\to \PP^+\RR^\MP$ be the corresponding projection.
The FG-parameters space $\RR^\MP$ is endowed with the 
standard compactification as a closed ball $\widetilde{\RR^\MP}$ with
boundary $\bordinf\RR^\MP= \PP^+\RR^\MP$.

Denote by $\ConeL\subset \RR^\MP$ the subset of \leftshifting{}
$\gFGp$, 
and  by $\ConeLS\subset \ConeL$ 
the subset of \leftshifting{} and \edgeseparating{} $\gFGp$,
which are  non empty open cones.

Since, for all \leftshifting{} $\gFGp$, the $\Cc$-length spectrum
$\ellCCEsgZ:\ga\mapsto \ellC(\ga,{\CEsgZ})$ 
of the $\typeA_2$-complex $\CEsgZ$ is not identically zero (that is differs
from $0\in\CGa$), 
Theorem \ref{theo- degeneration of reps} above implies the following result.

\begin{coro}
\label{coro- bord SPC}
  The FG-parametrization map 
\[ 
\function{\FGmap}{\RR^\MP}{\Xsept(\Ga,\G)}%
{\gFGp}{[\rho_{\exp(\gZ),\exp(\gEp)}]}
\]
extends continuously to the open subset $\PP^+\ConeLS$ of $\bordinf \RR^\MP$
by the restriction of the map 
\[
\begin{array}{cll}
  {\PP^+\ConeL} & \to & {\PCGa}\\
{[\gFGp]} & \mapsto & {[\ellCCEsgZ]}
\;.\end{array}
\]
\vskip -1.5em \qed
\end{coro}

Note that the image by $\FGmap$ of the open cone $\ConeLS$ 
is contained in the space $\SPC(\Sf)$
of convex projective structures on $\Sf$ 
with principal geodesic boundary (see \cite{GoldmanSPC,FoGoSPC}). 


\begin{thebibliography}{Otal92}
%

%

\bibitem[Al08]{Alessandrini08}
D.~Alessandrini, 
{\em Tropicalization of group representations},
Algebr. Geom. Topol. {\bfseries 8} (2008) 279–307.


%

\bibitem[BeHu14]{BeHu14}
Y.~Benoist, D.~Hulin,
{\em Cubic differentials and hyperbolic convex sets},
J. Differential Geom. {\bfseries 98} (2014) 1–19. 

\bibitem[Bes88]{Bestvina88}
M.~Bestvina, 
{\em Degenerations of the hyperbolic space}, 
Duke Math. J. {\bfseries 56} (1988), 143-161.

\bibitem[Bou96]{Bourdon96}
M. Bourdon,
{\em Sur le birapport au bord des CAT(-1)-espaces},
Pub. Math. I.H.E.S. {\bfseries 83} (1996), 95 – 104. 

\bibitem[BrHa]{BrHa}
M.R.~Bridson, A.~Haefliger, 
{\em Metric spaces with non-positive curvature}, 
Grund.~math.~Wiss. {\bfseries 319}, Springer Verlag, 1999.






\bibitem[Cap09]{Cap09CMH}
P.-E.~Caprace,
{\em Amenable groups and Hadamard spaces with a totally disconnected
  isometry group}, 
Comment. Math. Helv. {\bfseries 84} (2009), 437–455. 

\bibitem[CoLi14]{CoLi14}
B.~Collier, Q.~Li
{\em Asymptotics of certain families of Higgs bundles in the Hitchin
  component},
arXiv:1405.1106.

\bibitem[DuWo14]{DuWo14}
D.~Dumas, M.~Wolf,
{\em Polynomial cubic differentials and convex polygons in the
  projective plane},
arXiv:1407.8149.

\bibitem[FoGo06]{FoGoIHES}
V.~V.~Fock, A.~B.~Goncharov,
{\em Moduli spaces of local systems and higher Teichmüller theory},
Publ. Math. IHES. {\bfseries 103} (2006), pp 1–211.


\bibitem[FoGo07]{FoGoSPC}
V.~V.~Fock, A.~B.~Goncharov, 
{\em Moduli spaces of convex projective structures on surfaces}, 
Adv. Math. {\bfseries 208} (2007),  pp249–273.

\bibitem[Go90]{GoldmanSPC}
W.~M.~Goldman, 
{\em Convex real projective structures on compact  surfaces},
J. Diff. Geom. {\bfseries 31} (1990), 791–845.


%

\bibitem[KlLe97]{KlLe97}
B.~Kleiner, B.~Leeb,  {\em Rigidity of quasi-isometries for symmetric 
spaces of higher rank}, 
Pub. Math. IHES {\bfseries 86} (1997) 115-197.

\bibitem[KlLe06]{KlLe06}
B.~Kleiner, B.~Leeb
{\em Rigidity of invariant convex sets in symmetric spaces}, 
Invent. Math. {\bfseries 163} (2006), no. 3, 657–676.

%

\bibitem[Lab07]{Lab07}
F.~Labourie,
{\em Flat projective structures on surfaces and cubic holomorphic differentials},
Pure Appl. Math. Q. {\bfseries 3} (2007), 1057–1099. 

\bibitem[Le12]{Le12}
I.~Le, 
{\em Higher Laminations and Affine Buildings}, 
arXiv:1209.0812.
%


\bibitem[Leeb00]{Leeb00}
B.~Leeb, 
{\em A characterization of irreducible symmetric spaces and Euclidean
 buildings of higher rank by their asymptotic geometry},
Bonn. Math. Schr. 326,  Universität Bonn, Mathematisches Institut,
Bonn,  2000. 


\bibitem[Lof01]{Loftin01}
J.~Loftin,
{\em Affine spheres and convex $\mathbb{RP}^n$-manifolds},
Amer. J. Math. {\bfseries 123} (2001) 255–274.

\bibitem[Lof07]{Loftin07}
J.~Loftin,
{\em Flat metrics, cubic differentials and limits of projective holonomies},
Geom. Dedicata {\bfseries 128} (2007), 97–106. 


\bibitem[Mas06]{Masur}
H.~Masur,
{\em Ergodic theory of translation surfaces},
 Handbook of dynamical systems. Vol. 1B, 527–547, 
Elsevier, 2006. 

\bibitem[MSVM14]{MSVM14-I} 
B.~Mühlherr, K.~Struyve, H.~Van Maldeghem, 
{\em Descent of affine buildings - I. Large minimal
angles}, Trans. Amer. Math. Soc. {\bfseries 366} (2014),  4345–4366. 


\bibitem[Otal92]{Otal92}
J.-P.~Otal, 
{\em Sur la géometrie symplectique de l'espace des
  géodésiques d'une variété à courbure négative},
Rev. Mat. Iberoam. {\bfseries 8} (1992), 441–456.

\bibitem[Par99]{ParImm}
A. Parreau, 
{\em Immeubles affines, construction par les normes et étude des
isométries}, 
Crystallographic groups and their generalizations
(Kortrijk, 1999), 263--302, Contemp. Math., 262, Amer. Math. Soc.,
2000.

\bibitem[Par11]{ParComp}
A. Parreau,
{\em Compactification d'espaces de représentations de groupes de type
  fini}, 
Math. Z. {\bfseries 272} (2012), 51–86.

\bibitem[Par15a]{ParTriples}
A. Parreau,
{\em On triples of ideal chambers in  $\typeA_2$-buildings},
arXiv:1504.00285.
%

\bibitem[Par15b]{ParCd}
A. Parreau,
{\em La $\Cc$-distance dans les espaces symétriques et les
  immeubles affines} (work in progress).


\bibitem[Pau88]{Paulin88}
F.~Paulin, {\em Topologie de Gromov {\'e}quivariante, structures
hyperboliques et arbres r{\'e}els}, 
Invent. Math. {\bfseries 94} (1988),  53-80.

\bibitem[Pau97]{PauDg} 
F.~Paulin, {\em Dégénérescence de sous-groupes discrets des groupes de
Lie semi-simples}, C. R. Acad. Sci. Paris Sér. I Math. 324 (1997),
no. 11, 1217-1220.

\bibitem[Quint05]{Quint05}
 J.-F.~Quint {\em Groupes convexes cocompacts en rang supérieur},
Geom. Dedicata {\bfseries 113} (2005), 1–19. 

\bibitem[Rou09]{Rousseau09} 
G.~Rousseau, 
{\em Euclidean buildings},
In: 
{\em Géométries à courbure négative ou nulle, groupes discrets et
  rigidités}, 
Sémin. Congr. {\bfseries 18} (2009), 77–116. 

\bibitem[Tits86]{Tits86}
J.~Tits, {\em Immeubles de type affine},
dans ``Buildings and the geometry of diagrams'', Proc.~CIME Como 1984,
  L.~Rosati ed., Lect. Notes {\bfseries 1181}, Springer Verlag, 1986, 159-190.

\bibitem[Yoc10]{Yoccoz}
J.C.~Yoccoz,
{\em Interval exchange maps and translation surfaces},
 Homogeneous flows, moduli spaces and arithmetic, 1–69,
Clay Math. Proc., 10, Amer. Math. Soc., Providence, RI, 2010. 

\bibitem[Zha13]{Zhang13}
T.~Zhang, 
{\em The degeneration of convex $\mathbb{RP}^2$-structures on surfaces},
arXiv:1312.2452.

\bibitem[Zha14]{Zhang14}
T.~Zhang, 
{\em Degeneration of Hitchin representations along internal
  sequences},
arXiv:1409.2163

\end{thebibliography}
\end{document}